\documentclass[11pt,a4paper]{article}
\usepackage[utf8x]{inputenc}
\usepackage[T1]{fontenc}
\usepackage[english,german]{babel}

\usepackage{tikz}
\usepackage{amssymb}
\usepackage{amsmath}
\usepackage{hyperref}
\usepackage{mystyle}
\usepackage{enumerate}
\usepackage{StandardNewCommands}
\usepackage{multirow}
\usetikzlibrary{arrows}
\usepackage{mathtools}

\newcommand{\Fl}{\mathbb{F}_\ell}
\newcommand{\Fli}{\mathbb{F}_{\ell^i}}
\newcommand{\Flihalf}{\mathbb{F}_{\ell^{i/2}}}
\newcommand{\lihalf}{{\ell^{i/2}}}

\newcommand{\DZ}{\Delta(\ZZ)}
\newcommand{\Dv}{\Delta_v}
\newcommand{\D}{\Delta}
\newcommand{\B}{\mathfrak{P}}
\newcommand{\m}{\mathfrak{m}}
\newcommand{\lam}{\lambda}
\newcommand{\rcyclo}{\QQ(\zeta_r)}
\newcommand{\Qr}{{\QQ(\zeta_r)}}
\newcommand{\Zr}{{\ZZ[\zeta_r]}}
\newcommand{\Gcyclo}[1]{G_{\QQ(\zeta_{#1})}}
\newcommand{\GQr}{G_{\QQ(\zeta_{r})}}
\newcommand{\GQrl}{G_{\QQ(\zeta_{r\ell})}}
\newcommand{\zr}{\zeta_r}
\newcommand{\Glam}{G_{\lambda}}

\newcommand{\Slam}{S_{\lambda}}
\newcommand{\C}{\mathcal{C}}
\newcommand{\F}{\mathcal{F}}

\title{Superelliptic curves with large Galois images}
\author{Pip Goodman\footnote{Universitat de Barcelona; email: pip.goodman@ub.edu; ORCID: 0000-0001-6735-2367.} }
\date{}

\begin{document}

\maketitle
\selectlanguage{english}
\begin{abstract}
    Let $r>2$ and $\ell$ be primes. In this paper we study the mod $\ell$ Galois representations attached to curves of the form $y^r = f(x)$ where $f$ is monic and has coefficients belonging to the $r$th cyclotomic field.
    
    We provide conditions on the coefficients (and degree) of $f$ which allow one to verify the mod $\ell$ image is large outside of a (typically small) finite explicit set of primes. We allow all values of $r$ for which the $r$th cyclotomic field has odd class number. This appears to be the first explicit result for abelian varieties of dimension greater than two and not of $\GL_2$-type which allows the ground field to have unramified extensions.
    
    To determine the exact image we study the ``endomorphism character'', a certain algebraic Hecke character which generalises the CM character.
    This is achieved in entirety when $r=3$. To the author's knowledge, this is the first accurate description of the full image in the literature.
    
    Finally, we give several examples with genus ranging from 10 to 36. Applications to the Inverse Galois Problem are also included.
\end{abstract}
\selectlanguage{german}
\begin{abstract}
    Seien $r>2$ und $\ell$ Primzahlen. In dieser Arbeit untersuchen wir die mod $\ell$ Galois-Darstellungen von Kurven der Form $y^r = f(x)$, wobei $f$ normiert ist und Koeffizienten im $r$-ten Kreisteilungskörper hat.
    
    Wir stellen Bedingungen an die Koeffizienten (und den Grad) von $f$, die es erlauben, zu überprüfen, ob das mod $\ell$-Bild außerhalb einer (typischerweise kleinen) endlichen expliziten Menge von Primzahlen groß ist. Wir erlauben alle Werte von $r$, für die der $r$-te Kreisteilungskörper eine ungerade Klassenzahl hat.
    Soweit uns bekannt, ist dies das erste explizite Ergebnis für abelsche Varietäten von Dimension größer 2, die nicht von ${\rm GL}_2$-Typ sind, welches unverzweigte Erweiterungen des Grundkörpers zulässt
    
    Um das genaue Bild zu bestimmen, untersuchen wir den ``Endomorphismus-Charakter'', einen bestimmten algebraischen Hecke-Charakter, der den CM-Charakter verallgemeinert.
    Dies wird vollständig erreicht, wenn $r=3$. Nach dem Wissen des Autors ist dies in der Literatur die erste genaue Beschreibung des vollständigen Bildes.
    
    Schließlich geben wir mehrere Beispiele mit  Geschlecht 10 bis 36. Anwendungen auf das Inverse Galois-Problem sind ebenfalls enthalten.
\end{abstract}
\selectlanguage{english}
\tableofcontents
\newpage
\begin{section}{Introduction}
\label{section_introduction}
Let $r$ be a prime number and $\zr$ a primitive $r$-th root of unity. In this paper we study  mod $\ell$ Galois representations attached to superelliptic curves 
\[C\colon y^r = f(x) \in \Qr[x].\]
We produce criteria for the mod $\ell$ image to be ``large'' outside a small finite explicit set for $r\geq 3$ and $\Qr$ with odd class number.
Our methods allow us to produce explicit examples with small coefficients; for example for $d \in \{12,18,24\}$ the curves

\begin{equation} \label{eq:introduction_examples}
y^3 - \zeta_3^2 \pi y^2 - \zeta_3^2 y= x^{d} + x^{d-1} + 7x^3 + 14 x^2 + 45\zeta_3\pi
\end{equation}
where $\pi = 1- \zeta_3$ have large image outside of a finite set of primes listed in  \S \ref{subsection_examples}. Note that here we have taken a translation of the usual model, the corresponding curves listed in \S \ref{subsection_examples} have the standard superelliptic model as do the other examples listed there.

In recent years there has been much work on finding explicit examples of abelian varieties over $\QQ$ with trivial endomorphism ring and large mod $\ell$ images.
This has seen considerable success for those of low dimension.
For example, \cite{Dieulefait_genus2} provides an algorithm which for such abelian surfaces returns finitely many $\ell$ where the representation is not surjective.
The paper \cite{Anni_Lemos_Siksek} gives a genus 3 hyperelliptic curve with the images of the mod $\ell$ representations being as large as possible for all $\ell$.
Their methods require the curve to have everywhere semistable reduction.
This is in contrast to our approach and that of \cite{Anni_VDok} where primes of bad potentially good reduction play a crucial role.

In \cite{Anni_VDok}, the theory of clusters \cite{M2D2} is employed to determine the images of certain inertia subgroups.
The authors of \cite{Anni_VDok} then construct a genus 6 hyperelliptic curve with mod $\ell$ images as large as possible for all primes $\ell$.
To do this the Chinese Remainder Theorem is applied which in turn leads to the coefficients of the resulting hyperelliptic curve being large.
Our method, however, uses $\Dv-$regular curves \cite{Models_Curves_DVR}. This has a significant advantage in that the coefficients may be taken to be relatively small which allows the verification of large genus examples.

We also note the above references all require the ground field of their curves to have no unramified extensions, whereas we allow $\Qr$ with odd class number.

The image of a mod $\ell$ representation attached to the jacobian of a hyperelliptic curve is well-known to lie inside $\GSp_{2g}(\ell)$, where $g$ is the genus of the curve.
Indeed, for any principally polarised abelian variety $A/K$ the Weil pairing $\langle,\rangle \colon A[\ell]\times A[\ell] \rightarrow \Fl$ provides a non-degenerate symplectic pairing, which $G_K$, the absolute Galois group of the ground field, preserves up to similitude.
The similitude factor is given by $\chi_\ell$ the mod $\ell$ cyclotomic character, that is, for $P,Q\in A[\ell]$ and $\sigma \in G_K$ the relation $\langle \sigma(P), \sigma(Q) \rangle = \langle P, Q \rangle^{\chi_\ell(\sigma)}$ is satisfied.

Furthermore, for $\ell \neq 2$ there are no additional restrictions for a hyperelliptic curve with trivial endomorphism ring.
Thus, as the cyclotomic character is surjective (over $\QQ$), it suffices to show the mod $\ell$ image contains the isometry group $\Sp_{2g}(\ell)$ to conclude the representation is surjective.

The situation is radically different when $A/K$ has non-trivial endomorphism ring.
First, $G_K$ must normalise the endomorphisms of $A$ and further commute with those that are defined over $K$.
Second, and more subtly, the action of the endomorphism ring on the space of regular differentials $\Omega^1(A)$ comes into play.
This is discussed at length in \S \ref{section_endo_char}.

Let us expand upon what the above conditions mean in our situation, where $J={\rm Jac}(C)$ is a superelliptic jacobian $J/\Qr$ with endomorphism ring $\End(J) = \End_{\Bar{\QQ}}(J) \cong \Zr$.

First, $\End(J)$ is generated by the automorphism $[\zr]\colon J \rightarrow J$ which arises from the automorphism on the curve sending $(x,y) \mapsto (x,\zr y)$.
 This map preserves the Weil pairing (see Lemma \ref{lemma_preserves_Weil_pairing}).
It follows that the image of $\GQr$ is contained in the centraliser of $[\zr]$ inside $\GSp_{2g}(\ell)$.
The structure of this group is discussed in \S \ref{subsection_centralisers_in_GSp}.

We say the image of $\rho_\ell \colon \GQr \rightarrow \Aut(J[\ell])$ is \emph{large} if it  contains the largest perfect subgroup of the centraliser of $[\zr]$ in $\GSp_{2g}(\ell)$, or equivalently, the limit of its derived series.
In the case of hyperelliptic curves (defined over $\QQ$), this means exactly that $\rho_\ell(\GQ)$ contains $\Sp_{2g}(\ell)$.
We give a similar down to earth description for superelliptic curves in \S \ref{subsection_centralisers_in_GSp}.

In fact, we prove the following theorem (see Section \ref{section_Galois_images} for a more detailed statement). We note the class number of $\Qr$ is non-trivial precisely when $r>19$ \cite{Masley_Montgomery_cyclotomic_fields_UFD}.
Our notation for the various groups listed in the theorem below is detailed in \S \ref{subsection_classical_groups_and_their_automorphisms}.
The powers on the groups refer to direct products.
\begin{theorem}[$\subseteq$ Theorem \ref{thm_mod_l_image_is_huge}]
 Let $d \geq 12$ be a natural number divisible by $2r$ which is also the sum of two distinct primes $q_1<q_2$.

 Suppose there exists a prime $q_2<q_3<d$. If $r>23$ assume the class number of $\Qr$ is odd and $d=q_3+1$.

Let $n = \frac{2g}{r-1}$. Then given a polynomial $f \in \Qr[x]$ of degree $d$ whose coefficients satisfy certain congruence conditions, the image of the representation $\rho_\ell\colon\GQr \rightarrow \Aut(J[\ell])$ contains
\begin{itemize}
    \item $\SL_n(\ell^i)^{\frac{r-1}{2i}}$ if $i$ the inertia degree of $\ell$ in $\Qr$ is odd; and
    \item $\SU_n(\ell^{i/2})^{\frac{r-1}{i}}$ if $i$ the inertia degree of $\ell$ in $\Qr$ is even
\end{itemize}
for all $\ell$ outside of a finite explicit set. 
\end{theorem}

Identifying the exact image of $\rho_\ell$ is a more arduous task. We do this completely for $r=3$ using the endomorphism character studied in \S \ref{section_endo_char} (see also Theorem \ref{thm_r=3_exact_image_surjective}).
As far as the author is aware, this is the first time images of such representations have been correctly identified\footnote{Upton considers curves of the form $y^3=f(x)$ with $f \in \QQ[\zeta_3]$ of degree $4$ \cite{Upton}. However an oversight when taking the determinant leads to the wrong image being stated.}. 
The images of the mod $\ell$ representations coming from the examples in (\ref{eq:introduction_examples})  are (outside the finite set of $\ell$ listed in \S \ref{subsection_examples}):
\[\rho_\ell(\Gcyclo{3}) = \GL_{d-2}(\ell)^{\frac{d}{3},6} \rtimes \langle \chi_\ell \rangle \text{ for $\ell \equiv 1 \mod{3}$, and}
\]
\[\rho_\ell(\Gcyclo{3}) = \GU_{d-2}(\ell)^{\frac{d}{3},6} \rtimes \langle \chi_\ell \rangle \text{ for $\ell \equiv 2 \mod{3}$.}
\]
The notation here is as follows
\[\GL_n(\ell)^{s, t} = \{\sigma \in \GL_n(\ell)| \det(\sigma) \in \langle a^s,b\rangle \}\] where $a$ generates $\Fl^*$ and $b \in \Fl^*$ has order $t$; and
\[\GU_n(\ell)^{s,t} =
\{ \sigma \in \GU_n(\ell)| \det(\sigma) \in \langle a^s,b\rangle \}\] where $a$ generates $(\FF_{\ell^2}^*)^{\ell-1}$ and $b \in \FF_{\ell^2}^*$ has order $t$.

In Section \ref{section_lambda_adic_reps} we define representations $\rho_\lambda$ of $\GQr$ for $\lambda| \ell$.
For $\ell \equiv 1 \mod{r}$, the maximal image of the $\rho_\lambda(\GQr)$ is $\GL_n(\ell)$, where $n = \frac{2g}{r-1}$ and $g$ is the genus of $C$.
Whenever the image of $\rho_\ell$ is large,  $\rho_\lambda$ surjects (Proposition \ref{Prop_l_1_mod_r_IGP}, see also Theorem \ref{thm_l_equiv_1_mod_r_image_of_mod_lambda}).
The examples in (\ref{eq:introduction_examples}) thus also satisfy
\[\rho_\lambda(\Gcyclo{3}) = \GL_{d-2}(\ell)  \text{ for $\ell \equiv 1 \mod{3}$.}\]

For $\ell \equiv -1 \mod{r}$, the maximal image of the $\rho_\lambda(\GQr)$ is $\DU_n(\ell)$, the similitude group of a non-degenerate unitary pairing. However this upper bound is not achieved for all $\ell$, though we can list values of $\ell$ for which it is (Proposition \ref{prop_DU_image}, see also Theorem \ref{thm_DU_image_r=-1}).
The examples in (\ref{eq:introduction_examples}) satisfy
\[\rho_\lambda(\Gcyclo{3})=\DU_{d-2}(\ell) \text{ for $\ell \equiv 5,29 \mod{36}$.}\]

In Section \ref{section_group_theory} we give the relevant background on group theory.
Including a discussion of centralisers in the symplectic group and the restrictions they impose on the Galois images, the classification of maximal subgroups of certain classical groups and products of classical groups.
The classification of maximal subgroups uses modern tools from group theory.
We believe these tools may be of significant use to number theorists studying questions related to images of Galois representations.

Section \ref{section_inertia} gives local conditions on $f$ which in turn provides a reasonably comprehensive description of possible inertia subgroups.

In section \ref{section_image_of_inertia}, we use local Galois groups to rule out the images of the $\rho_\lambda$ being contained in maximal subgroups.

Section \ref{section_endo_char} reviews well-known results on algebraic Hecke characters and explains their relevance to our setting in identifying the exact image of Galois.
Several examples are given.

The last section ties together results from the previous sections and gives a method for constructing superelliptic jacobians with large images outside a finite explicit set of primes.
We then finish by giving a number of examples.

\begin{acknowledgements}
The author would like to thank his supervisor Tim Dokchitser for suggesting this project. He would further like to thank Tim Dokchitser and Scott Harper for many useful conversations.
He also thanks Naemi Fischer, Pedro Lemos, David Loeffler, Jeremy Rickard, David Zywina and the anonymous reviewer for their help and comments.

The majority of this work was carried out during the author's PhD at the University of Bristol, which was supported by the Engineering and Physical Sciences Research Council, grant number EP/N509619/1.

The author also made changes to the paper whilst in Université Clermont Auvergne supported by an early career fellowship from the London Mathematical Society, and during the time he transitioned from Univeristät Bayreuth to the Univeristat de Barcelona, where he was supported by the Deutsche Forschungsgemeinschaft (DFG) project grant STO~299/18-2 (AOBJ: 686837) and by the Spanish Ministry of Science and Innovation via the grant "Abelian varieties, L-functions, and rational points" (code PID2022-137605NB-I00) respectively.
He thanks these institutions and funding bodies for both their support and providing excellent working conditions.

\end{acknowledgements}

\end{section}
\begin{section}{$\lambda$-adic representations}
\label{section_lambda_adic_reps}
Fix distinct primes $r , \ell$. Let $f \in \Qr[x]$ be a polynomial  of degree $d \geq 5$ without repeated roots. The jacobian $J$ of the superelliptic curve
$$C \colon y^r = f(x)$$
has a natural endomorphism $[\zr]\colon J \rightarrow J$ which arises from the map on $C$ defined by $(x,y) \mapsto (x,\zr y)$.
This endomorphism realises $\ZZ[\zr]$ as a subring of $\End(J)$.
For the ease of exposition, let us assume the endomorphism ring $\End(J)$ is isomorphic to $\Zr$.
In order to study the representation of the absolute Galois group $\GQr$ on $J[\ell]$, the $\ell$-torsion of $J$, we first study $\lambda$-adic representations of $\GQr$ associated to $J$, where $\lambda|\ell$ is a prime above $\ell$ in $\Qr$. Let us introduce these now.

The Galois group $\GQr$ acts continuously on the Tate module $T_\ell(J)$, which gives rise to a $\Ql$-representation,
$$\rho_{\ell^\infty}  \colon \GQr \rightarrow \Aut(V_\ell) $$
where $V_\ell(J) = T_\ell(J) \otimes \QQ_\ell.$
The action of the endomorphism algebra $\End^0(J) \cong \QQ(\zeta_r)$ on $V_\ell(J)$ is compatible with that of $\Ql$, allowing us to view $V_\ell(J)$ as a $\QQ(\zeta_r)_\ell = \Qr \otimes \Ql$-module.
The idempotents giving rise to the decomposition $\Qr_\ell = \prod_{\lambda | \ell} \Qr_\lambda$ thus induce a decomposition $V_\ell = \prod_{\lambda | \ell} V_\lambda$.
As the endomorphisms of $J$ are defined over $\Qr$, the idempotents commute with the action of $\GQr$ on $V_\ell(J)$, leading to representations $$ \rho_{\lambda^\infty} \colon \GQr \rightarrow \Aut_{\Qr_\lambda}(V_\lambda),$$
one for each $\lambda|\ell$. We call these  \emph{$\lambda$-adic representations}.

The dimension of each of the vector spaces $V_\lambda$ over $\Qr_\lambda$ is equal to $n=\frac{2g}{r-1}$, \cite[Thm 2.1.1]{ribet_RM}. This entails that $V_\ell$ is a free $\Qr_\ell$-module of rank $n$.
In order to prove our main results, one may assume $n \geq 10$ throughout.
However, many of the results presented here hold for smaller values of $n$ and we indicate these as we go along.

As with $V_\ell$, we may view $T_\ell$ as a $\Zr \otimes \Zl$-module. Since $\Zr_\ell \coloneqq \Zr \otimes \Zl$ is a product of discrete valuation rings it follows from the above that $T_\ell$ is a free  $\Zr_\ell$-module of rank $n$ \cite[Prop 2.2.1]{ribet_RM}. The module $T_\lambda = T_\ell \otimes_{ \Zr_\ell} \Zr_\lambda$ is then a lattice in $V_\lambda$ 
and setting $J[\lambda] =T_\lambda \otimes_{ \Zr_\lambda} k_\lambda  $, where $k_\lambda$ is the residue field, we obtain a representation
\[\rho_\lambda \colon \GQr \rightarrow \Aut(J[\lambda]).
\]
The image of this representation will be of central interest to us. In fact, the subgroup $\Glam = \rho_\lambda(\GQrl)$ of the image will be of even greater importance.

Work of Shimura \cite[\S 11.10]{shimura_ANF}, \cite[Thm 2.1.2]{ribet_RM} shows the system of representations $(\rho_{\lambda^\infty})_\lambda$ is a strictly compatible system of $\Qr$-rational representations. Let us recall what this means.
\begin{definition}
We say that a collection of representations $\rho_{\lambda^\infty}\colon \GQr \rightarrow \GL_n(\Qr_\lambda)$ (one for each prime $\lambda$ of $\Qr$) forms a strictly compatible system if there exists a finite set $S$ of primes of $\Qr$ such that
\begin{itemize}
    \item  any $\rho_{\lambda^\infty}$ is unramified outside of $S\cup S_\ell$ (where $\ell$ is the rational prime below $\lambda$, and $S_\ell$ contains the primes above $\ell$ in $\Qr$);
    \item for each $\p \not \in S$, there is a monic polynomial $P_\p \in \Zr[x]$ whose image in $\Qr_\lambda[x]$ coincides with the characteristic polynomial of $\rho_{\lambda^\infty}(\Frob_\p)$ for any $\lambda$ whenever $\p \not \in S \cup S_\ell$.
\end{itemize}
\end{definition}
The set $S$ in our case is taken to be the primes of bad reduction for $J/\Qr$.

It is clear from the above definition that the system of  representations $(\det \circ \rho_{\lambda^\infty})_\lambda$ is also strictly compatible.
Each representation in this family is abelian, and so by \cite[\S 1.1, Prop 1.4]{Periods_of_Hecke_characters} (see also \cite[Thm 2.13]{ribet_RM}, \cite[Main Thm]{Henniart}) there is an algebraic Hecke character $\Omega$ of $\Qr$ with values in $\Qr$ such that for every prime $\lambda$ of $\Qr$,  the $\lambda$-adic avatar (see \S\ref{subsection_Alg_Heck_chars}, or \cite[\S0.5]{Periods_of_Hecke_characters}) $\Omega_\lambda $ equals $ \det \circ\rho_{\lambda^\infty}$. We call $\Omega$ the \emph{endomorphism character}. In section \ref{section_endo_char}, we shall study $\Omega$ in detail and describe the mod $\lambda$ images of its $\lambda$-adic avatars.

Finally, the following proposition will be useful in understanding the action of the inertia groups.

\begin{proposition}
\label{prop:rationality_andlambda_adicreps}
    Let $G$ be a group,  $E/\QQ$ a finite Galois extension, $n \geq 1$ an integer and $V$ an  $(E\otimes \QQ_\ell)[G]$-module such that the underlying $E\otimes \QQ_\ell$-module structure is free of rank $n$.
    Let $\theta$ denote the corresponding representation $\theta \colon G \rightarrow \Aut(V)$.
    Decompose $V$ as above into $\prod_{\lambda |\ell} V_\lambda$ where $\lambda$ runs through the primes above $\ell$ in $E$ and label the corresponding $E_\lambda$-representations of $G$ by $\theta^\lambda$.

    Suppose that $\theta$ may be realised over $E  \subset E\otimes \QQ_\ell$ (where the inclusion is given by $\alpha \mapsto \alpha \otimes 1$).
    Then for each $\lambda$, there is an embedding $\tau \colon E \hookrightarrow E_\lambda$ such that for any  $g \in G$, we have $\theta^\lambda(g) = \theta^\tau(g)$.
    That is, after fixing appropriate bases, the matrix $\theta^\lambda(g)$ is equal to the matrix obtained from applying $\tau$ to the entries of $\theta(g)$.

    In particular, for $\lambda, \lambda'$ primes above $\ell$ we have $\theta^{\lambda'}=(\theta^{\lambda})^\gamma$ for some $\gamma \in \Gal(E/\QQ)$.
\end{proposition}

\begin{proof}
    Let us begin by fixing an isomorphism $\varphi \colon E\otimes_\QQ \QQ_\ell \xrightarrow{\sim} \prod_{\lambda|\ell} E_\lambda$.
    By the Primitive Element Theorem, we may write $E = \QQ(\alpha)$ for some $\alpha$ with minimal polynomial $f$.
    Writing $f(x) = f_1(x) \cdots f_r(x)$ for the factorisation of $f(x)$ over $\QQ_\ell$, we have $\varphi(\alpha \otimes 1) = (\alpha_1, \ldots , \alpha_r)$ where $\alpha_i$ is a root of $f_i(x)$ in $\Bar{\QQ}_\ell$ \cite[Pg. 163]{Neukirch_ant_book}.

    Now, by assumption, there is a basis $v_1, \ldots, v_n$ of $V$ with respect to which $\theta$ takes values in $E$.
    Let $e_\lambda$ denote the idempotent which cuts out $V_\lambda$ from $V$.
    It is easy to see $e_\lambda v_1,\ldots,e_\lambda v_n$ form a basis of $V_\lambda$.
    Let $\tau\colon E \rightarrow E_\lambda$ be the embedding satisfying $\tau (\alpha) = e_\lambda \varphi(\alpha \otimes 1) $.

    Let $g \in G$.
    We have $g\cdot v_i = \sum \varphi(\beta_{i,j} \otimes 1) v_j$ for some $\beta_{i,j} \in E$.
    The claim now follows by considering \[g \cdot (e_\lambda v_i) = e_\lambda( g \cdot  v_i) = e_\lambda \sum \varphi(\beta_{i,j} \otimes 1) v_j = \sum \tau(\beta_{i,j}) e_\lambda v_j.\]
\end{proof}

\end{section}

\begin{section}{Group Theory}
\label{section_group_theory}
\begin{subsection}{Classical groups and their automorphisms}
\label{subsection_classical_groups_and_their_automorphisms}
Let $i \in \NN_{>0}$ be a natural number and $V$ be an $n$-dimensional vector space over $\Fli$. 
We denote the group of $\Fli$-semilinear invertible transformations of $V$ by $\GammaL_n(\ell^i)$. 

Let $\langle , \rangle \colon V \times V \rightarrow \Fli$ be either identically zero, a non-degenerate symplectic pairing or a non-degenerate unitary pairing.
We note the last case may only occur if $i$ is even.

We call a semisimilarity any element $\tau \in \GammaL_n(\ell^i)$ for which there exists $\chi(\tau) \in \Fli$, $\sigma \in \Aut(\Fli)$ such that
\[\langle \tau v, \tau w \rangle = \chi(\tau) \langle  v,  w \rangle^\sigma\]
for all $v,w \in V$.
The set of semisimilarities forms a group $\Gamma$.
The similitude group $\Delta \leq \Gamma$ is the subgroup for which $\sigma$ is the identity. This group coincides with $\Gamma \cap \GL_n(\ell^i)$ \cite[Lemma 2.1.2 (iv)]{KleidmanLiebeck}.
The isometry group $I \leq \Delta$ is the subgroup which preserves $\langle, \rangle$; that is the $\tau \in \Delta$ for which $\chi(\tau)=1$.
We denote the kernel of $\det\colon I \rightarrow \Fli$ by $S$.
If $\langle, \rangle$ is identically zero, let $A=\Gamma\langle \iota \rangle$ where $\iota$ is the inverse-transpose map $u \mapsto u^{-T}$. Else let $A=\Gamma$.

When $\langle, \rangle$ is a unitary pairing, $\DU_n(\lihalf)$ denotes the similitude group,
$\GU_n(\lihalf)$ the isometry group,
and $\SU_n(\lihalf)$ the kernel of the determinant map $\det \colon \GU_n(\lihalf) \rightarrow \Fli$. If $\langle, \rangle$ is identically zero, the similitude and isometry group coincide with $\GL_n(\ell^i)$.
When $\langle, \rangle$ is symplectic the similitude group is denoted by $\GSp_n(\ell^i)$. In this case the isometry group $\Sp_n(\ell^i)$ consists of determinant one matrices.

The above defines a natural chain of groups
\[S \leq I \leq \Delta \leq \Gamma \leq A.\]
This chain is $A$ invariant, that is, each group is normalised by $A$.

The group of scalar matrices isomorphic to $\Fli^*$ is a normal subgroup of $A$.
This allows us define $\Bar{X}$ the reduction modulo scalars for any group $X$ in the chain above.

For simplicity suppose $\ell\geq 5$ and $n \geq 4$. Then we have the following.
\begin{theorem}{\cite[Theorems 2.1.3 and 2.1.4]{KleidmanLiebeck}}
$\Bar{S}$ is a non-abelian simple group and its automorphism group $\Aut(\Bar{S})$ equals $\Bar{A}$.
\end{theorem}
References to both original and modern proofs of the above may be found on page 16 of \cite{KleidmanLiebeck}.

Combining the above theorem with the well-known lemma below gives us a corollary of great importance for \S \ref{subsection_products_of_classical_groups}.

\begin{lemma}
Let $\Bar{S}$ be a non-abelian simple group. Let $\Bar{S}\leq H \leq \Aut(\Bar{S})$ where $\Bar{S}$ is identified with its inner automorphism group. Then
\[N_{\Aut(\Bar{S})}(H) \cong \Aut(H)\]
via the canonical map.
\end{lemma}

\begin{corollary}
\label{cor_auts_of_almost_simple_groups}
Let $\Bar{S} \leq H \leq \Aut(\Bar{S})$, $n \geq 4$. Then if
\begin{itemize}
    \item $\Bar{S} = \PSL_n(\ell^i)$, every automorphism  of $H$ is of the form $u \mapsto (MuM^{-1})^{\sigma}$ or $u \mapsto (Mu^{-T}M^{-1})^{\sigma}$ where $M \in \PGL_n(\ell^i)$ and $\sigma \in \Aut(\Fli)$.
    \item $\Bar{S} = \PSU_n(\ell^{i/2})$, every automorphism of $H$ is of the form $u \mapsto (MuM^{-1})^{\sigma}$ where $M \in \PGL_n(\ell^i)$ normalises $\Bar{S}$ and $\sigma \in \Aut(\Fli)$.
\end{itemize}
\end{corollary}

We will make use of the following lemma in \S \ref{ss:ruling_out_subfield_sbgps_and_classical_sbgps}.

\begin{lemma}
\label{lemma_symplectic_characteristic_poly}
Let $M \in \GSp_{2g}(\ell^i)$, then $\phi$, the minimal polynomial of $M$, satisfies $$ \phi(x) = \frac{x^{2g}}{\chi(M)^g}\phi\left(\frac{\chi(M)}{x}\right).$$
\end{lemma}

\begin{proof}
We have $M^TBM=\chi(M)B$, where $B$ represents $\langle,\rangle$ with respect to some basis. Using this we obtain the following:
\begin{align*}
\det(M- x I_n)  & = \det(BMB^{-1} - x I_n) \\
 & = \det(M^{-T}\chi(M) - x I_n) \\
& = x^{2g}\det(M)^{-1}\det(M- \frac{\chi(M)}{x} I_n). 
\end{align*}
The determinant of $M$ is equal to $\chi(M)^g$ \cite[Lemma 2.4.5]{KleidmanLiebeck}, so the result follows.
\end{proof}

When necessary we shall write a subscript on the determinant map $\det_{\ell^i}$ to emphasise the determinant is being taken with the underlying vector space viewed over $\Fli$.
\end{subsection}
\begin{subsection}{Centralisers in the symplectic group}
\label{subsection_centralisers_in_GSp}
In the next subsection, we will see that the automorphism $[\zeta_r]$ of $J$ preserves the Weil pairing and commutes with the action of the Galois group $\GQr$.
Recall that in \S \ref{section_lambda_adic_reps}, we defined subrepresentations
\[ \rho_\lambda\colon \GQr \rightarrow \Aut(J[\lambda])
\]
in order to help study the representation attached to the $\ell$-torsion
\[\rho_\ell \colon \GQr \rightarrow \Aut(J[\ell]).
\]
In this subsection we will lay the purely group theoretic foundations needed to give a rough ``upper bound'' for the images of the above representations.
That is, we shall now give describe the centraliser of an element of prime order $r \neq \ell$ and trivial determinant in $\GSp_n(\ell)$.

The description of these centralisers is due to Wall \cite{wall}, see also \cite[Chapter 3.4.1]{Burness}.
Giving this description will rely upon  various other facts about the classical groups and we will recall these as we go along, our main reference for these is \cite{KleidmanLiebeck}.

Let $\zeta \in \Sp_{2g}(\ell)$ be an element of prime order $r \neq \ell$ which fixes no non-zero vector. Let $i$ be the least positive integer such that $\FF_{\ell^i}$ contains $r$-th roots of unity, equivalently, the inertia degree of $\ell$ in $\rcyclo$.

If $i$ is \emph{odd}, the characteristic polynomial of $\zeta$ is of the form
$$ (\phi_1 \Bar{\phi}_1 )^{a_1}\ldots (\phi_t \Bar{\phi}_t )^{a_t}$$
where each $\phi_j$ is irreducible of degree $i$ and $\Bar{\phi}_j$ is the polynomial whose roots are the multiplicative inverses of the roots of $\phi_j$.

If $i$ is \emph{even} then the characteristic polynomial of $\zeta$ is of the form
$$ \phi_1^{a_1}\ldots \phi_t^{a_t}$$
and each $\phi_j$ is irreducible of degree $i$ and has coefficients fixed by the involutory automorphism of $\Fli$.

In the following we let $G.n$ denote some extension of a group $G$ by a cyclic group of order $n$. We note also that as $1$ is not an eigenvalue of $\zeta$, there is no $\Sp_e(\ell)$ factor as appears in \cite{Burness} (See \cite[Prop. 3.4.3]{Burness}).

\begin{theorem}{\cite[Remark 3.4.4]{Burness}, \cite[Page 36]{wall}}
\label{centralisers_in_Sp}
The centraliser of $\zeta$ in $\GSp_{2g}(\ell)$ satisfies
$$C_{\GSp_{2g}(\ell)}(\zeta) = C_{\Sp_{2g}(\ell)}(\zeta)\langle \chi \rangle = C_{\Sp_{2g}(\ell)}(\zeta).(\ell-1) $$
where $\chi$ acts as a similarity on each factor in an orthogonal decomposition fixed by $\zeta$.
Furthermore if $i$ is odd, then
$$ C_{\Sp_{2g}(\ell)}(\zeta) \cong \GL_{a_1}(\ell^i) \times \cdots \times  \GL_{a_t}(\ell^i),$$
if $i$ is even, then
$$ C_{\Sp_{2g}(\ell)}(\zeta) \cong \GU_{a_1}(\ell^{i/2}) \times \cdots \times  \GU_{a_t}(\ell^{i/2}).$$
\end{theorem}

We will need to switch between the mod $\ell$ and mod $\lambda$ representations. Doing this amounts to understanding how the centralisers embed in $\GSp_{2g}(\ell)$.

Let us first consider $i$ odd.
In this case, using the characteristic polynomial, one can show $\zeta$ fixes an orthogonal decomposition of the form $$(U_{1,1} \oplus U_{1,1}^*) \bot \ldots \bot (U_{1,a_1}\oplus U_{1,a_1}^*) \bot \ldots \bot (U_{t,1}\oplus U_{t,1}^*) \bot \ldots \bot (U_{t,a_t}\oplus U_{t,a_t}^*)$$
where each $\{ U_{j,k}, U_{j,k}^*\}$ is a pair of totally isotropic $\langle \zeta \rangle$-irreducible spaces of dimension $i$ such that $U_{j,k} \oplus U_{j,k}^*$ is non-degenerate \cite[Prop. 3.4.3]{Burness}.
Furthermore $$U_j = U_{j,1} \bot \ldots \bot U_{j,a_j} $$ is equal to the sum of eigenspaces of $\zeta$ with eigenvalues the roots of $\phi_j$. Likewise $$U_j^* = U_{j,1}^* \bot \ldots \bot U_{j,a_j}^* $$ is equal to the sum of eigenspaces of $\zeta$ with eigenvalues the roots of $\Bar{\phi}_j$.

An element in the centraliser of $\zeta$ must preserve each of the $U_j$ and $U_j^*$.
In fact, viewing $U_j$ and $U_j^*$ as $a_j$-dimensional spaces over $\Fli$, an element of $C_{\Sp_{2g}(\ell)}(\zeta)$, the centraliser of $\zeta$ in $\Sp_{2g}(\ell)$, has, with respect to a suitable basis, a matrix on $U_j\oplus U_j^*$ of the form
$$ \left(\begin{array}{cc}
    A & 0 \\
    0 & A^{-T}
\end{array}\right) \text{ with } A \in \GL_{a_j}(\ell^i).$$
This describes an embedding $\GL_{a_j}(\ell^i) \hookrightarrow \Sp_{2a_j}(\ell^i) \leq \GSp_{2a_j}(\ell^i)$, which may be followed by embedding $\GSp_{2a_j}(\ell^i) \hookrightarrow \GSp_{2ia_j}(\ell)$.
Running over all $j$ we find $C_{\Sp_{2g}(\ell)}(\zeta) \leq \GL_{a_1}(\ell^i)\times \ldots \times \GL_{a_t}(\ell^i)$.

To correctly identify the image of Galois we need to go one step further and describe the centraliser of $\zeta$ in $\GSp_{2g}(\ell)$.
 As $U= U_1\oplus \cdots \oplus U_t$ and $U^* = U_1^*\oplus \cdots \oplus U_t^*$ are both totally isotropic and $U\oplus U^*$ is non-degenerate, an element which acts as a scalar on $U$ and the identity on $U^*$ belongs to the similitude group $\GSp_{n}(\ell)$. In particular, if $\langle\mu\rangle = \Fl^*$, then 
$$ \left(\begin{array}{cc}
    \mu I_{g} & 0 \\
    0 & I_{g}
\end{array}\right)$$
together with $\Sp_{2g}(\ell)$ generate $\GSp_{2g}(\ell)$.

The similitude factor $\chi$ of  $\GSp_{2a_j}(\ell^i)$ agrees with that of $\GSp_{2ia_j}(\ell)$ on elements whose image lie in $\Fl$ \cite[Lemma 4.3.5]{KleidmanLiebeck}. 
Thus by restricting the action of the above element to $U_j\oplus U_j^*$, we find
$$  \left(\begin{array}{cc}
    \mu I_{a_j} & 0 \\
    0 & I_{a_j}
\end{array}\right)
\text{ and }
\left(\begin{array}{cc}
    A & 0 \\
    0 & A^{-T}
\end{array}\right) \text{ with } A \in \GL_{a_j}(\ell^i)$$
belong to the centraliser of $\zeta$ on $U_j\oplus U_j^*$.
Ranging over $j$, we deduce $C_{\GSp_{2g}(\ell)}(\zeta) = C_{\Sp_{2g}(\ell)}(\zeta)\rtimes \langle \chi \rangle$.

Recall that when $i$ is \emph{even}, the characteristic polynomial of $\zeta$ is of the form
$$ \phi_1^{a_1}\ldots \phi_t^{a_t}$$
and each $\phi_j$ is irreducible of degree $i$ and has coefficients fixed by the involutory automorphism of $\Fli$.
In this case, $\zeta$ fixes an orthogonal decomposition of the form $$U_{1,1} \bot \ldots \bot U_{1,a_1} \bot \ldots \bot U_{t,1}\bot \ldots U_{t,a_t}$$
where each $ U_{j,k}$ is a non-degenerate $\langle \zeta \rangle$-irreducible space of dimension $i$ \cite[Prop. 3.4.3]{Burness}. Furthermore $$U_j = U_{j,1} \bot \ldots \bot U_{j,a_j} $$ is equal to the sum of eigenspaces of $\zeta$ with eigenvalues the roots of $\phi_j$.

An element in the centraliser of $\zeta$ must preserve each $U_j$. In fact, viewing $U_j$ as an $a_j$-dimensional space over $\Fli$, an element of $C_{\Sp_{2g}(\ell)}(\zeta)$, the centraliser of $\zeta$ in $\Sp_{2g}(\ell)$, preserves, when restricted to $U_j$, a non-degenerate unitary pairing.
Thus there is a clear containment \[C_{\Sp_{2g}(\ell)}(\zeta) \leq \GU_{a_1}(\ell^{i/2})\times \ldots \times \GU_{a_t}(\ell^{i/2}).\] 

Let us describe the structure of $C_{\GSp_{2g}(\ell)}(\zeta)$ and give an indication of how the $\GU_a(\ell^{i/2})$ factors are embedded into $\Sp_{2g}(\ell)$. This will be useful later when studying the image of our Galois representations in the quotient $\GU_a(\ell^{i/2})/\SU_a(\ell^{i/2})$. 

Let $\lr_{\sharp}\colon U_j \times U_j \rightarrow \Fli$ denote the unitary pairing from above.
Let $\mu \in \Fli^*$ be an element sent to zero under the trace map ${\mathrm T}\colon\Fli \rightarrow \FF_{\ell^{i/2}}$. Then the map ${\mathrm T}\left(\mu \lr_{\sharp}\right)\colon U_j \times U_j \rightarrow \FF_{\ell^{i/2}}$ is a symplectic pairing \cite[Pg.117 - 118]{KleidmanLiebeck}.
This provides us with an embedding of isometry groups $\GU_{a_j}(\ell^{i/2}) \hookrightarrow \Sp_{2a_j}(\ell^{i/2})$.
In fact, as the image of the similitude group $\DU_{a_j}(\ell^{i/2})$ of $\lr_{\sharp}$  under the multiplier lands in $\Flihalf$ \cite[Pg.23]{KleidmanLiebeck}, we also obtain an embedding of similitude groups $\DU_{a_j}(\ell^{i/2}) \hookrightarrow \GSp_{2a_j}(\ell^{i/2})$ \cite[Pg.118]{KleidmanLiebeck}.
Finally one embeds $\GSp_{2a_j}(\lihalf)$ into $\GSp_{ia_j}(\ell) \leq \GSp_{n}(\ell)$ in the usual way. As the multiplier $\chi_\sharp$ of $\DU_{a_j}(\ell^{i/2})$ agrees with that of $ \GSp_{2a_j}(\ell^{i/2})$, we have that if $\sigma \in \DU_{a_j}(\ell^{i/2})$ satisfies $\chi_\sharp(\sigma) \in \Fl$, then $\chi_\sharp(\sigma) =\chi(\sigma)$, where $\chi$ is the multiplier of $\GSp_{ia_j}(\ell)$ (and $\GSp_{n}(\ell)$) \cite[Lemma 4.3.5]{KleidmanLiebeck}.

 The following corollary of Theorem \ref{centralisers_in_Sp} is of particular interest to us.
We note that by $H^t$ for $H$ a group and $t \geq 1$ a natural number, we mean the direct product of $t$ copies of $H$.
 
\begin{corollary}
\label{cor_centraliser_structure_of_zr}
Let $n = \frac{2g}{r-1}$. Suppose the characteristic polynomial of $\zeta \in \GSp_{2g}(\ell)$ is $$(x^{r-1} + \ldots + x +
1 )^n.$$

If $i$ the inertia degree of $\ell$ in $\rcyclo$ is odd, then 
$$ C_{\Sp_{2g}(\ell)}(\zeta) \cong \GL_{n}(\ell^i)^{t}$$
where  $2t$ is the number of distinct primes above $\ell$ in $\rcyclo$.

If $i$ is even, then 
$$ C_{\Sp_{2g}(\ell)}(\zeta) \cong \GU_{n}(\ell^{i/2})^t$$
where $t$ is the number of distinct primes above $\ell$ in $\rcyclo$.
\end{corollary}

\begin{proof}
Immediate by Dedekind-Kummer and Theorem \ref{centralisers_in_Sp}.
\end{proof}

\subsection{Restrictions on the image of $\rho_\ell$}
We shall now use results of the previous subsection to describe restrictions on the image of $\rho_\ell$.
Further restrictions will be analysed in \S \ref{section_endo_char}.
Throughout this subsection we set $n=\frac{2g}{r-1}$.

\begin{lemma}
\label{lemma_preserves_Weil_pairing}
    The linear map induced by $[\zeta_r]$ on $J[\ell]$ is invertible and preserves the Weil pairing.
\end{lemma}

\begin{proof}
    The inverse of $[\zeta_r]$ is given by taking its complex conjugate $[\Bar{\zeta}_r]$.
    Indeed, as the eigenvalues of each of these maps are units they always induce non-singular linear maps modulo any given prime.
    
    Standard properties of the Rosati involution 
\cite[pgs. 189, 192]{Mumford_ab_vars} along with Albert's Classification \cite[Thm. 2, pg. 201]{Mumford_ab_vars}
give us the following equality: \[\langle [\zr]P, [\zr]Q \rangle = \langle P, [\Bar{\zeta}_r][\zr]Q \rangle = \langle P,Q \rangle \text{ for $P,Q \in J[\ell]$}\]
where $\langle, \rangle$ denotes the Weil pairing in the above equality.
\end{proof}

\begin{lemma}
\label{lemma_char_poly}
    The characteristic polynomial of $[\zr]$ acting on $J[\ell]$ equals $( x^{r-1} + \ldots + x +
1 )^{n}$.
\end{lemma}

\begin{proof}
    The characteristic polynomial of $[\zr]$ on $T_\ell(J)$ equals $( x^{r-1} + \ldots + x +
1 )^{n}$.
Indeed, by \cite[Thm. 3.6]{Lang_CM} and \cite[\S5 Lemma 1, pg. 35]{Shimura_CM_book} the representation of $\End^0(J)$ on $V_\ell(J)$ is a sum of a multiple of the reduced representation of $\QQ[\zr] \cong \End^0(J)$ and a $0$-representation.
However, as $1 \in \QQ[\zr] \cong \End^0(J)$ does not kill any torsion points, we see there is no $0$-representation.
The claim for $T_\ell(J)$ then follows by restricting to $[\zr]$ and for $J[\ell]$ by reducing modulo $\ell$.
\end{proof}

Recall that we use $\chi_\ell$ to denote the mod $\ell$ cyclotomic character.

\begin{theorem}
    If $i$, the inertia degree of $\ell$ in $\rcyclo$, is odd, then the image of $\rho_\ell\colon \GQr \rightarrow \Aut(J[\ell])$ is contained in a group isomorphic to
$$\GL_{n}(\ell^i)^{t} \langle \chi_\ell \rangle$$
where  $2t$ is the number of distinct primes above $\ell$ in $\rcyclo$.

If $i$ is even, then the image of $\rho_\ell\colon\GQr \rightarrow \Aut(J[\ell])$ is contained in a group isomorphic to
$$\GU_{n}(\ell^{i/2})^t\langle \chi_\ell\rangle$$
where $t$ is the number of distinct primes above $\ell$ in $\rcyclo$.
\end{theorem}

\begin{proof}
Immediate from Theorem \ref{centralisers_in_Sp}, Corollary \ref{cor_centraliser_structure_of_zr} and Lemmas \ref{lemma_preserves_Weil_pairing} and \ref{lemma_char_poly}.
\end{proof}

We note that there are further restrictions on the image of the mod $\ell$ representation, as will be discussed in \S \ref{section_endo_char}.

\begin{theorem}
\label{thm_containment_of_the_image_of_GQr_and_GQrl}
The image of $\rho_\lambda \colon \GQr \rightarrow \Aut(J[\lambda])$ is contained in $\GL_n(\ell^i)$ if $i$ is odd and $\DU_n(\ell^{i/2})$ if $i$ is even.

Moreover, the subgroup $G_\lambda = \rho_\lambda(\GQrl)$ is contained in $\GL(\ell^i)$ if $i$ is odd and $\GU_n(\ell^{i/2})$ if $i$ is even.
\end{theorem}

\begin{proof}
    By Dedekind-Kummer, we may write $\lambda=(\ell, \phi(\zr))$ for $\phi$ some irreducible factor of $\Phi_r$ modulo $\ell$.
    Thus the action of $[\zr]$ on $J[\lambda] \cong T_\lambda/\lambda T_\lambda$ satisfies $\phi([\zr])=0$ (recall that by abuse of notation we use $[\zr]$ to denote both the endomorphism of $J$ and the linear map it induces on $J[\lambda]$).
    As $\phi$ is irreducible, it coincides with the minimal polynomial of $[\zr]$ on $J[\lambda]$.

    Let $\lambda'|\ell$ be a prime above $\ell$.
    Write $\lambda' = (\ell, \phi')$.
    If $\phi$ has a root in common with $\phi'$, then by irreducibility $\phi=\phi'$ and so $\lambda = \lambda'$.
    Thus by a dimension count, it is easy to see that the sum of eigenspaces (over $\Fl$) of $[\zr]$ with eigenvalues roots of $\phi$ is $J[\lambda]$.

    Hence, the discussion in \S \ref{subsection_centralisers_in_GSp} implies that $\rho_\lambda(\GQr)$ is contained in $\GL_n(\ell^i)$ if $i$ is odd, and in $\DU_n(\ell^{i/2})$ if $i$ is even.
    Moreover, it shows $\rho_\lambda(\GQrl)$ is contained in $\GL_n(\ell^i)$ if $i$ is odd, and in $\GU_n(\ell^{i/2})$ if $i$ is even.
\end{proof}

Since the image of the mod $\ell$ cyclotomic character sits diagonally in the action of $\rho_\ell(\GQr)$ on the $J[\lambda]$, we see the task of determining $\rho_\ell(\GQr)$ is really to determine the $\Glam$.
We break this down into two parts: first showing $\Glam$ contains $\SL_n(\ell^i)$ for $i$ odd (resp. $\SU_n(\ell^{i/2})$ for $i$ even), and then identifying $\det_{\ell^i}(\Glam)$.

Since $\SL_n(\ell^i)$ (resp. $\SU_n(\ell^{i/2})$) is perfect for $n\geq 3$ and $\ell\geq 5$, it suffices to prove $\rho_\lambda(\GQr)$ contains $\SL_n(\ell^i)$ (resp. $\SU_n(\ell^{i/2})$) for $i$ odd (resp. even) to accomplish the first part.

To achieve this we study the maximal subgroups of $\GL_n(\ell^i)$ and $\DU_n(\ell^{i/2})$ in \S \ref{subsection_maximal_subgroups}.
The image of inertia subgroups are then used to rule out the containment of $\rho_\lambda(\GQr)$ in any maximal subgroup which does not contain $\SL_n(\ell^i)$ when $i$ is odd and $\SU_n(\ell^{i/2})$ when $i$ is even.
Our method for proving primitivity is dependent on unramified extensions of the base field, see \S \ref{subsection_primitivity}.
For this reason we work with $\rho(\GQr)$ rather than $\Glam$ directly.

Let us now return to the definition of \emph{large} used in \S \ref{section_introduction}.
There we said the group $\rho_\ell(\GQr)$ was large if it contained the limit of the derived series of $C_{\GSp_{2g}(\ell)}(\zr)$. 
It follows from Corollary \ref{cor_centraliser_structure_of_zr} that this definition is equivalent to the following:

\begin{definition}
When $i$ is odd, we say the image of $\rho_\lambda\colon\GQr \rightarrow \Aut(J[\ell])$ is \emph{large} if $\rho_\ell(\GQr)$ contains $\SL_{n}(\ell^i)^t$ where $2t$ is the number of distinct primes above $\ell$ in $\rcyclo$.

When $i$ is even,  we say the image is \emph{large} if $\rho_\ell(\GQr)$ contains $\SU_{n}(\ell^{i/2})^t$ where $t$ is the number of distinct primes above $\ell$ in $\rcyclo$.
\end{definition}

For convenience, we make the definition:
\begin{definition}
We say $\Glam$ is \emph{large} when $i$ is odd if  $\Glam$ contains
 $\SL_{n}(\ell^i)$. When $i$ is even we say $\Glam$ is \emph{large} if  it contains
 $\SU_{n}(\ell^{i/2})$.
\end{definition}

\end{subsection}

\begin{subsection}{Maximal subgroups of classical groups}
\label{subsection_maximal_subgroups}
In this section we will prove the necessary results concerning maximal subgroups of classical groups.
For our applications in \S \ref{section_Galois_images}, we could, strictly speaking, get away with using results of  Zalesski\u{i} and Sere\v{z}kin \cite{Zalesskii_Serezkin_transvections_paper}  which concern maximal subgroups of classical groups containing transvections.
However, we prefer to give a classification of maximal subgroups containing a broader class of elements using techniques from modern group theory.

Our main reason for doing this is to provide the reader with alternative ways of constructing superelliptic curves with large Galois images should they desire it.
Indeed, the reader will see that the tools provided by modern group theory allow one to easily classify maximal subgroups containing certain classes of elements.

Let $S$ be one of $\SL_n(\ell^i)$, $\SU_n(\ell^{i/2})$, $\Sp_n(\ell^i)$ where $n>4$, and $G$ be a group satisfying $S \leq G \leq A$ (where $A$ is as in \S \ref{subsection_classical_groups_and_their_automorphisms}). For $\tau \in G$, we define $$ \nu(\tau) \text{ to be the codimension of the largest eigenspace of } \tau .$$ For example, a transvection $\tau$ satisfies $\nu(\tau)=1$.
Indeed, the condition $\nu(\tau)=1$ is equivalent to there being a constant $\mu \in \Fli$ such that $\tau - \mu I_n$ has rank one, and for a transvection we may simply take $\mu=1$.
In this section we shall give a description of maximal subgroups of $G$ which contain an element $\tau$ of odd prime order with $\nu(\tau) =1$.

To accomplish this task, we make use of the seminal work of Aschbacher \cite{Aschbacher}, that of Kleidman and Liebeck \cite{KleidmanLiebeck}, and of various other group theorists in recent times. Indeed, the maximal subgroups of $G$ have been shown to belong to one of eight natural geometric collections $\C_1, \ldots, \C_8$ or an exceptional set $\mathcal{S}$. 

Our first lemma will show any subgroup $H \leq G$ containing an element of the above form cannot lie in $\C_3, \C_4,\C_6$ or $\C_7$. For this reason we do not give a description of these families. On the other hand we now give a rough description of the groups belonging to $\C_1,\C_2,\C_5$ and $\C_8$ along with the geometric structure they stabilise.

\begin{table}[tp]
    \centering
    \begin{tabular}{|c| c |c|}
    \hline
    $\C_j$ & structure stabilised & rough description in $\GL_n(\ell^i)$ \\ [0.5ex] 
        \hline \hline
\multirow{2}{1em}{$\C_1$} & non-degenerate or totally  & \multirow{2}{8.5em}{maximal parabolic} \\
 & singular subspace & \\
\hline
\multirow{2}{1em}{$\C_2$} &  decompositions of the form & \multirow{2}{9.2em}{$\GL_a(\ell^i) \wr S_k$, $n = ka$} \\
& $V = \bigoplus_{j=1}^k V_j$, $\dim V_j =a$ & \\
\hline
$\C_5$ & subfields of $\Fli$ of prime index $b$ & $\GL_n(\ell^{\frac{i}{b}})$, $b$ prime \\
\hline
\multirow{3}{1em}{$\C_8$}  & \multirow{3}{14em}{non-degenerate classical forms} &  $\GSp_n(\ell^i)$, $n$ even \\ 
& & $\mathrm{GO}^\varepsilon_n(\ell^i)$, $\ell$ odd \\ 
& & $\GU_n(\ell^{i/2})$, $i$ even\\
\hline
    \end{tabular}
\end{table}

These families are referred to in the following way: $\C_1$ reducible subgroups; $\C_2$ imprimitive subgroups; $\C_5$ subfield subgroups; $\C_8$ classical subgroups.

\begin{lemma}
\label{Get_rid_Aschbacher_families}
Let $G$ be as above, with $n > 4$ and suppose $\tau \in H < G$ has odd prime order and $\nu(\tau) =1$. Then $H$ is contained in a maximal subgroup of type $\C_1, \C_2, \C_5, \C_8$ or $\mathcal{S}$.
\end{lemma}

\begin{proof}
We shall work through the subgroups of type $\C_3,\C_4,\C_6,\C_7$ successively showing $\tau$ cannot be contained in any of them.

We start with $\C_3$, the field extension subgroups. As $\nu(\tau)=1$, the Jordan normal form of $\tau$ cannot be that of a field extension subgroup \cite[Lemma 4.2]{Liebeck_Shalev}. Hence $H \not \in \C_3$.

For $\C_4$ subgroups we use \cite[Lemma 3.7]{Liebeck_Shalev} which tells us that an element with $\nu(\tau)=1$ cannot preserve a non-trivial tensor decomposition. Thus $H \not \in \C_4$.

Subgroups of type $\C_6$ may be dealt with using \cite[Lemma 6.3]{Burness_II}, which states that any element $\sigma$ belonging to a $\C_6$ type group must have $\nu(\sigma) \geq n/4$. As $n>4$, and $\nu(\tau) = 1$, we deduce immediately that $H \not \in \C_6$.

Finally, we deal with subgroups of type $\C_7$ by appealing to \cite[Lemma 7.1]{Burness_II}. Here, again, the fact that $\nu(\tau)=1$ would force $n=1$, but $n>4$ so this is not possible. Hence $H \not \in \C_7$.
\end{proof}

The groups belonging to $\mathcal{S}$ are known to be, modulo scalars, almost simple and act absolutely irreducibly on the underlying vector space. However a full list of possible groups in $\mathcal{S}$ is unknown. Nevertheless, results of Guralnick and Saxl \cite{Guralnick_Saxl} allow us to discount all groups appearing in $\mathcal{S}$.

\begin{lemma}
Let $\ell \geq 5$ and $G$ be as above with $n > 8$ (and $n \neq 10$ if $S= \Sp_n(\ell)$). Suppose $\tau \in H < G$ has odd prime order and $\tau - I_n$ has rank one. Then $H$ belongs to a subgroup of type $\C_1, \C_2, \C_5$, or $\C_8$.
\end{lemma}

\begin{proof}
By Lemma \ref{Get_rid_Aschbacher_families} and Theorem 7.1 of \cite{Guralnick_Saxl} (see also \cite[Table 2.3]{Burness_IV} for the exception when $n=10$) it suffices to show $H$ is not the alternating or symmetric group of degree either $\dim V + 1$ or $\dim V +2$ with $V$, the underlying vector space, isomorphic to the fully deleted permutation module.

Let us argue by the contrapositive.
Then any element $\sigma$ of order $\ell$ in $H$ has $\nu(\sigma) \geq \ell-3$.
Indeed, let $\FF[\Omega]$ be the corresponding permutation module (from which one obtains $V$).
As a member of $\mathrm{Sym}(\Omega)$, our element $\sigma$ is a product of $\ell$-cycles, and thus as $\langle\sigma\rangle$-modules, we have
$$\FF[\Omega] \cong \FF[\sigma]^a \oplus \FF^b $$
for some $a,b$ with $a\neq 0$.
The only eigenvalue of $\sigma$ is 1 and the above description of $\FF[\Omega]$ shows the codimension of the 1-eigenspace is $ a(\ell-1) \geq \ell-1$.
It follows that codimension of the $1$-eigenspace of $\sigma$ on $V$ has dimension at least $\ell-3$.
Thus $H$ does not contain transvections.

Furthermore, as the symmetric group has only two linear representations, $\det(V) \subseteq \{\pm 1 \}$.
It follows that $H$ cannot contain such an element $\tau$.
This completes our proof by contrapositive, and hence the lemma.
\end{proof}

The following well-known lemma allows us to translate between the maximal subgroups of our groups $S \leq G \leq A$ and that of their projectivisations.

\begin{lemma}
Let $M$ be a maximal subgroup of $G$ such that $MZ(G)=G$. Then $M$ contains $S$.

In particular, every maximal subgroup $M$ of $G$ not containing $S$, contains $Z(G)$ and hence gives rise to a maximal subgroup $\Bar{M}$ of $\Bar{G}$ not containing $\Bar{S}$.
\end{lemma}

As the proof of this lemma is very short, we shall provide it for the reader's convenience.

\begin{proof}
    We have $M \geq [M,M] = [MZ(G),MZ(G)] = [G,G] \geq [S,S] = S$, as $S$ is perfect \cite[Thm. 1.7, Prop. 3.7, Thm. 11.22]{Grove}.
\end{proof}

The next two theorems now follow directly from the above and the main theorem of \cite{KleidmanLiebeck}. 

\begin{theorem}
\label{thm_maximal_sbgps_GU}
Let $H$ be a proper irreducible subgroup of $G$ where $S = \SU_n(\ell^{i/2})$, $\ell \geq 5$, $n > 8$. Suppose $\tau \in H$ has odd prime order and $\tau-I_n$ has rank one. Then one of the following holds
\begin{enumerate}
    \item $H$ preserves a (transitive) imprimitivity decomposition of $V$;
    \item $H$ is contained in a subfield subgroup; or 
    \item $H$ contains $\SU_n(\ell^{i/2})$.
\end{enumerate}
\end{theorem}

\begin{theorem}
\label{thm_maximal_sbgps_GL}
Let $H$ be a proper irreducible subgroup of $G$ where $S = \SL_n(\ell^i)$, $\ell \geq 5$, $n > 8$. Suppose $\tau \in H$ has odd prime order and $\tau-I_n$ has rank one. Then one of the following holds
\begin{enumerate}
    \item $H$ preserves a (transitive) imprimitivity decomposition of $V$;
    \item $H$ is contained in a subfield subgroup; 
    \item $H$ is contained in a classical subgroup; or
    \item $H$ contains $\SL_n(\ell^i)$.
\end{enumerate}
\end{theorem}
\end{subsection}
\begin{subsection}{Generating products of classical groups}
\label{subsection_products_of_classical_groups}
Once we have shown $\Glam \coloneqq \rho_\lambda(\GQrl)$ is large, we will need to argue $\rho_\ell(\GQrl)$ is also large. To do this, we adapt a method pioneered by Serre and Ribet.
Our starting point is Goursat's Lemma. 

\begin{lemma}(Goursat's Lemma)
Suppose $H$ is a subgroup of a product of groups $G_1\times G_2$ such that each projection map surjects $H \rightarrow G_1,G_2$. Let $N_2$ (resp. $N_1$) be the kernel of the projection onto $G_1$ (resp. $G_2$).

Then, viewing $N_1$ (resp. $N_2$) as a subgroup of $G_1$ (resp. $G_2$), the image of $H$ in $G_1/N_1 \times G_2/N_2$ is the graph of an isomorphism $\phi\colon G_1/N_1 \rightarrow G_2/N_2$.
\end{lemma}

As we will, in general, have more than two factors to deal with, Ribet's Lemma will prove indispensable.

\begin{lemma}(Ribet's Lemma \cite[Lemma 5.2.1]{ribet_RM})
\label{ribet's_lemma}
Let $S_1, \ldots,S_k$ be finite perfect groups. Let $G$ be a subgroup of $S_1 \times \cdots \times S_k$ such that each projection $G \rightarrow S_i \times S_j$ ($1 \leq i<j\leq k$) is surjective. Then $G=S_1 \times \cdots \times S_k $.
\end{lemma}

Recall from Corollary \ref{cor_centraliser_structure_of_zr} that $C_{\Sp_{2g}(\ell)}(\zr) \cong G_1 \times \cdots\times  G_t$ where each $G_j \cong \GL_n(\ell^i)$ or $\GU_n(\ell^{i/2})$. In the following, we assume $n\geq 4$. 

\begin{proposition}
\label{prop_two_factors}
Let $H$ be a subgroup of $G = C_{\Sp_{2g}(\ell)}(\zr) \cong G_1 \times \cdots \times G_t$ such that each projection $H_j$ of $H$ onto $G_j$ contains the commutator subgroup $S_j$.
Suppose there is an element $\tau \in H$ such that its projection onto each $G_j$ has exactly one non-trivial eigenvalue of order $r$, and when viewed as an element of $\Sp_{2g}(\ell)$ its non-trivial eigenvalues are $\zr, \zr^2, \ldots, \zr^{r-1}$ and all of multiplicity one.
Then the image of $H \rightarrow G_j\times G_k $, $j \neq k$, contains $S_j \times 
S_k$. 
\end{proposition}

\begin{proof}

Denote the image of $H \rightarrow G_j\times G_k $ by $H_{jk}$. 
Let $N_j$ denote the kernel of $ H_{jk} \rightarrow H_k$ and $N_k$ the kernel of $ H_{jk} \rightarrow H_j$.
As $H_{jk}$ surjects onto each factor, we may view $N_j$ (resp $N_k$) as a subgroup of $G_j$ (resp $G_k$).
If either $N_s$ contains $S_s$ ($s =j,k$), then we are done. Let us suppose this is not the case. It follows that $N_s$ is contained in the centre $Z_s$ of $G_s$.

Goursat's Lemma  tells us the image of $H_{jk}$ in $G_j/N_j \times G_k/N_k$ is the graph of an isomorphism $\phi\colon G_j/N_j \rightarrow G_k/N_k$. As $\phi$ maps the centre of $G_j/N_j$ to the centre of $G_k/N_k$, there is an induced isomorphism $\Bar{\phi}\colon G_j/Z_j \rightarrow G_k/Z_k$ (the centre of $G_s/N_s$ is $Z_s/N_s$).

Now either $G_j \cong G_k \cong \GL_n(\ell^i)$ or $G_j \cong G_k \cong  \GU_n(\ell^{i/2})$. Let us suppose for now we are in the first case.
Then $\Bar{\phi}$ is an automorphism of $\PGL_n(\ell^i)$, by Corollary \ref{cor_auts_of_almost_simple_groups} we see that for $(h_j,h_k) \in H_{jk}$, we have either $h_k = \chi(h_j)(Mh_jM^{-1})^{\sigma}$ or $h_k = \chi(h_j)(Mh_j^{-T}M^{-1})^{\sigma}$ where $M \in \GL_n(\ell^i)$, $\sigma \in \Aut(\Fli/\Fl)$ and $\chi\colon\GL_n(\ell^i) \rightarrow \Fli$ is a linear character.

Let $\tau_s$ denote the projection of $\tau$ to $G_s$ and $\alpha_s$ be its non-trivial eigenvalue. The embeddings described in \S \ref{subsection_centralisers_in_GSp} show that when $G_s$ is viewed as a subgroup of $\GSp_{2g}(\ell)$, then, on the corresponding subspace, $\tau$ has eigenvalues $\alpha_s^\sigma, \alpha_s^{-\sigma}$ as $\sigma$ varies over all of $ \Aut(\Fli/\Fl)$.
As the non-trivial eigenvalues of $\tau$, when viewed as an element of $\Sp_{2g}(\ell)$ are distinct, we see $\alpha_k \neq \alpha_j^\sigma, \alpha_j^{-\sigma}$ for any $\sigma \in  \Aut(\Fli/\Fl)$.

As $\tau \in H$, we have by the above that either $\tau_k = \chi(\tau_j)(M\tau_jM^{-1})^{\sigma}$ or $\tau_k = \chi(\tau_j)(M\tau_j^{-T}M^{-1})^{\sigma}$. In either case, we must have $\chi(\tau_j)=1$, because $\tau_j$ and $\tau_k$ have only one non-trivial eigenvalue. These equalities now imply $\alpha_k = \alpha_j^\sigma$ or $ \alpha_j^{-\sigma}$ for some $\sigma \in  \Aut(\Fli/\Fl)$, but this contradicts the above. We conclude $H_{jk} \supseteq S_j \times S_k$.

Let us now assume $G_j \cong G_k \cong  \GU_n(\ell^{i/2})$, $n \geq 3$.
As above, Corollary \ref{cor_auts_of_almost_simple_groups} combined with Goursat's Lemma shows that for $(h_j,h_k) \in H_{jk}$, we have $h_k = \chi(h_j)(Mh_jM^{-1})^{\sigma}$. The above reasoning also applies to show the non-trivial eigenvalues of $\tau_j$ and $\tau_k$ cannot be Galois conjugate. These statements contradict, proving  $H_{jk} \supseteq S_j \times S_k$.
\end{proof}

Applying Ribet's Lemma gives the following:

\begin{theorem}
\label{thm_socle}
Let $H$ be a subgroup of $C_{\Sp_{2g}(\ell)}(\zr) \cong G_1 \times \cdots \times G_t$ such that each projection $H_j$ of $H$ onto $G_j$ contains the commutator subgroup $S_j$.
Suppose there is an element $\tau \in H$ such that its projection onto each $G_j$ has exactly one non-trivial eigenvalue of order $r$, and when viewed as an element of $\Sp_{2g}(\ell)$  its non-trivial eigenvalues are $\zr, \zr^2, \ldots, \zr^{r-1}$, all of multiplicity one.
Then $H \supseteq S_1 \times \cdots \times S_t$.
\end{theorem}

\begin{proof}
Let $S= S_1 \times \cdots \times S_t$. Proposition \ref{prop_two_factors} shows the image of the projection $H \cap S \rightarrow G_j \times G_k$ equals $S_j \times S_k$ for any $j \neq k$. 
Applying Ribet's Lemma (Lemma \ref{ribet's_lemma}) to $H \cap S$ yields the result.
\end{proof}
\end{subsection}
\end{section}

\begin{section}{Controlling inertia groups}
\label{section_inertia}

In this section we continue to let $\ell, p, r$ be distinct primes. Let $F/\Qp$ be a finite extension with ring of integers $\OO_F$.
Denote the discrete valuation of $F$ by $v_F$, a uniformiser by $\pi$, and by $I_F \trianglelefteq \Gal(\overline{F}/F)$ the inertia subgroup. 

For the convenience of the reader we recall the necessary notation from \cite{Models_Curves_DVR} for the below theorem.

Let $f = \sum_{i,j} a_{ij}x^iy^j \in F[x,y]$ be a non-zero polynomial, which is not a monomial.
The following are Newton polytopes of $f$:
$$
\begin{array}{lclrlll}
  \Delta & = & \text{convex hull}\bigl(\>  (i,j) &\bigm|\>a_{ij}\ne 0\>\bigr)  &\subset\RR^2,\\[2pt]
  \Dv    & = & \text{lower convex hull}\bigl(\>  (i,j,v_F(a_{ij}))\!\!\! &\bigm|\>a_{ij}\ne 0\>\bigr)  &\subset\RR^2\times\RR.
\end{array}
$$
with $\Delta$ being the ordinary Newton polygon of $f$. For every point $P = (i,j) \in \Delta$, we have a corresponding point $v_F(a_{ij}) \in \RR$. This gives us a piecewise affine map $v\colon\Delta \rightarrow \RR$ 
which breaks $\Delta$ into 
2-dimensional \emph{$v$-faces} and 1-dimensional \emph{$v$-edges},
the images of faces and edges of the polytope $\Dv$ under the homeomorphic projection to $\Delta$.

Let us write $\Delta(\ZZ)$ for the integer points lying inside $\Delta$, that is
$\Delta(\ZZ)=\text{interior}(\Delta)\cap\ZZ^2$.
We write $\Delta(\ZZ)^\F\subseteq\Delta(\ZZ)$ for points that are in the interiors of $v$-faces,
and $\Delta(\ZZ)^L$ for those lying on the $v$-edges (note $\Delta(\ZZ)^L = \Delta(\ZZ)\setminus\Delta(\ZZ)^\F$). For any of the above sets, we write $\ZZ_p$ as a subscript to indicate the subset of points for which $v(P)\in\ZZ_p$.

We shall later state conditions which imply $\Dv-$regularity, but we will not define it  here. Instead, we refer the reader to \cite[Definition 3.9]{Models_Curves_DVR}.

\begin{theorem}{\cite[Thm 6.4]{Models_Curves_DVR}}
\label{Tim}
Suppose $C/F$ is a $\Delta_v-$regular curve, and $\ell \neq p$. For $P \in \Delta(\ZZ)_{\ZZ_p}$ define a tame character

$$\chi_p\colon I_F \rightarrow \{\text{roots of unity}\}, \hspace{5mm} \sigma \mapsto \sigma(\pi^{v(P)})/{\pi^{v(P)}}$$

Let $V^{ab}_{tame}, V^{toric}_{tame}$ be the unique continuous $\ell$-adic representations of $I_F$ that decompose over $\Bar{\QQ}_\ell$ as

$$ V^{ab}_{tame} \cong_{\Bar{\QQ}_\ell} \bigoplus_{P \in \Delta(\ZZ)_{\ZZ_p}^\F}\left(\chi_P \oplus \chi_P^{-1}\right), \hspace{5mm} V^{toric}_{tame} \cong_{\Bar{\QQ}_\ell} \bigoplus_{P \in \Delta(\ZZ)_{\ZZ_p}^L}\chi_P.$$
Then there are isomorphisms of $I_F$-modules,
$$H_{\acute{e}t}^1(C_{\overline{F}},\QQ_\ell)^{I_{wild}} \cong V_\ell(J(C))^{I_{wild}} \cong V^{ab}_{tame}\oplus  V^{toric}_{tame} \otimes \mathrm{Sp}(2)$$
where $\mathrm{Sp}(2)$ denotes the 2-dimensional special $\ell$-adic representation (see \cite[4.1.4, 4.2.1]{Tate_number_theoretic_background}).
In particular, $J(C)$ is wildly ramified $\iff \Delta(\ZZ)_{\ZZ_P}\subsetneq \Delta(\ZZ)$.
\end{theorem}

\begin{remark}
\label{remark:ab_toric_tame_are_both_rational}
The representations $V^{ab}_{tame}, V^{toric}_{tame}$ are rational, that is they may be realised over $\QQ$, see the proof of \cite[Thm. 6.4]{Models_Curves_DVR}.
\end{remark}

\begin{remark}
\label{remark:Sp_has_rational_restriction}
    Let us briefly define the representation $\mathrm{Sp}(2)\colon I_F \rightarrow \GL_2(\QQ_\ell)$ and give a subgroup with respect to which it is rational.

    For every natural number $n$ fix an $\ell^n$-th root of unity $\zeta_{\ell^n}$ such that $\zeta_{\ell^{n+1}}^\ell = \zeta_{\ell^n}$.
    Define the $\ell$-adic tame character $t_\ell\colon I_F \rightarrow \ZZ_\ell$ by $\sigma \mapsto \sigma(\pi^{1/\ell^n}) = \zeta_{\ell^n}^{t_\ell(\sigma)}{\pi^{1/\ell^n}} $.
    Then $\mathrm{Sp}(2)\colon I_F \rightarrow \GL_2(\QQ_\ell)$ is given by 
    \[\sigma \mapsto \begin{pmatrix}
        1 & t_\ell(\sigma) \\
        0 & 1
    \end{pmatrix}.\]
    Clearly the preimage of matrices of the form $\left(\begin{smallmatrix}
    1 & m \\
       0  & 1
    \end{smallmatrix}
    \right)$,
    where $m \in \ZZ$, forms a subgroup of $I_F$ whose image is rational.
\end{remark}

Although we could get away with a less general version of the theorem below, we prove it in this form for its versatility, and afterwards deduce  all cases we actually use as examples. 

\begin{theorem}
\label{inertia_action_from_polytope}
Suppose $C/F$ is a $\Delta_v-$regular curve, $p \neq r,\ell$ with Newton polytope $\Dv$ of the form:
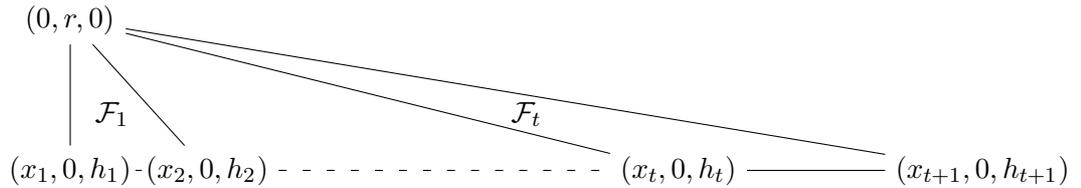
\begin{center}
\begin{tikzpicture}
\label{Fs_triangle_diagram}
    \node (1) at (0,0) {$(x_1,0,h_1)$};
    \node (0y) at (0,2) {$(0,r,0)$};
    \node (2) at (1.8,0) {$(x_2,0,h_2)$};

    \node (n) at (8,0) {$(x_t,0,h_t)$};
    \node (0x) at (12,0) {$(x_{t+1},0,h_{t+1})$};

    \draw (0y) to (1) to (2) to (0y) to (0x) to (n) to (0y);
    \draw [loosely dashed] (2) -- (n);

    \node (F1) at (0.55,0.75) {$\F_1$};
    \node (Fn) at (6,0.75) {$\F_t$};

\end{tikzpicture}
\end{center}
where
\begin{itemize}
    \item $q_1 , \ldots , q_t$ is an increasing sequence of primes distinct from $r$ and $p$;
    \item $h_1, \ldots, h_t, h_{t+1}$ a strictly decreasing sequence of natural numbers with $h_{t+1}=0$; and

\item $x_{s+1} = \sum^{s}_{j=1}q_{j}$, in particular $x_1 =0$. 
\end{itemize}

Let $D_s\coloneqq h_s - h_{s+1}$ for $1\leq s\leq t$.
Define the symbol $\gamma_s$ to be $0$ if $r|x_{s}$ and $ 1$ otherwise.
Define $\delta_s = 0$ if $q_sh_s+D_sx_s \equiv 0 \mod{r}$  and $1$ otherwise.
Then as $I_F$-modules,
\begin{align*}
V^{ab}_{tame}  & \cong_{\Bar{\QQ}_\ell} \bigoplus^t_{s=1}  (\QQ_\ell[C_{q_s}]^{D_s} \ominus \Ql )\otimes (\QQ_\ell[C_{r}]^{\delta_s}  \ominus \Ql) \oplus (\QQ_\ell[C_{r}]^{\delta_s} \ominus \Ql)^{\oplus \gamma_s + \gamma_{s+1} -1}  \\
V^{toric}_{tame} & \cong_{\Bar{\QQ}_\ell} \bigoplus_{s=2}^{t}(\QQ_\ell[C_{r}]^{h_{s}} \ominus \Ql)^{\oplus(1- \gamma_{s})} 
\end{align*}
 where $\QQ_\ell[C_{*}]^{a}$ means we raise each individual irreducible constituent of $\QQ_\ell[C_{*}]$ to the power of $a$. 
 
 Moreover, $I_F$ acts tamely on $V_\ell(J(C))$ and as $I_F$-modules,
 \[ V_\ell(J(C)) \cong V^{ab}_{tame}\oplus  V^{toric}_{tame} \otimes \mathrm{Sp}(2).
 \]
 Note $\dim V^{ab}_{tame}+ 2\dim V^{toric}_{tame} = (r-1)(x_{t+1}-2+\gamma_{t+1}) = 2g$.

\end{theorem}
\begin{proof}
Let us begin by looking at the contribution of one face at a time. To simplify notation, let us work with the face $\F$ as labelled below.
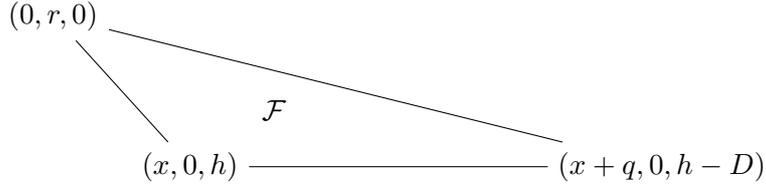
\begin{center}
\begin{tikzpicture}
\label{F_triangle_diagram}
    \node (0y) at (0,2) {$(0,r,0)$};
    \node (2) at (1.8,0) {$(x,0,h)$};

    \node (n) at (8,0) {$(x + q,0,h-D)$};

    \draw (0y) to (2) to (n) to (0y);

    \node (F) at (2.9,0.75) {$\F$};

\end{tikzpicture}
\end{center}

The points on $\F$ satisfy the equation $rDX+(qh+Dx)(Y-r)+rqZ = 0 $. To apply Theorem \ref{Tim}, it will be easier to work with the following extension of $\F$,

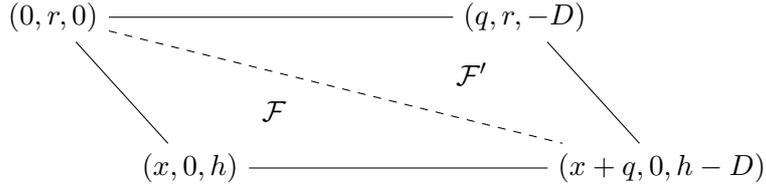
\begin{center}
\begin{tikzpicture}
\label{F'_triangle_diagram}
    \node (0y) at (0,2) {$(0,r,0)$};
    \node (2) at (1.8,0) {$(x,0,h)$};

    \node (n) at (8,0) {$(x + q,0,h-D)$};
    \node (q) at (6.2,2) {$(q,r,-D)$};

    \draw [ dashed] (0y) to (n);
    \draw (0y) to (2) to (n) to (q) to (0y);

    \node (F) at (2.9,0.75) {$\F$};
    \node (F') at (5.5,1.25) {$\F'$};

\end{tikzpicture}
\end{center}
We denote by $\hat{\F}$ the parallelogram obtained by joining $\F$ and $\F'$.
We shall call an integral point of $\hat{\F}$ a point $(X,Y,Z)$ lying on the interior of $\hat{\F}$ (in particular, not on its boundary) with integral $X,Y$ values.
There are $(q-1)(r-1)$ many integral points on $\hat{\F}$  if $r$ divides $x$ and $q(r-1)$ otherwise.
If $r|x+q$, then $r-1$ of these points lie on the line $L$ connecting $(0,r,0)$ to $(x+q,0,h-D)$.
We call the denominator of an integral point the denominator of $Z$ expressed in lowest terms.
Note the denominator of an integral point divides $rq$ and so by assumption is never divisible by $p$.
Thus $\Delta(\ZZ)_{\ZZ_P} = \Delta(\ZZ)$ and $J(C)$ is tamely ramified by Theorem \ref{Tim}.

There is a bijection between integral points on $\F$ and $\F'$ given by $$(X,Y,Z) \leftrightarrow (X,Y,Z)' = (x+q - X, r-Y,h-D - Z).$$ In the notation of Theorem \ref{Tim}, for a point $P$ on $\F$, the associated character satisfies  $\chi_{P'} = \chi_{P}^{-1}$.
Thus to determine the contribution of $\F$ towards $V^{ab}_{tame}$ it suffices to determine the equivalences classes of $rqZ$ modulo $rq$ as $(X,Y,Z)$ ranges over the integral points on $\hat{\F}$ not lying on $L$.

Let $Z(X,Y) = (qh+Dx)(r-Y) - rDX$.
Note $Z(X,Y)$ is an integer if $X$ and $Y$ are, and $(X,Y, \frac{1}{rq} Z(X,Y))$ is a point on $\hat{\F}$.
Suppose $X$ and $Y$ are integers.
Then $Z(X,Y) \equiv -(qh+Dx)Y \mod{r}$.
As $Y$ takes values strictly between $0$ and $r$, we see $Z(X,Y) \equiv 0 \mod{r}$ if and only if $-(qh+Dx) \equiv 0 \mod{r}$ and furthermore if $-(qh+Dx) \not \equiv 0 \mod{r}$ then the equivalence class $Z(X,Y) \mod{r}$ only depends on $Y$.

Let us now study $Z(X,Y) = (qh+Dx)(r-Y) - rDX$ modulo $q$. If $q|D$, then $Z(X,Y) \equiv 0 \mod{q}$ always holds.
Suppose $q\nmid D$, then $Z(X,Y) \equiv Dx(r-Y) - rDX \mod{q}$.
As $rD$ is invertible modulo $q$, we see for fixed $Y$ value and $X$ ranged over this line on $\hat{\F}$ that the quantity $Z(X,Y) \mod{q}$ runs over all equivalences classes if $r\nmid x$ and all non-zero classes if $r| x$.

The line $L$ contains integral points in $\hat{\F}$ if and only if $r|x+q$.
These points contribute towards $V^{toric}_{tame}$ if and only if they belong to $\D(\ZZ)$, equivalently, $x+q \neq x_{n+1}$.
The equation of the line $L$ connecting $(0,r,0)$ to $(x+q,0,h-D)$ is given by $\frac{-rX}{x+q}=Y-r=\frac{rZ}{D-h}$.
Thus we see that if $(X,Y,Z)$ is an integral point on $L$, then $Z(X,Y) = q(Y-r)(D-h)$ is always divisible by $q$ and  $Z(X,Y) \equiv qY(D-h) \equiv 0 \mod{r}$ if and only if $h-D \equiv 0 \mod{r}$.

Applying the Chinese Remainder Theorem and Theorem \ref{Tim} concludes the proof.
\end{proof}

\begin{corollary}
\label{cor_PGR_inertia_description}
Suppose $C/F$  is a $\Delta_v-$regular curve determined by the affine chart $y^r=f(x)$, with potentially good reduction and Newton polytope $\Dv$ of the form given in Theorem \ref{inertia_action_from_polytope}. Then as $\Qr_\lambda [I_F]$-modules, 
$$V_\lambda(J(C))^{I_{wild}}   \cong V_\lambda(J(C))   \cong \bigoplus^t_{s=1}  \bigoplus^{q_s-1}_{j=1}(\chi_{q_s}^{j D_s} \otimes \chi_{s,j}^{\delta_s})   \oplus (\chi_s^{\delta_s})^{\oplus \gamma_s + \gamma_{s+1} -1} $$
where the $\chi_{s,j}, \chi_s$ (resp. $\chi_{q_s}$) are primitive characters of order $r$ (resp. $q_s$). 

In particular, if $r \neq \ell$ and no $q_s$ equals $\ell$, then
$$\rho_\lambda|_{I_F} = \bigoplus^t_{s=1}  \bigoplus^{q_s-1}_{j=1}(\chi_{q_s}^{j D_s} \otimes \chi_{s,j}^{\delta_s})   \oplus (\chi_s^{\delta_s})^{\oplus \gamma_s + \gamma_{s+1} -1}$$
where by abuse of notation, we use the same symbols to denote the reductions of $\chi_{s,j}, \chi_s, \chi_{q_s}$.
\end{corollary}

\begin{proof}
Theorem \ref{inertia_action_from_polytope} gives us the isomorphism
\[ V_\ell(J(C))^{I_{wild}} \cong V_\ell(J(C)) \cong V^{ab}_{tame}\oplus  V^{toric}_{tame} \otimes \mathrm{Sp}(2)
\]
where $V^{toric}_{tame}=0$ as $C/F$ has potentially good reduction and
\[ V^{ab}_{tame}   \cong_{\Bar{\QQ}_\ell} \bigoplus^t_{s=1}  (\Ql[C_{q_s}]^{D_s} \ominus \Ql )\otimes (\Ql[C_{r}]^{\delta_s}  \ominus \Ql) \oplus (\Ql[C_{r}]^{\delta_s} \ominus \Ql)^{\oplus \gamma_s + \gamma_{s+1} -1}
\]
(note $\gamma_1=0$ and $\gamma_s=1$ for $2\leq s \leq t$).
As $V_{tame}^{ab}$ is a rational representation (see Remark \ref{remark:ab_toric_tame_are_both_rational}), the corresponding $E \otimes _\QQ \QQ_\ell$ representation $V_\ell(J(C))$ of $I_F$, may be realised over $E$, so we may apply Proposition \ref{prop:rationality_andlambda_adicreps} to see that the entries of the matrix representations of $I_F$ on any two $V_\lambda$, $V_{\lambda'}$ with $\lambda, \lambda'|\ell$ differ by applying an element of $\Gal(\Qr/\QQ)$.
Sorting the characters appearing in $\rho_\ell|_{I_F}$ accordingly gives the required result.

The eigenvalues of $\rho_{\lambda^\infty}(\sigma)$ on $T_\lambda$ are the same as those on $V_\lambda = T_\lambda \otimes \Ql$, thus the result for $\rho_\lambda|_{I_F}$ follows by reducing $\rho_{\lambda^\infty}(\sigma)$ modulo $\lambda$ and noting $\rho_{\lambda}|_{I_F}$ is semisimple as $\ell$ does not divide the order of $\rho_\lambda(I_F)$.
\end{proof}

\begin{proposition}
\label{prop_transvection_creation}
Suppose $C/F$ is a $\Delta_v-$regular curve, $d,h$ positive integers, $p \nmid r\ell h$, $\ell\nmid rh$, with Newton polytope $\Dv$ of the form:
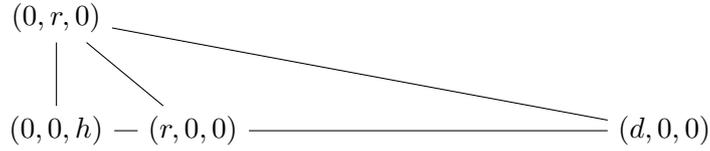
\begin{center}
\begin{tikzpicture}
    \node (0y) at (0,1.5) {$(0,r,0)$};
    \node (2) at (0,0) {$(0,0,h)$};
    \node (3) at (1.8,0) {$(r,0,0)$};
    \node (n) at (8,0) {$(d,0,0)$};

    \draw (0y) to (2) to (3) to (0y) to (n) to (3);

\end{tikzpicture}
\end{center}
Then $\rho_\lambda : I_F \rightarrow \Aut(J(C)[\lambda])$ factors through $I_F^{tame}$. The $r$-th power of a generator $\tau$ of $\rho_\lambda (I_F)$ is a transvection.
Furthermore, if $r \nmid h$ then $\tau$ has exactly $r-2$ non-trivial eigenvalues all of which are primitive $r$-th roots of unity, else all of its eigenvalues are trivial.
\end{proposition}

\begin{proof}
 Let $v'$ be the extension of $v$ to $F^{1/r}\coloneqq F({\pi^{1/r}})$, normalised so that $v'(\pi)= r$.
 The curve $C$ has semistable reduction reduction over $F^{1/r}$ and is $\Delta_{v'}-$regular, as can be seen from the Newton polygon: all points of $\DZ$ have integer value under $v'$.

 Theorem 3.13 of \cite{Models_Curves_DVR} provides us with a (proper, flat) regular model $\mathcal{C}/\OO^{nr}_{F^{1/r}}$ of $C \times (F^{1/r})^{nr}$ with strict normal crossings.
In the notation of this theorem, the inner $v'$-edge $L$ has slopes $s_1^L=h, s_2^L=0$, and the denominator of every $v'$-edge and $v'$-face equals one.
The theorem thus shows the special fibre of $\mathcal{C}$ consists of exactly two irreducible components with positive genus linked by $r$ chains of $\PP^1$'s of length $h-1$. Furthermore, all components have multiplicity one. In particular the special fibre $\mathcal{C}_{\Bar{k}_{v'}}$ is geometrically reduced.
The dual graph of $\mathcal{C}_{\Bar{k}_{v'}}$ is thus given by two vertices joined by $r$ arcs each consisting of $h$ edges. Proposition 9.6.10 of \cite{Neron_models_BLR_book} now applies to show $|\Phi(\Bar{k}_{v'})| = rh^{r-1}$.

Applying Theorem \ref{Tim}, we obtain an isomorphism of tame $I_{F^{1/r}}$-modules $V_\ell(J(C)) \cong \QQ_\ell^{\oplus 2(g-r+1)} \oplus \QQ_\ell^{\oplus r-1} \otimes \mathrm{Sp}(2)$.
Thus, by Proposition \ref{prop:rationality_andlambda_adicreps} applied to an appropriate subgroup of $I_{F^{1/r}}$ (see Remark \ref{remark:Sp_has_rational_restriction}), we see that for each $\lambda$, the group $\rho_{\lambda^\infty}(I_{F^{1/r}})$ contains a transvection.
As $\ell$ does not divide $|\Phi(\Bar{k}_{v'})| = rh^{r-1}$, it follows from \cite[Lemmas 1 \& 2]{Serre_Tate} that ${2g - \dim_{\Fl}J(C)[\ell]^{I_{F^{1/r}}}} = r-1$, thus we see each $\rho_{\lambda}(I_{F^{1/r}})$ also contains a transvection.

Theorem \ref{Tim} applies to show that as tame $I_F$-modules $V_\ell(J(C)) \cong V^{ab}_{tame} \oplus V^{toric}_{tame} \otimes \mathrm{Sp}(2)$ where $V^{ab}_{tame} \cong (\QQ_\ell[C_{r}]^h \ominus \Ql )^{\oplus r-2}$ and $V^{toric}_{tame} \cong  \Ql ^{\oplus r-1}$.
If $r|h$, then the claim follows immediately.
Thus suppose $r \nmid h$.

Recall that the action of $I_F$ on $V_\ell(J(C))$ factors through $I_F^{tame} \cong \prod_{p' \neq p} \ZZ_{p'}$ \cite[Pg. 410]{Neukirch_cohom_of_number_fields}.
Furthermore, the action of $I_F^{tame}$ on $V^{ab}_{tame}$ factors through a group isomorphic to $\ZZ_r$ and its action on $V^{toric}_{tame}\otimes \mathrm{Sp}(2)$ through a group isomorphic to $\ZZ_\ell$.
In particular, we can find a subgroup $I$ of $I_F^{tame}$ which acts faithfully (and rationally) on $V^{ab}_{tame}$, rationally on $V^{toric}_{tame}\otimes \mathrm{Sp}(2)$ and whose image $\rho_{\lambda^\infty}(I)$ in $\rho_{\lambda^\infty}(I_F)$ is dense (see Remarks \ref{remark:ab_toric_tame_are_both_rational}, \ref{remark:Sp_has_rational_restriction}).
Thus Proposition \ref{prop:rationality_andlambda_adicreps} applies to the $(E \otimes \QQ_\ell)[I]$-module $V_\ell(J(C))$, and moreover there exist generators $\tau \in \rho_{\lambda^\infty}(I_F)$ and $\tau' \in \rho_{\lambda^\infty}(I)$ with the same eigenvalues.

It follows that the non-trivial eigenvalues of $\tau$ are primitive $r$-th roots of unity and are $r-2$ in number.
As the eigenvalues of $\tau$ on $T_\lambda$ are equal to those on  $V_\lambda$, we see the reduction modulo $\lambda$ also has exactly $r-2$ non-trivial eigenvalues which are all primitive $r$-th roots of unity.
Observing that $\rho_\lambda(I_F)^r = \rho_\lambda(I_{F^{1/r}})$ completes the proof.
\end{proof}

\begin{remark}
Suppose $C'\colon y^r = g(x)$ defines a $\Dv-$regular superelliptic curve with $g(x)\in F[x]$ monic and degree coprime to $r$. Let $h(x)\in F[x]$ be both coprime to $g$ modulo $\pi$ and separable modulo $\pi$. Let $C \colon y^r = g(x)h(x)$, then $C$ defines a $\Dv-$regular superelliptic curve.
Furthermore, by considering the Newton polytopes of $C$ and $C'$, we see that as $I_F$-modules,  \[H_{\acute{e}t}^1(C_{\overline{F}},\Ql)^{I_{wild}} \cong H_{\acute{e}t}^1(C'_{\overline{F}},\Ql)^{I_{wild}} \oplus \Ql^{\oplus m}\] for some appropriate value of $m$.
The analogous  statement carries through for $V_\lambda(J(C))$ and $V_\lambda(J(C'))$.
\end{remark}

The restriction and reduction of a polynomial, along with their relation to $\Dv-$regularity, are defined on pages 2533 and 2534 of \cite{Models_Curves_DVR}.

\begin{remark}
Suppose the Newton polytope of the superelliptic curve $C \colon y^r=f(x)$ with $f(0)\neq 0$, is of the form required in either Theorem \ref{inertia_action_from_polytope} or Proposition \ref{prop_transvection_creation}. To check $C$ is $\Dv-$regular, it suffices to check $\overline{f|_L}$ is squarefree modulo $\pi$ for each horizontal $v$-edge $L$ of $\Delta$.

Indeed, if $L$ is a non-horizontal $v$-edge of $\D$, then $\overline{(y^r-f)|_L}$ is of the form $ X + a$ or $X^r+a$ where $\pi \nmid a$ and so is square free as $r \neq p$.
Likewise, for a $v$-face $\F$, either $\overline{(y^r-f)|_\F} = Y + G(X)$ or $\overline{(y^r-f)|_\F} = Y^r + G(X)$ where $G(X)$ is some function of $X$. As we work in $\GG_m^2$, taking the partial derivative with respect to $Y$ shows the associated scheme $X_\F$ (see Definition 3.7 \cite{Models_Curves_DVR}) is smooth in both cases.

The point in the Newton polygon corresponding to $y^r$ does not belong to ${\bar L(\ZZ)_{\ZZ}}$ for any horizontal $v$-edge of $\D$. Thus $\overline{(f-y^r)|_L} = \overline{f|_L}$.
\end{remark}

\begin{definition}
We shall say that a polynomial $f\in F[x]$ has \emph{$v$-degree} $(q_1, \ldots, q_t)$ and \emph{height} $(h_1, \ldots, h_t)$, if
\begin{itemize}
    \item $f$ has no repeated roots over $F[x]$;
    \item there exists some $a \in \OO_F$ such that $f(x-a)$ can be factored into monic polynomials $g,h$ over an algebraic closure such that $g(x) \equiv x^{\deg(g)} \mod{\pi}$, $h$ is separable modulo $\pi$, $h(0) \neq 0 \mod{\pi}$, $r|\deg(g) \iff h=1$; and
    \item either $C\colon y^r=f(x-a)$ defines a superelliptic curve satisfying the hypotheses of Proposition \ref{prop_transvection_creation}, or $C\colon y^r=g(x)$ defines a superelliptic curve satisfying Theorem \ref{inertia_action_from_polytope} with $q_1,\ldots, q_t$ and $h_1,\ldots,h_t$ as given.
\end{itemize}
In particular this means $p,\ell,q_1, \ldots ,q_t$ are not equal to $r$, unless $t=1$ and $C$ satisfies Proposition \ref{prop_transvection_creation}.

If $h_{s} =h_{s+1} +1 $ and $ h_t =1$, then we shorten the above and simply say $f \in F[x]$ has $v$-degree $(q_1, \ldots, q_t)$.
\end{definition}

\begin{remark} 
Let $f\in F[x]$ have $v$-degree $(q_1, \ldots, q_t)$ and height $(h_1, \ldots, h_t)$ with $h_s-h_{s+1}$  coprime to $q_s$, where we take $h_{t+1}=0$, as in Theorem \ref{inertia_action_from_polytope}. Suppose that $\frac{h_j-h_{j+1}}{q_j} \neq \frac{h_k-h_{k+1}}{q_k}$ for $j \neq k$.
Then $C\colon y^r=f(x-a)$ is $\Dv-$regular. Indeed, let $f_a(x)\coloneqq f(x-a)$, then for every horizontal line $L$ in the Newton polytope, $\overline{f_a|_L}$ is either linear or $\overline{f_a|_L} = h$ and hence squarefree modulo $\pi$ in either case. 
\end{remark}

\begin{example}
\label{example_one_prime}
Let $p,q,r$ be distinct primes, and $h_1 \in \NN_{>0}$ not divisible by $q$.
Let $f \in F[x]$ have $v$-degree $q$ and height $h_1$.
For example, if $v(p)=1$, we can take $h(x) \in F[x]$ monic, separable modulo $\pi$ with $v\left( h(0) \right) =0$ and  $f(x) = (x^q-p^{h_1})h(x)$.
The above implies $C\colon y^r=f(x)$ is $\Dv-$regular and has Newton polygon
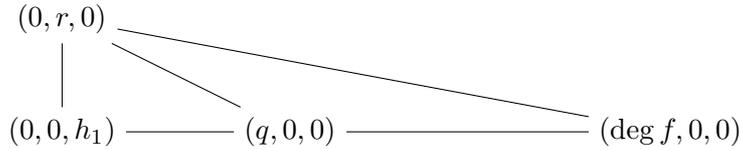
\begin{center}
\begin{tikzpicture}
    \node (0y) at (0,1.5) {$(0,r,0)$};
    \node (2) at (0,0) {$(0,0,h_1)$};
    \node (3) at (3,0) {$(q,0,0)$};
    \node (n) at (8,0) {$(\deg{f},0,0)$};

    \draw (0y) to (2) to (3) to (0y) to (n) to (3);

\end{tikzpicture}
\end{center}
and as $I_F$-modules, \[V_\ell(J(C)) \cong_{\Bar{\QQ}_\ell}\Ql^{\oplus m} \oplus (\QQ_\ell[C_{q}]^{h_1} \ominus \Ql )\otimes (\QQ_\ell[C_{r}]^{h_1}  \ominus \Ql)\]
with wild inertia acting trivially and $m$ some integer divisible by $r-1$. This in turn carries over to give
\[V_\lambda \cong_{\Bar{\QQ}_\ell} \Qr_{\lambda}^{\oplus \frac{m}{r-1}} \oplus \bigoplus_{j=1}^{q-1} \chi_q^{jh_1} \otimes \chi_j^{h_1}\]
where the $\chi_j$ are primitive characters of order $r$ and $\chi_q$ is a primitive character of order $q$.

The representation modulo $\lambda$ then satisfies
\[ (\rho_\lambda)_{I_F} =  \Fli^{\oplus \frac{m}{r-1}} \oplus \bigoplus_{j=1}^{q-1} \chi_q^{jh_1} \otimes \chi_j^{h_1}
\]
where again by abuse of notation, we use the same symbols to denote the reductions of $ \chi_j, \chi_{q}$.

In particular, if $q=2$ and $h_1 =1$, that is, $f$ has $v$-degree $2$, then $\rho_\ell(I_F)$ is generated by a semisimple element $\tau$ with exactly $r-1$ non-trivial eigenvalues which are all distinct and have common order equal to $2r$.

Furthermore, the restriction $\tau_\lambda$ of $\tau$ to $J[\lambda]$ has exactly one non-trivial eigenvalue when $J[\lambda]$  is viewed as a vector space over $\Fli$.
\end{example}

\begin{example}
\label{example_two_primes}
Let $p \neq q_1,q_2$ be primes not equal to $r$, with $q_1 \leq q_2$ and $q_1+q_2 \not \equiv 0 \mod{r}$. Let $f \in F[x]$ have $v$-degree $(q_1,q_2)$ and height $(3,1)$.
For example, if $v(p)=1$, we may take $h(x) \in F[x]$ monic, separable mod $\pi$ with $v\left( h(0) \right) =0$ and then set $f(x) = (x^{q_1}-p)(x^{q_2}-p^2)h(x)$.  The above implies $C\colon y^r=f(x)$ is $\Dv-$regular with Newton polygon
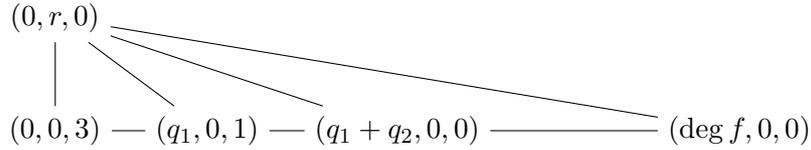
\begin{center}
\begin{tikzpicture}
    \node (0y) at (0,1.5) {$(0,r,0)$};
    \node (2) at (0,0) {$(0,0,3)$};
    \node (3) at (2,0) {$(q_1,0,1)$};
    \node (4) at (4.5,0) {$(q_1+q_2,0,0)$};
    \node (n) at (9,0) {$(\deg{f},0,0)$};

    \draw (0y) to (2) to (3) to (0y) to (4) to (n) to (0y);
    \draw (3) to (4);

\end{tikzpicture}
\end{center}
 and as $I_F$-modules, \[V_\ell(J(C)) \cong_{\Bar{\QQ}_\ell}\Ql^{\oplus m} \oplus (\QQ_\ell[C_{q_1}]^2 \ominus \Ql)\otimes (\QQ_\ell[C_{r}]^{3}  \ominus \Ql) \oplus (\QQ_\ell[C_{q_2}] \ominus \Ql)\otimes (\QQ_\ell[C_{r}]^{q_1+2q_2}  \ominus \Ql)\]
with wild inertia acting trivially and $m$ being some integer divisible by $r-1$. This in turn carries over to give
\[V_\lambda \cong_{\Bar{\QQ}_\ell} \Qr_{\lambda}^{\oplus \frac{m}{r-1}} \oplus \bigoplus_{j=1}^{q_1-1} \chi_{q_1}^{2j} \otimes \chi_j^{3} \oplus \bigoplus_{k=1}^{q_2-1} \chi_{q_2}^{k} \otimes \chi_k^{q_1+2q_2}\]
where the $\chi_j, \chi_k$ are primitive characters of order $r$ and the $\chi_{q_s}$ are primitive characters of order $q_s$.

The representation modulo $\lambda$ then satisfies
\[ (\rho_\lambda)_{I_F} = \Fli^{\oplus \frac{m}{r-1}} \oplus \bigoplus_{j=1}^{q_1-1} \chi_{q_1}^{2j} \otimes \chi_j^{3} \oplus \bigoplus_{k=1}^{q_2-1} \chi_{q_2}^{k} \otimes \chi_k^{q_1+2q_2}
\]
where again by abuse of notation, we use the same symbols to denote the reductions of $\chi_{j},\chi_k, \chi_{q_s}$.
\end{example}

\begin{example}
\label{example_Golbach}
Let $p,q_1,q_2,r$ be distinct primes with $q_1 < q_2$ and $q_1+q_2 \equiv 0 \mod{r}$. Let $f$ have $\pi$-degree $(q_1,q_2)$. For example, if $v(p)=1$, we may take $f(x) = (x^{q_1}-p)(x^{q_2}-p)$. The above implies $C\colon y^r=f(x)$ is $\Dv-$regular and has Newton polygon
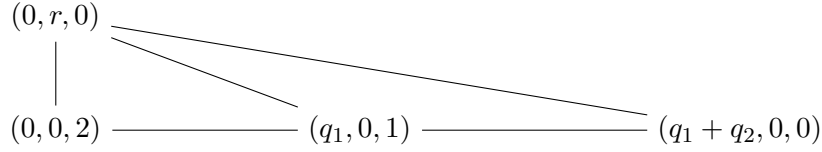
\begin{center}
\begin{tikzpicture}
    \node (0y) at (0,1.5) {$(0,r,0)$};
    \node (2) at (0,0) {$(0,0,2)$};
    \node (3) at (4,0) {$(q_1,0,1)$};
    \node (n) at (9,0) {$(q_1+q_2,0,0)$};

    \draw (0y) to (2) to (3) to (0y)  to (n) to (3);

\end{tikzpicture}
\end{center}
and as $I_F$-modules, \[V_\ell(J(C)) \cong_{\Bar{\QQ}_\ell} (\QQ_\ell[C_{q_1}] \ominus \Ql)\otimes (\QQ_\ell[C_{r}]^{2} \ominus \Ql) \oplus (\QQ_\ell[C_{q_2}] \ominus \Ql)\otimes (\QQ_\ell[C_{r}]^{q_1+2q_2} \ominus \Ql)\]
with wild inertia acting trivially and $m$ some integer divisible by $r-1$. This in turn carries over to give
\[V_\lambda \cong_{\Bar{\QQ}_\ell}  \bigoplus_{j=1}^{q_1-1} \chi_{q_1}^{j} \otimes \chi_j^2 \oplus \bigoplus_{k=1}^{q_2-1} \chi_{q_2}^{k} \otimes \chi_k^{q_1+2q_2}\]
where the $\chi_j, \chi_k$ are primitive characters of order $r$ and the $\chi_{q_s}$ are primitive characters of order $q_s$.

The representation modulo $\lambda$ then satisfies
\[ (\rho_\lambda)_{I_F} =  \bigoplus_{j=1}^{q_1-1} \chi_{q_1}^{j} \otimes \chi_j^2 \oplus \bigoplus_{k=1}^{q_2-1} \chi_{q_2}^{k} \otimes \chi_k^{q_1 + 2q_2}
\]
where again by abuse of notation, we use the same symbols to denote the reductions of $\chi_{j},\chi_k, \chi_{q_s}$.
\end{example}

\begin{example}
\label{example_transvection}
Let $p \neq r$ be primes, and $h_1 \in \NN_{>0}$ not divisible by $r$. Let $f \in F[x]$ have $p$-degree $r$ and height $h_1$. For example, if $v(p)=1$, we can take $h(x) \in F[x]$ monic, separable mod $\pi$ with $v(h(0))=0$ and set $f(x) = (x^r-p^{h_1})h(x)$. The above implies $C\colon y^r=f(x)$ is $\Dv-$regular and has Newton polygon
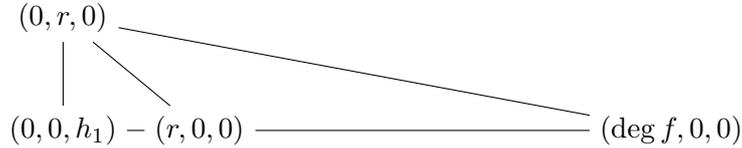
\begin{center}
\begin{tikzpicture}
    \node (0y) at (0,1.5) {$(0,r,0)$};
    \node (2) at (0,0) {$(0,0,h_1)$};
    \node (3) at (1.8,0) {$(r,0,0)$};
    \node (n) at (8,0) {$(\deg{f},0,0)$};

    \draw (0y) to (2) to (3) to (0y) to (n) to (3);

\end{tikzpicture}
\end{center}
and as $I_F$-modules, \[V_\ell(J(C)) \cong_{\Bar{\QQ}_\ell}\Ql^{\oplus m} \oplus (\QQ_\ell[C_{r}]^{h_1}  \ominus \Ql)^{\oplus r-2} \oplus \mathrm{Sp}(2) \otimes \Ql^{\oplus r-1}\]
with wild inertia acting trivially and $m$ some integer divisible by $r-1$. 

Furthermore the action of $\rho_{\lambda^\infty}(I_F)$ and $\rho_{\lambda}(I_F)$ is as given by Proposition \ref{prop_transvection_creation}.
\end{example}

A theorem of Silverberg \cite[Thm 4.1]{Silverberg} shows $\Qr$ must be contained in the $\ell$-torsion field of $J$ for any odd prime $\ell$. As a consequence, the criterion of N\'eron-Ogg-Shafarevich tells us that if $F/\QQ_r$ is a finite extension and $J/F$ has good reduction then $F$ must contain primitive $r$-th roots of unity.
The below lemma not only shows we may find superelliptic curves $J/\QQ_r(\zr)$ with good reduction, but also tells us how to produce them.

\begin{lemma}
\label{lemma_good_reduction_at_r}
Let $F$ be a finite extension of $\QQ_r(\zeta_r)$ and $\pi$ a uniformiser in $\QQ_r(\zeta_r)$. Set  $$f(x) = x^{rs} +a_{rs-1}x^{rs-1}+ \ldots + a_1x+a_0 \in F[x].$$
If either
\begin{enumerate}[(i)]
    \item $a_0 - \left(\frac{1}{\pi}\right)^r \in \OO_F, a_{rs-1} \in \OO^*_F$ and $a_j \in \OO_F$ for $1\leq j \leq rs-2 $, or
    \item $a_0\equiv b \pi^{r(s-1)} \mod{\pi^{rs}}$, where $b \equiv 1 \mod{\pi^r}$, $a_{rs-1} \equiv u\pi\mod{\pi^2}$ with $u \in \OO_F^*$, and $a_{j} \equiv 0\mod{\pi^{rs-j}} $ for $1\leq j \leq rs-2 $,
\end{enumerate}
then $C/F$ has good reduction. In particular $I_F$ acts trivially on $J[\ell]$ for $\ell \neq r$.
\end{lemma}

\begin{proof}
The second assertion follows from the first by taking $x = \pi X$ and $y = \pi^s Y$.
It thus suffices to prove the first assertion, which we shall now do.

The substitution $y \mapsto Y+\frac{1}{\pi}$, provides us with the model
\begin{equation} \label{eq:model_after_Substitution}
\left(Y +\frac{1}{\pi}\right)^r - \left(\frac{1}{\pi}\right)^{r} = f(x) - \left(\frac{1}{\pi}\right)^{r}.
\end{equation}
The leading coefficient on the left hand side is non-zero modulo $\pi$.
Likewise the coefficient of $Y$ is a unit, since by the binomial expansion this equals $\frac{r}{\pi^{r-1}}$ and $\QQ_r(\zeta_r)/ \QQ_r$ is a totally ramified extension of degree $r-1$.
As the binomial coefficient ${r}\choose{i}$ is divisible by $r$ for $1 \leq i \leq r-1$, we see the coefficient of $Y^i$ is zero modulo $\pi$ for $2 \leq i \leq r-1$. 

It follows the left hand side of (\ref{eq:model_after_Substitution}) taken modulo $\pi$ is congruent to $Y^r + uY$ reduced modulo $\pi$ where $u$ is a unit in $\ZZ_r[\zeta_r]$.
Our assumptions on the coefficients of $f(x)$ allow us to reduce our affine model (\ref{eq:model_after_Substitution}) modulo $\pi$.
As the partial derivative with respect to $Y$ is non-zero modulo $\pi$, we find all points on this chart are smooth.

Making the substitution $V = 1/x$, $W = Y/x^s$ in (\ref{eq:model_after_Substitution}) we obtain a chart containing the points at infinity:
$$\left(W+\frac{1}{\pi}V^{s}\right)^r - \left(\frac{1}{\pi}\right)^rV^{rs} = V^{rs}\left(f(1/V)-\left(\frac{1}{\pi}\right)^r\right).$$
Any point at infinity has $V=0$, so the above simplifies to $W^r = 1$ giving $r$ distinct points.
The partial derivative with respect to $V$ at $V=0$ is a unit since the coefficient of $V$ on the right hand side is $a_{rs-1} \in \OO^*_F$. It follows that our curve has good reduction.
Hence, by the criterion of N\'eron-Ogg-Shafarevich $I_F$ acts trivially on $J[\ell]$ for $\ell \neq r$.
\end{proof}
\end{section}

\begin{section}{Image of inertia}
\label{section_image_of_inertia}
Let $\ell$ be a prime, $m$ a positive integer and $r$ a prime. Let $K$ be a number field. For primes $\p$, $\p_j$ in $K$ we let $p$, $p_j$ denote the rational primes below.
Let $V$ denote a finite dimensional vector space over $\Flb$. 

We allow subscripts on representations to denote restrictions to subgroups.

The following definition along with Proposition \ref{irred_blocks_aux_Anni_VladDok} were inspired by \cite[Proposition 3.1]{Anni_VDok}.

\begin{definition}
Let $p,q_1, \ldots, q_t$ be primes distinct from $\ell$. We say $\rho\colon \GK \rightarrow \GL(V)$ has \emph{$\p$-system $(q_1, \ldots, q_t)$} if
 neither $\ell$ nor $p$ divide  $|\rho(I_\p)|$ and
$$\rho_{I_\p} \cong_{\Flb} \psi \oplus \bigoplus_{s=1}^{t} \bigoplus_{j=1}^{q_s-1} \chi_{s,j} \otimes \chi_{q_s}^{j},$$
where the $\chi_{s,j}$ are either trivial characters or have order dividing some $m$ so that the residue field $k_\p$ satisfies $|k_\p| \equiv 1 \mod{m}$, the $\chi_{q_s}$ are primitive characters of order $q_s$ and $\psi$ is some $\Flb$-representation of $I_\p$.
When it is necessary to be explicit about $m$, we will say the $\p$-system is \emph{twisted by $m$}.

If furthermore, the subvector space of $V$ corresponding to $\bigoplus_{j=1}^{q_s-1} \chi_{s,j} \otimes \chi_{q_s}^{j}$ is an irreducible $D_\p$ module for every $s$, then we say the $\p$-system $(q_1, \ldots, q_t)$ is \emph{irreducible}.

We say a $\p$-system is \emph{strict} if $\psi$ is zero or is a direct sum of copies of the trivial representation.

If the restriction may be written in the form \[\rho_{I_\p} \cong_{\Flb} \psi \oplus \bigoplus_{s=1}^{t} \chi_{s} \otimes (\Flb[C_{q_s}] \ominus \Flb)
\]
where the $\chi_{s}$ are characters of order dividing $m$, we say the $\p$-system $(q_1, \ldots, q_t)$ is \emph{tidy}.
\end{definition}

\begin{remark}
By definition of a $\p$-system, $p \nmid \#\rho(I_\p)$ and thus $I_\p$ acts tamely, so $\rho(I_\p)$ is cyclic. 
\end{remark}

\begin{remark}
Suppose $f \in \Qr[x]$ has $\p$-degree $(q_1,\ldots,q_t)$ and $\p$ is a prime of potentially good reduction for the jacobian $J/\Qr$ of the superelliptic curve $y^r=f(x)$.
Then the representation $\rho_\lambda \colon \GQr \rightarrow \Aut(J[\lambda])$ has $\p$-system $(q_1,\ldots,q_t)$ twisted by $r$.
Indeed by Corollary \ref{cor_PGR_inertia_description}, the only condition to check is that $|k_\p|\equiv 1 \mod{r}$ (the $D_s$ in the statement of Corollary \ref{cor_PGR_inertia_description} is always one by definition of such polynomials).
This condition is equivalent to the well-known result that the inertia degree of $\p$ over $p$ is equal to the order of $p$ modulo $r$.
\end{remark}

\begin{remark}
\label{remark_p_degree_and_strict_p-systems}
Let $f$ and $\p$ be as in the previous remark. If $f$ has $\p$-degree $q$ ($\neq r$) and height $h$ coprime to $q$, then by Example \ref{example_one_prime} the $\p$-system associated to $\rho_\lambda$ is strict and twisted by $r$.

Likewise, if $r|\deg f = q_1 +q_2$ where $q_1,q_2$ ($\neq r$) are prime and $f$ has $\p$-degree $(q_1,q_2)$, then by Example \ref{example_Golbach} the $\p$-system associated to $\rho_\lambda$ is strict.
\end{remark}

\begin{subsection}{Irreducibility}
    
In this subsection we relate the decomposition of $V_{D_\p}$ into irreducible subspaces to their corresponding $\p$-systems.
We then use this to give a criterion for irreducibility and will use it also in the next subsection to help prove results about primitivity.

Recall that if $I \vartriangleleft D$ are finite groups and $V$ is a $D$-module, and $W \subseteq V_I$ is an $I$-submodule, then $\varphi W$, where $\varphi \in D$, is also an $I$-submodule.
Moreover, if $\chi$ is the character corresponding to $W$, then the character corresponding to $\varphi W$ is $^\varphi\chi$, where $^\varphi\chi(\tau) \coloneqq \chi(\varphi^{-1}\tau \varphi)$ for $\tau \in I$.
\begin{proposition}
\label{irred_blocks_aux_Anni_VladDok}
Suppose $\rho: \GK \rightarrow \GL(V)$ has a $\p$-system  $(q_1, \ldots, q_t)$ and the size of the residue field $|k_\p|$ is a primitive root modulo each $q_s$.

Then the $\p$-system is irreducible and tidy.
In particular, the socle of $V$, when viewed as a $D_\p$-module, contains $\bigoplus_{s} W_s$, where each $W_s$ is an irreducible $D_\p$-module of dimension $q_s-1$.  

\end{proposition}

\begin{proof}
By assumption the inertia group $I_\p$ acts tamely and so $\rho(I_\p)$ is cyclic. It follows that the action of $D_\p$ factors through the finite group $D = \rho(D_\p) = \langle \tau , \varphi \rangle$, where $\rho(I_\p) = \langle \tau \rangle \vartriangleleft D$ is the inertia subgroup and $\varphi$ is a lift of Frobenius.

Let $\chi_{s,j} \otimes \chi_{q_s}^j$ be one of the characters appearing in $\rho_{I_\p}$.
Using that $\varphi \tau \varphi^{-1} = \tau ^{|k_\p|}$ \cite[7.5.3]{Neukirch_cohom_of_number_fields}, we have
$$^{\varphi^{-1} }\left(\chi_{s,j} \otimes \chi_{q_s}^j\right) = \left(\chi_{s,j} \otimes \chi_{q_s}^j\right)^{|k_\p|} = \chi_{s,j} \otimes \chi_{q_s}^{j|k_\p|} .$$
Set $W_s$ to be the orbit of $\chi_{s,j} \otimes \chi_{q_s}^j$ under $D$. Clearly this is an irreducible $\Flb[D]$-submodule of $V$. Furthermore, its dimension is $q_s-1$ since $|k_\p|$ is a primitive root modulo $q_s$. 
Finally, as the action of $D$ on $\chi_{s,j} \otimes \chi_{q_s}^j$ fixes $\chi_{s,j}$ the $\p$-system is tidy.
\end{proof}

\begin{proposition}
\label{prop_irreducibility}
Suppose there are odd primes $q_1<q_2<q_3$ coprime to $m\ell$ such that $q_3-1 \leq \dim(V) = q_1+q_2 - 2$.
Suppose also there are primes $\p_1,\p_2$ with residue characteristics different from $\ell$ such that $|k_{\p_1}|$ is a primitive root modulo both $q_1,q_2$ and $|k_{\p_2}|$ is a primitive root modulo $q_3$.

Suppose $\rho\colon \GK \rightarrow \GL(V)$ has a (strict) $\p_1$-system $(q_1,q_2)$ and a strict $\p_2$-system $(q_3)$.
Then $V$ is irreducible.
\end{proposition}

\begin{proof}
Proposition \ref{irred_blocks_aux_Anni_VladDok} tells us the socle of $\rho_{D_{\p_1}}$ contains a direct sum of irreducible submodules of dimensions $q_1-1$ and $q_2-1$. Hence either $\rho$ is irreducible or its semisimplification is a direct sum of exactly two modules with dimensions $q_1-1$ and $q_2-1$.

Likewise the socle of $\rho_{D_{\p_2}}$ contains an irreducible submodule of dimension at least $q_3-1$. As $q_2-1 < q_3-1$,  we see $V$ is irreducible.
\end{proof}

\end{subsection}

\begin{subsection}{Primitivity}
\label{subsection_primitivity}
Recall that $V$ denotes a finite dimensional vector space over $\Flb$. 
We shall say that a group $G$ acts imprimitively on $V$ if it preserves a decomposition $V= \bigoplus^k_{j=1} V_j$ with $k>1$ \emph{and} permutes the subspaces $V_j$ transitively.

In this subsection we give criteria to prove that a group does not act imprimitively.
To help the reader navigate themselves through this subsection, let us explain the basic idea of how to deal with this case.

Suppose the absolute Galois group $G_K$ of some field $K$ acts transitively on an imprimitivity decomposition $V= \bigoplus^k_{j=1} V_j$, then we obtain a homomorphism $G_K \rightarrow S_k$ and thus some finite extension $L/K$.
If $L=K$, then $G_K$ acts trivially, but also transitively by assumption, so $k=1$.

If $K$ is known to have no unramified extensions, then one can achieve this by showing the action of $G_K$ on the imprimitivity decomposition is unramified.
See Lemma \ref{lemma_primitivity_p-systems_imply_primtive_when_no_unramified_extensions}.

If $K$ has unramified extensions, then in favourable circumstances, one might be able to obtain information on these unramified extensions.
For example, if it has odd class number, or maybe even obtain information on a maximal unramified extension (if such an extension exists).
See Lemmas \ref{lemma_primitivity_unramified_extns_of_23} and \ref{lemma_primitivity_unramified_extns_of_31}.

On the other hand, the description of the action of the decomposition groups $D_\p$ on $V$ given in Proposition \ref{irred_blocks_aux_Anni_VladDok} can be used to gain information on the cycle types of elements in the image of $G_K \rightarrow S_k$.

Combining these two pieces of information allows one to give criteria for primitivity. See Proposition \ref{prop_primitivity_odd_class_number} and Theorem \ref{thm_primitivity_23_31}.

\begin{lemma}{\cite[Lemma 4.15]{Anni_VDok}}
\label{lemma_primitive_evalues}
Let $V= \bigoplus^k_{j=1} V_j$ be a finite dimensional vector space over $\Flb$. Suppose $\tau \in \GL(V)$ permutes the subspaces $V_j$ cyclically. If the eigenvalues of $\tau^k$ on $V_1$ are $\alpha_1 , \ldots , \alpha_d$ (with multiplicity), then the eigenvalues of $\tau$ on $V$ (with multiplicity) are 
\[\zeta_k^s \alpha_t^{1/k}\]
for $t = 1, \ldots, d$ and $s = 0, \ldots, k-1$.

In particular, if $\tau$ has order $k$, then each $k$-th root of 1 is an eigenvalue of $\tau$ and has multiplicity $d=\dim V_j$.
\end{lemma}

\begin{lemma}
\label{lemma_primitive_transvection}
Let $V= \bigoplus^k_{j=1} V_j$, $k>1$ be a finite dimensional vector space over $\Flb$. Suppose $\tau \in \GL(V)$ satisfies $(\tau -1)^2 =0$ and preserves the above decomposition.
If $\tau$ does not fix the subspaces $V_j$, then $\ell=2$.

In particular, if $G \leq \GL(V)$ both preserves an imprimitivity  decomposition $V= \bigoplus^k_{j=1} V_j$ with $\dim V_j =1$ and contains a non-identity element $\tau$ satisfying $(\tau -1)^2 =0$, then $\ell =2$.
\end{lemma}

\begin{proof}
The first assertion is \cite[Lemma 4.16]{Anni_VDok}. We now prove the second. All the eigenvalues of an element $\tau \in \GL(V)$ satisfying $(\tau -1)^2 =0$ are equal to 1, thus if $\tau$ were to fix the $V_j$, then $\tau$ would be the identity element. Applying the first assertion now proves the result.
\end{proof}

\begin{corollary}
\label{cor_primitivity_for_almost_transvection}
Suppose $\ell>2$. Let $\rho \colon \GK \rightarrow \GL(V)$ be a representation whose image preserves an imprimitivity  decomposition $V= \bigoplus^k_{j=1} V_j$. 

Suppose $\rho(I_\p)$ is generated by an element which has at most $2r-1$ non-trivial eigenvalues all being $r$-th roots of unity, and whose $r$-th power is a transvection.
Then $\rho(I_\p)$ fixes each $V_j$. 
\end{corollary}

\begin{proof}
 Let $\tau$ be a generator of $\rho(I_\p)$. By the preceding lemma $\tau^r$ fixes each $V_j$.
 As $\tau$ has at most $2r-1$ non-trivial eigenvalues, Lemma \ref{lemma_primitive_evalues} implies that either $\tau$ fixes each $V_j$ or each $V_j$ has dimension one.
 The latter possibility is ruled out by Lemma \ref{lemma_primitive_transvection}.
\end{proof}

\begin{proposition}
\label{prop_primitivity_strict_p_system_and_transvection}
Suppose $\ell>2$. Let $\rho \colon \GK \rightarrow \GL(V)$ be a representation whose image both preserves an imprimitivity  decomposition $V= \bigoplus^k_{j=1} V_j$ and contains a transvection. 
Let $q_1, \ldots, q_t$ be distinct primes different from $\ell$ and coprime to $m$.

Suppose $\rho$ has a strict $\p$-system $(q_1, \ldots, q_t)$ twisted by $m$.
Then $\rho(I_\p)$ fixes each $V_j$. 
\end{proposition}

\begin{proof}
 Let $\tau$ be a generator of $\rho(I_\p)$. As $m$ is coprime to $q_1, \ldots,q_s $ which are distinct, no subset of the eigenvalues of $\tau$ is fixed by multiplication by a non-trivial $m$-th root of $1$. Thus by Lemma \ref{lemma_primitive_evalues} there are no cycles of length dividing $m$ in the action of $\tau$ on $\{V_1, \ldots, V_k\}$.

For ease, we may now suppose each $\chi_s$ is trivial. All non-trivial eigenvalues of $\tau$ have prime order and multiplicity 1, thus if $\tau$ does not fix the $V_j$ then by Lemma \ref{lemma_primitive_evalues} the dimension of each $V_j$ is one (take all $\alpha_t=1$). But this possibility is ruled out by Lemma \ref{lemma_primitive_transvection}.
\end{proof}

Recall that we write $n=\frac{2g}{r-1}$ where $g$ is the dimension of our jacobian $J/\Qr$.
We also only require $n \geq 10$ for our main results, but everything in this section holds for $n \geq 3$.

\begin{proposition}
\label{prop_primitivity_l_big_enough_no_worries}
Let $\lambda|\ell \neq r$ be a prime of semistable reduction for $J/\Qr$, with $\ell> \mathrm {max}(\frac{n}{2},3)$.

Let $\lambda'$ be some Galois conjugate of $\lambda$.
If $\rho_{\lambda} \colon \GQr \rightarrow \Aut(J[\lambda])$ preserves an imprimitivity decomposition $J[\lambda] = \bigoplus^{k}_{j=1} V_j$ and contains a transvection, then $\rho_{\lambda}(I_{\lambda'})$ fixes each $V_j$.
\end{proposition}

\begin{proof}
 As $\rho_{\lambda}(\GQr)$ contains a transvection and $\ell>2$, there are at most $\frac{n}{2}$ subspaces in the imprimitivity decomposition. We may view $J[\lambda] \subseteq J[\ell]$ as an $\Fl$-vector space. The imprimitivity decomposition $V_1, \ldots, V_k$ of $J[\lambda]$ (viewed with the structure of an $\Fli$-vector space) induces an imprimitivity decomposition $W_1, \ldots, W_k$ of  $J[\lambda] $ as an $\Fl$-vector space.
 It therefore suffices to show $I_{\lambda'}$ fixes each $W_j$.
  The argument used in \cite[Prop.4.12]{Anni_VDok} now carries over verbatim to achieve this.
\end{proof}

Let $G$ be a group acting on a vector space $V$. Suppose $G$ preserves an imprimitivity decomposition $V = \bigoplus_{j=1}^kV_j$. This corresponds to an action of $G$ on the set $\{V_1, \ldots , V_k\}$, which in turn gives us a natural homomorphism $G\rightarrow S_k$.

\begin{lemma}
\label{lemma_primitivity_p-systems_imply_primtive_when_no_unramified_extensions}
 Let $\ell> \mathrm{max}(\frac{n}{2},3)$ be a prime of semistable reduction for $J/\Qr$.
 
Suppose $\rho_\lambda\colon \GQr \rightarrow \Aut(J[\lambda])$ preserves an imprimitivity decomposition $V=\bigoplus^{k}_{j=1}V_j$, its image contains a transvection, and for any place $\p$ of residue characteristic different to $\ell$ either $J/\Qr$ is semistable at $\p$, $\rho_\lambda$ has a strict $\p$-system, or $\rho_\lambda(I_\p)$ is as in Corollary \ref{cor_primitivity_for_almost_transvection}. Then the induced homomorphism $\theta\colon \GQr \rightarrow S_k$ is unramified.
 
 In particular, if $\Qr$ has no unramified extensions contained in $\Qr(J[\lambda])$, then $k=1$.
\end{lemma}

\begin{proof}
It suffices to show $\rho_\lambda(I_v)$ fixes each $V_j$ for each place $v$ of $\Qr$.
For $v|l$, this is achieved by Proposition \ref{prop_primitivity_l_big_enough_no_worries}.
If $v$ is a place of semistable reduction for $J/\Qr$, then Grothendieck's semistable reduction theorem implies that for any $\tau \in \rho(I_v)$, we have $(\tau-1)^2=1$ \cite[Proposition 3.5]{SGA7} \cite[Page 184, Thm 6]{Neron_models_BLR_book}. By Lemma \ref{lemma_primitive_transvection}, such a $\tau$ stabilises each $V_j$.
Finally, if $\rho_\lambda$ has a strict $v$-system (or $\rho_\lambda(I_v)$ is as in Corollary \ref{cor_primitivity_for_almost_transvection}), then by Proposition \ref{prop_primitivity_strict_p_system_and_transvection} (resp. Corollary \ref{cor_primitivity_for_almost_transvection}), $\rho_\lambda(I_v)$ preserves each $V_j$.
\end{proof}

\begin{proposition}
\label{prop_primitivity_odd_class_number}
Let $K$ be a field with odd class number.
Suppose $\rho\colon \GK \rightarrow \GL(V)$ preserves an imprimitivity decomposition $V=\bigoplus^{k}_{j=1}V_j$ and the induced homomorphism $\theta\colon \GK \rightarrow S_k$ is unramified.
The following hold:
\begin{enumerate}
    \item If $\rho$ has an irreducible $\p$-system $(q_1,q_2)$ where $q_1+q_2-2 = \dim V$, then either $k$ is even or $k=1$.
    \item If $\rho$ has an irreducible $\p$-system $(q)$ where $q-1 = \dim V$, then $k$ is odd.
\end{enumerate}
In particular if $\rho$  has an irreducible $\p_1$-system $(q_1,q_2)$ and an irreducible $\p_2$-system $(q_3)$ where $q_1+q_2- 2 = q_3-1= \dim V$, then $k=1$.
\end{proposition}

\begin{proof}
The action of $D_{\p_s}$ factors through the finite group $D_s = \rho(D_{\p_s}) = \langle \tau_s , \varphi_s \rangle$, where $I_s = \rho(I_{\p_s}) = \langle \tau_s \rangle \vartriangleleft D_s$ is the inertia subgroup and $\varphi_s$ is a lift of Frobenius.
Let $\sigma_s$ be the image of a Frobenius element at $\p_s$ under $\theta$.

We first prove 1. Let us assume $k>1$.
By assumption $V$ is the sum of two irreducible $\Flb[D_{1}]$-modules. As $I_1$ fixes each $V_j$ and there are exactly two irreducible $D_1$-submodules, $\varphi_1$ has exactly two orbits on the $V_j$. This implies $\sigma_1$ is product of two cycles, the sum of whose lengths is $k$.
If $k$ is odd, then exactly one of these cycles has even length. This implies $\sigma_1$ is an odd permutation, and so the image of $\theta$ does not land in $A_k$. This gives rise to an unramified degree 2 extension of $K$, contrary to assumption.

We now prove $2$.
 As $I_2$ fixes each $V_j$, and $V$ is an irreducible $\Flb[D_{2}]$-module, $\varphi_2$ must permute the $V_j$ transitively.
Consequently the permutation $\sigma_2$ is a $k$-cycle. 
In particular $\sigma_2$ is an odd permutation if $k$ is even.
This again gives us an unramified degree 2 extension of $K$ contrary to assumption.
\end{proof}

\begin{lemma}
\label{lemma_primitivity_unramified_case}
Suppose $\rho\colon \GK \rightarrow \GL(V)$ preserves an imprimitivity decomposition $V=\bigoplus^{k}_{j=1}V_j$ and has an irreducible $\p$-system $(q)$, where $q-1> \frac{1}{2}\dim V$.

Suppose the induced homomorphism $\theta\colon \GK \rightarrow S_k$ is unramified and let $\sigma$ be the image of a Frobenius element at $\p$. Let $L/K$ be the extension cut out by $\theta$.

The following hold:
\begin{itemize}
    \item if $\sigma$ has order $2$, then $k=2$ or $3$ and $K$ has an unramified (Galois) extension of degree 2 inert at $\p$;
    \item if $\sigma$ has order $3$, $L/K$ is soluble and $K$ has odd class number then $k=3$ or $4$ and $K$ has an unramified Galois extension of degree $3$ inert at $\p$.
\end{itemize}
\end{lemma}

\begin{proof}
By assumption $V$ is an $\Flb[D_\p]$-module with an irreducible submodule of dimension $q-1$.
Moreover, the action of $D_\p$ factors through the finite group $D = \rho(D_\p) = \langle \tau , \varphi \rangle$, where $\rho(I_\p) = \langle \tau \rangle \vartriangleleft D$ is the inertia subgroup and $\varphi$ is a lift of Frobenius.
Let $a \in \{2,3\}$ denote the order of $\sigma$.

As the orbits of the $V_j$ under $\varphi$ are $D$-modules, one of these orbits must contain the forementioned irreducible module of dimension $q-1>\frac{1}{2}\dim V$.
Since the $V_j$ have equal dimension, $\sigma$ has a cycle of length greater than $k/2$.
Having prime order, the cycles in the decomposition of $\sigma$ all have the same length. 
This implies $\sigma$ is an $a$-cycle and $k<2a$.

In the case $a=2$, we see $\Gal(L/K)$ is isomorphic to a subgroup of $S_3$. As $S_3$ has a normal subgroup of order $3$ with quotient isomorphic to $C_2$, the statement follows for $a=2$.

Suppose we are in the second case, i.e., $a=3$.
As $L/K$ is a soluble extension and $2$ does not divide the class number of $K$, we may rule out $k=5$ since $A_5$ does not have a proper subgroup which is both transitive on $5$ points and contains an element of order $3$.
We deduce $k=3$ or $4$ and $\Gal(L/K)$ is isomorphic to a subgroup of $A_4$ (again using $K$ has odd class number).
As $A_4$ has a normal subgroup of order $4$ with quotient isomorphic to $C_3$, the final statement follows.
\end{proof}

\begin{lemma}
\label{lemma_primitivity_unramified_extns_of_23}
Let $L/\QQ(\zeta_{23})$ be a non-trivial unramified Galois extension. Then either $\Gal(L/\QQ(\zeta_{23})) \cong C_3$ or $ A_4$.
\end{lemma}

\begin{proof}
The root discriminant of $\QQ(\zeta_{23})$ is equal to $19.94$ to two decimal places.
As $L/\QQ(\zeta_{23})$ is an unramified extension, it has equal root discriminant to $\QQ(\zeta_{23})$ \cite[Lemma 11.22]{washington}.
 Table 1 in \cite{DiazyDiazTables} shows any totally imaginary number field of absolute degree greater than or equal $462 = 21\times 22$ has root discriminant at least 19.98.
 This shows us such an $L$ satisfies $[L:\QQ(\zeta_{23})] < 21$.

The class number of $\QQ(\zeta_{23})$ is 3 \cite[Tables, \S 3]{washington}. Thus any abelian quotient of $\Gal(L/\QQ(\zeta_{23}))$ must have order dividing 3.
This with the above bound implies $\Gal(L/\QQ(\zeta_{23})) \cong C_3$ or $A_4$.
\end{proof}

\begin{lemma}
\label{lemma_primitivity_unramified_extns_of_31}
Let $L/\QQ(\zeta_{31})$ be a non-trivial unramified Galois extension. Then assuming GRH either $\Gal(L/\QQ(\zeta_{31})) \cong C_3$, $C_3 \times C_3$, or $C_9$.
\end{lemma}

\begin{proof}
The root discriminant of $\QQ(\zeta_{31})$ is equal to $27.65$ to two decimal places.
As $L/\QQ(\zeta_{31})$ is an unramified extension, it has equal root discriminant to $\QQ(\zeta_{31})$ \cite[Lemma 11.22]{washington}.
 The bounds given in Table 1 of \cite{Odlyzko_bounds} show, upon assumption of GRH, any totally imaginary number field of absolute degree greater than or equal $720 = 24\times 30$ has root discriminant at least 27.98.
 This shows us such an $L$ satisfies $[L:\QQ(\zeta_{31})] < 24$.

The class number of $\QQ(\zeta_{31})$ is 9 \cite[Tables, \S 3]{washington}. Thus any abelian quotient of $\Gal(L/\QQ(\zeta_{31}))$ must have order dividing 9.
This with the above bound implies $\Gal(L/\QQ(\zeta_{31})) \cong C_3$, $C_3\times C_3$, $C_9$, $A_4$ or $C_7 \rtimes C_3$.

The compositum of two unramified Galois extensions is an unramified Galois extension.
Thus if $\Gal(L/\QQ(\zeta_{31})) \cong A_4$ (resp. $C_7 \rtimes C_3$) then there would be an unramified Galois extension of $\QQ(\zeta_{31})$ of degree $36$ (resp. $63$). But these contradict the above bound.
We deduce $\Gal(L/\QQ(\zeta_{31})) \cong C_3$, $C_3\times C_3$, or $C_9$.
\end{proof}

The following lemma shows we may always find a prime $q$ fulfilling the hypothesis of Theorem \ref{thm_primitivity_23_31}.
\begin{lemma}
\label{lemma_prim_bertrand_postulate}
For $m \geq 6$ there is a prime $m+1< q < 2m$ congruent to $2$ modulo $3$.
\end{lemma}

\begin{proof}
For $m \geq 21$ this follows from section 4 of \cite{Moree_Bertrands_postulate_primes_in_Arithmetic_progressions}.
The remaining values of $m$ can easily be checked by hand.
\end{proof}

\begin{theorem}
\label{thm_primitivity_23_31}
Let $ r\in \{23,31\}$. If $r=31$ assume GRH.
Suppose $\rho\colon \Gcyclo{r} \rightarrow \GL(V)$ preserves an imprimitivity decomposition $V=\bigoplus^{k}_{j=1}V_j$ and the induced homomorphism $\theta\colon \Gcyclo{r} \rightarrow S_k$ is unramified.
Let $L/\QQ(\zeta_{r})$ be the extension cut out by $\theta$.

Let $p,q$ be primes distinct from $r$ and $\ell$, with $|k_\p|$ a primitive root modulo $q$.
If $\rho$ has a  $\p$-system $(q)$, where $q-1> \frac{1}{2}\dim V$ and $q \equiv 2 \mod{3}$, then $k=1$. 
\end{theorem}

\begin{proof}
If $L=\QQ(\zeta_{r})$, then we are done, so suppose not.
By Proposition \ref{irred_blocks_aux_Anni_VladDok}, the $\p$-system $(q)$ is irreducible, and in particular, $V$ has an irreducible $D=\rho(D_\p)$ submodule of dimension $q-1> \frac{1}{2} \dim V$.

By Lemmas \ref{lemma_primitivity_unramified_extns_of_23} and \ref{lemma_primitivity_unramified_extns_of_31}, either $\Gal(L/\QQ(\zeta_r)) \cong C_3, C_3 \times C_3, C_9$ or $A_4$.
It follows from Lemma \ref{lemma_primitivity_unramified_case} that $\sigma$, the image of the Frobenius element at $\p$, has order dividing $9$. Furthermore, if $\sigma$ has order 3 then $k=3$ or $4$, and if $\sigma$ has order 9 then $\Gal(L/\QQ(\zeta_r)) \cong C_9$ giving $k=9$ since $\Gal(L/\QQ(\zeta_r))$ acts transitively.

Let us rule out the case where $\sigma$ is trivial. Here, every $V_j$ is a $D$-module, and some $V_j$ must have an irreducible constituent of dimension at least $q-1$, forcing $k=1$. Hence we may suppose the order of $\sigma$ is 3 or 9.

Let $\varphi \in D$ be a lift of Frobenius and $\tau$ be a generator of the tame inertia group $\rho(I_\p)$, so that $D = \langle \tau, \varphi \rangle$.
Let $U_1, \ldots,U_t$ be the orbits of $\varphi^3$ on the $V_j$.
Then each $U_j$ is a $\langle \tau, \varphi^3 \rangle$-module, and as $3$ divides the order of $\sigma$, either $t=3$ or $4$ by the above. The dimension of each $U_j$ is less than or equal to $\frac{1}{3} \dim V$.
We deduce $|k_\p|^3$ is not a primitive root modulo $q$, else following the proof of Proposition \ref{irred_blocks_aux_Anni_VladDok}, we would be able to produce an irreducible $\langle \tau, \varphi^3 \rangle$-submodule of $V$ with dimension $q-1> \frac{1}{2} \dim V$.
It follows that the cubing map does not induce an automorphism on $(\ZZ/q\ZZ)^*$.
Thus $q\equiv 1 \mod{3}$, completing the proof.
\end{proof}
\end{subsection}

\begin{subsection}{Subfield subgroups and classical subgroups}
    \label{ss:ruling_out_subfield_sbgps_and_classical_sbgps}

The following lemma will come in handy when ruling out the possible containment of $\Glam = \rho_\lambda(\GQrl)$ in certain maximal subgroups.

We note the only importance of the conditions of this lemma is that there exists an element in $G_\lambda$ which has exactly one eigenvalue of order $2r$ and the rest equal to $1$. 

\begin{lemma}
\label{lemma_not_subfld_sbgp_not_in_GSp_n}
Suppose $f(x) \in \Qr[x]$ has $\p$-degree 2 for some prime $\p$ with residue characteristic distinct from $ r,\ell$.
Then \begin{itemize}
    \item the image of $\det\colon G_\lambda \rightarrow \Fli^*$ is not contained in a proper subfield of $\Fli$;
    \item $G_\lambda$ is not contained in a subfield subgroup;
    \item $G_\lambda$ does not preserve a symplectic pairing up to scalars.
\end{itemize}
\end{lemma}
\begin{proof}
By Example \ref{example_one_prime},  $I_\p$ acts tamely on $J[\lambda]$ and furthermore,
\[(\rho_\lambda)_{I_\p} = (\varepsilon \otimes \chi_r ) \oplus \Fli^{\oplus n-1}\] where $\varepsilon, \chi_r$ are characters of orders $2$ and $r$ respectively. Note also that as $p \neq r,\ell$ we have the containment $\rho_\lambda(I_\p) \leq \Glam$.

Let $\tau$ be a generator of $\rho_\lambda(I_\p)$. Then $\det ( \tau^2) = \chi_r(\tau^2)$ which is a primitive $r$-th root. But $\Fli$ is the smallest extension of $\Fl$ containing a primitive $r$-th root, whence the first statement.
As $\tau^2$ is not a scalar, the second statement follows.

We now show the third statement. Suppose $\tau^r$ is similar to an element of $\GSp_n(\Fli)$.
As the similitude character $\chi_\sharp : \GSp_n(\Fli) \rightarrow \Fli^*$ is a homomorphism and $\tau^r$ has order two, we find $\chi_\sharp (\tau^r)=\pm 1$.
Using \cite[Lemma 2.4.5]{KleidmanLiebeck}, we evaluate $\chi_\sharp (\tau^r)^{n/2}= \det( \tau^r)=-1$. Thus $\chi_\sharp (\tau^r)=-1$.

The only eigenvalues of $\tau^r$ are plus and minus one, thus Lemma \ref{lemma_symplectic_characteristic_poly} combined with $\chi_\sharp (\tau^r)=-1$ implies they are equal in number. Since $n>2$, this contradicts the description of $(\rho_\lambda)_{I_\p}$ given above.
\end{proof}

\end{subsection}
\end{section}

\begin{section}{The endomorphism character}
\label{section_endo_char}
\begin{subsection}{Algebraic Hecke characters}
\label{subsection_Alg_Heck_chars}
Here we collect a few facts about algebraic Hecke characters which we shall need later in this section. 
Everything in this subsection is well-known and one may consult \cite{Periods_of_Hecke_characters}, \cite{Chai_Conrad_Oort}, or \cite{Serre72} for more details.
The reader should, however, be aware that none of these references contain all the details we require and the conventions and (mathematical) language changes in each of them.

Let $E$ be a number field and fix an algebraic closure $\Bar{E}$. Let $K$ be a number field in $\Bar{E}$ containing all conjugates of $E$. Let $\lambda |\ell$ be a prime above $\ell$ in $E$.
 Choose a valuation $v_\lambda$ of $\Bar{E}$ which extends the $\lambda$-adic valuation of $E$. 
 The completion $\Bar{E}_\lambda$ of $\Bar{E}$ with respect to $v_\lambda$ is then an algebraic closure of $E_\lambda$.
 Let $k_\lambda$ denote the residue field of $\Bar{E}_\lambda$.

Let $\Gamma$ denote the set of embeddings  $K \hookrightarrow \Bar{E}$. Every element $\sigma$ of $\Gamma$ extends by linearity to a homomorphism $K_\ell^* = (K\otimes \QQ_\ell)^* \rightarrow \Bar{E}_\lambda$, by abuse of notation we denote this homomorphism again by $\sigma$. This extension is trivial on all but one of the $K_v^*$'s in the decomposition $K_\ell^* = \prod_{v|l} K_v^*$, to be precise, the one corresponding to $v_\lambda \circ \sigma$. We denote by $\Gamma(v)$ the elements $\sigma \in \Gamma$ such that $v_\lambda \circ \sigma$ is equivalent to $v$. By the above we have for any embedding $\sigma \in \Gamma(v)$ a canonical embedding $\sigma_v:K^*_v \hookrightarrow \Bar{E}_\lambda$.

We write $\mathbb{A}_K^\times$ for the ideles of $K$ and identify $K^*$ with its diagonal embedding in $\mathbb{A}_K^\times$ and likewise $K_\ell^*$ with its image in $\mathbb{A}_K^\times$.

A Hecke character with values in $E$ is a continuous homomorphism $\chi: \mathbb{A}_K^\times \rightarrow E^*$ trivial on $K^*$. All such characters may be written as $\chi = \chi^\infty\chi_\infty^{-1}$, a product of its restriction to the finite places $\chi^\infty:  \prod_{v\nmid \infty} K^*_v \rightarrow E^*$ and its restriction to the infinite places $\chi_\infty^{-1}: K^*_\infty = \prod_{v| \infty} K^*_v \rightarrow E^*$.

An algebraic Hecke character with values in $E$ is a Hecke character $\chi$ taking values in $E$ such that the restriction of $\chi$ to $K_\infty^{*, 0}$, the connected component of 1 in $K_\infty^*$, is given by an algebraic homomorphism. That is, there exists $T = \sum_{\sigma \in \Gamma} n_\sigma \sigma $ such that $\chi_\infty = T$ on $K_\infty^{*, 0}$. We call $\sum_\sigma n_\sigma \sigma$ the \emph{infinity type} of $\chi$.

Let us now construct the $\lambda$-adic avatar of $\chi$.
Let $\alpha \in \mathbb{A}_K^\times$, write $\alpha_\infty$ for the components of $\alpha$ in $K^*_\infty$.
Define $\Tilde{\chi}: \mathbb{A}_K^\times \rightarrow E^*$ by $ \Tilde{\chi}(\alpha) = \chi(\alpha) T(\alpha_\infty) $. Note $\Tilde{\chi}|_{K^*} = T$ since ${\chi}|_{K^*}$ is trivial. 
As elements of $\Gamma$ may be extended to homomorphisms $K_\ell^* \rightarrow \Bar{E}_\lambda^*$, the infinity type $T$ may also be extended to a homomorphism $T_\lambda:K_\ell^* = (K\otimes \QQ_\ell)^* \rightarrow E_\lambda^*$.
Let $\alpha_\ell = (\alpha_v)_{v|\ell} \in K_\ell^* = \prod_{v|\ell} K_v^*$.
The character $\chi_\lambda:\mathbb{A}_K^\times \rightarrow E_\lambda^*$ defined by $$ \chi_\lambda(\alpha) = \Tilde{\chi}(\alpha)T_\lambda(\alpha_\ell^{-1})$$
is continuous and trivial on $K^*$ (here we view $\Tilde{\chi}(\alpha)\in E^* \hookrightarrow E^*_\lambda$ under the obvious embedding).
We call $\chi_\lambda$ the \emph{$\lambda$-adic avatar} of $\chi$.

Clearly, $\chi_\lambda$ may be viewed as a character on $C_K = \mathbb{A}_K^\times/K^*$, the idele class group of $K$, furthermore as $E_\lambda$ is totally disconnected, the connected component $C_K^0$ of $C_K$ is sent to 1 under $\chi_\lambda$. Class field theory allows us to view $\chi_\lambda$ as a homomorphism $G_K^{ab} \cong C_K/C^0_K \rightarrow E_\lambda^*$ via the Artin map. In particular for all but finitely many $\p$, the image of the arithmetic Frobenius is $\chi_\lambda(\Frob_\p) = \Tilde{\chi}(\pi_\p)T_\lambda(1) = \chi(\pi_\p)$, where $\pi_\p \in \mathbb{A}_K^\times$ is a uniformiser at $\p$ and 1 in every other component (here equality is viewed under the embedding $E^* \hookrightarrow E^*_\lambda$).

The profinite completion $\hat{\OO}_K$ of $\OO_K$, the ring of integers of $K$, may be decomposed as a product $\prod_{v \nmid \infty}\OO_v$, where $\OO_v$ is the subring of elements in $K_v$ having non-negative valuation.
Owing to the fact that $\chi$ is continuous and $\hat{\OO}_K^* \subseteq \mathbb{A}_K^\times$ is profinite, $\chi(\OO_v^*)=1$ for all but finitely many $v$. When $\chi(\OO_v^*)\neq1$, there is some subgroup $1+\p_v^{m_v} \subseteq \OO_v^*$ whose image under $\chi$ is trivial.
We say $\chi$ is unramified at $v$ if $\chi(\OO_v^*)=1$, and ramified otherwise.
Let $\mathfrak{m} =\prod \p_v^{m_v} $ where the product is taken over the ramified primes of $\chi$.

For $\p\nmid \mathfrak{m}$ define $\chi(\p)= \chi(\pi_\p)$, where $\pi_\p \in \mathbb{A}_K^\times$ is a uniformiser at $\p$ and $1$'s elsewhere.
Note this is independent of the choice of uniformiser as $\chi$ is unramified at $\p$. This allows us to extend $\chi^\infty$ to a group homomorphism from ideals of $K$ coprime to $\mathfrak{m}$ to $\Bar{\QQ}^*$.
Let $(\alpha)$ be an ideal of $K$ coprime to $\m$, and suppose $\alpha$ is totally positive, then
$$\chi^\infty((\alpha)) = \chi_{\infty}(\alpha) = T(\alpha) = \prod_\sigma \sigma(\alpha)^{n_{\sigma}} $$
as $\chi$ is trivial on $K^*$ and $\alpha_\infty \in K_\infty^{*, 0}$. We may extend $T$ to a map of ideals satisfying $$(\chi(\p)) = T(\p) .$$
The above relation gives us a method to determine $T$, the infinity type of $\chi$, by factorising $(\chi(\p))$.

In the introduction we briefly mentioned that knowledge of the infinity type helps determine the mod $\ell$ image of Galois. Let us indicate how here.
Continue to let $\m$ be as above. We define $U_{v}$ as the connected component of $K^*_v$ if $v|\infty$, and otherwise as $\OO^*_{v}$ if $v \nmid \m$ and  the subgroup $1 + \p_v^{m_v} \subseteq \OO^*_v$ if $v| \m$. Finally, we write $U_\m = \prod_v U_{v}$.
By construction of $U_\m$, we have $\Tilde{\chi}(\alpha)=1$ for all $\alpha \in U_\m$, which leads to the equality 
$$\chi_\lambda(\alpha) = T_\lambda(\alpha_\ell^{-1})=\prod_{\sigma}\sigma_v(\alpha_\ell^{-1})^{n_\sigma} \: \text{ for } \alpha \in U_\m. $$
The image of $\chi_\lambda$ is compact, and thus lands in $\OO_{E_\lambda}^* $, the maximal compact subgroup of $E_\lambda^*$. This allows to reduce modulo $\lambda$.
Local Class Field Theory associates $\OO_v^*$ to the inertia group of the maximal abelian extension of $K_v^*$ and $1+\p_v$ to its wild inertia subgroup.
The multiplicative subgroup $k_v^*$ of the residue field  of $K_v$ thus corresponds to the tame inertia group. Passing to the residue field, the $\sigma \in \Gamma(v)$ induce embeddings $k_v \rightarrow k_\lambda$, which we denote by $\Bar{\sigma}$. For $v \nmid \m$, the above provides us with following description of $\Bar{\chi}_\lambda$ on the inertia groups
\begin{equation*}
\Bar{\chi}_\lambda(\alpha) = 
\begin{cases} 
1 & \text{ for } \alpha \in U_v, v\nmid \ell,\\ 
\prod_{\sigma \in \Gamma(v)} \Bar{\sigma}(\Bar{\alpha}_\ell^{-1})^{n_\sigma}  & \text{ for } \alpha \in U_v, v | \ell
\end{cases}
\end{equation*}
where $\Bar{\alpha}_\ell$ denotes the reduction of $\alpha_\ell$ modulo $\p_v$.

Using Class Field Theory, the characters $\Bar{\alpha}_\ell \mapsto \Bar{\sigma}(\Bar{\alpha}_\ell^{-1})$ may be viewed as characters on the inertia group $I_v$. From this viewpoint they turn out to be fundamental characters of level $[k_v:\Fl]$, see \cite[Prop 3]{Serre72}.

One may rewrite the above equation by letting $\theta(\Bar{\alpha}_\ell) = \Bar{\sigma}(\Bar{\alpha}_\ell^{-1})$ for some $\sigma \in \Gamma(v)$ and then for any $\tau \in \Gamma(v)$ writing $$\theta(\Bar{\alpha}_\ell)^{\ell^{m_\tau}} = \Bar{\tau}(\Bar{\alpha}_\ell^{-1})$$ for some appropriate integer $m_\tau$. This provides us with the following description (for $v \nmid \m$):
\begin{equation}
 \label{eqn_endo_Char_inertia}
\Bar{\chi}_\lambda(\alpha) = 
\begin{cases} 
1 & \text{ for } \alpha \in U_v, v\nmid \ell,\\ 
\prod_{\sigma \in \Gamma(v)} \theta(\Bar{\alpha}_\ell)^{n_\sigma \ell^{m_\sigma}}  & \text{ for } \alpha \in U_v, v | \ell
\end{cases}
\end{equation}
where $\theta$ is some fundamental character of level $[k_v:\Fl]$.

A fundamental character of level $[k_v:\Fl]$ differs by an automorphism of $k_v$ from any other fundamental character of the same level.
As all the automorphisms of $k_v$ are given by raising its elements to some power of $p$, we see that the image of $U_v$ is independent of $\theta$.
\end{subsection}
\begin{subsection}{The endomorphism character}
\label{subsection_endo_char}
In this subsection let us consider a $g$-dimensional abelian variety $A/K$. Suppose we have an embedding $E \hookrightarrow \End^0_K(A)$.
Then, as in Section \ref{section_lambda_adic_reps}, we may attach a system of $\lambda$-adic representations to $A$. 
These representations form a strictly compatible system $(\rho_\lambda)$ with exceptional set $S$ equal to the set of places of $K$ where $A$ has bad reduction, see Section II of \cite{ribet_RM} for more details.

As explained in Section \ref{section_lambda_adic_reps}, taking the determinant of these representations leads to the existence of an algebraic Hecke character $\Omega$ such that $\Omega_\lambda = \det \circ \rho_{\lambda^\infty}$ for each $\lambda$.
We call $\Omega$ the \emph{endomorphism character}.

Fité has recently determined the infinity types of the $\Omega_\lambda$ up to a suitable equivalence when $K/E$ and $E/\QQ$ are Galois extensions, that is, viewing $\Omega_\lambda$ as a map on ideals, he determines the integers $n_\sigma$ appearing in the factorisation 
\[ \Omega_\lambda((\alpha)) = \prod_{\sigma:K \hookrightarrow \Bar{E}} \sigma(\alpha)^{n_\sigma}
\]
up to a certain equivalence, where $\alpha \in K$. Note that for this product to belong to $E$, we must be able to rewrite the above as
\[ \Omega_\lambda((\alpha)) = \prod_{\tau: E \hookrightarrow \Bar{E}} \tau(N^K_E(\alpha))^{n_\tau}
\]
for some integers $n_\tau$.
We now describe the aforementioned  equivalence.
The Galois group $\Gal(E/\QQ)$ acts on the tuple  $(n_\tau)_\tau$ by $h \cdot (n_\tau)_\tau = (n_{\tau \circ h^{-1}})_\tau$ for $h \in \Gal(E/\QQ)$.
Two such tuples are then called equivalent if they belong to the same orbit under this action.

The endomorphism algebra $\End_K^0(A)$ acts on the regular differentials of $A$, and by restriction so does $E$.
Enlarging $K$ by a finite extension, we may assume there exists a basis $\omega_1,\ldots, \omega_g $ of $\Omega^1(A)$ satisfying
$$\alpha\cdot\omega_i = \psi_i(\alpha)\omega_i$$
for all $i$ and $\alpha \in E$, where the $\psi_i$ are embeddings of $E$ into $\Bar{E}$. Let $m_\tau$ denote the multiplicity of $\tau$ in $\{\psi_1, \ldots \psi_g \}$.  

\begin{theorem}{\cite[Prop. 14]{Fite_Ordinary_primes}}
\label{thm_fite_inf_type_up_to_equiv}
The tuple $(n_{\tau^{-1}})_\tau$ is equivalent to $(m_{\Bar{\tau}})_\tau$ the complex conjugate of the tuple given by the action of $E$ on $\Omega^1(A)$.
\end{theorem}

The above coupled with the content in \S \ref{subsection_Alg_Heck_chars} provides us with valuable information about a given $\Omega_\lambda$.
Namely, it gives a description of the images of Frobenius elements outside the places belonging to $S \cup S_\ell$, ($S_\ell$ being the set of primes above $\ell$ in $K$) as well as the image of inertia at primes above  $\ell$, where $\lambda|\ell$.
The following proposition places restrictions on $\Omega_\lambda(I_\p)$ for $\p \in S$, $\p \nmid \ell$.

\begin{proposition}
\label{prop_endo_car_inertia_at_p_ne_l}
Let $\p \nmid \ell $ be a prime of $ \OO_K$. Then $\Omega_\lam(I_\p)$ is contained in the subgroup of $E^*$ generated by roots of unity.
Furthermore, if $\p$ is a prime of semistable reduction then $\Omega_\lambda(I_\p)=1$.
\end{proposition}

\begin{proof}
If $\p$ is a prime of semistable reduction then \cite[Proposition 3.5]{SGA7} \cite[Page 184, Thm 6]{Neron_models_BLR_book} implies the inertia group $I_\p$ acts unipotently and therefore $\Omega_\lambda(I_\p) = \det \circ \rho_{\lam^\infty}(I_\p)=1$.

By the semistable reduction theorem \cite{SGA7} any prime $\p$ is a prime of potential semistable reduction for $A/K$. Let $\tau\in I_\p$. Then, as $A$ acquires semistable reduction over some finite extension, there exists $n \in \NN$ such that $\det \circ \rho_{\lam^\infty}(\tau^n)=1$. Since $\det \circ \rho_{\lam^\infty}: I_\p \rightarrow E^*$ is a homomorphism, we obtain $\Omega_\lam (\tau) =\det \circ \rho_{\lam^\infty}(\tau)$ is a root of unity in $E$.
\end{proof}
\end{subsection}
\begin{subsection}{Image of the endomorphism character}
Let us resume our study of superelliptic jacobians. 
We decompose
\[\Omega^1(J) = \bigoplus_{j=1}^{r-1} \Omega^1_j(J)\] into the eigenspaces of $[\zr]^*$ where we write $\Omega^1_j(J)$ for the eigenspace of $\zr^j$.
In the following we write $\lfloor x \rfloor_< $ for the greatest integer strictly less than $x$.

\begin{lemma}
    The dimension of $\Omega^1_j(J)$ equals $\lfloor \frac{(r-j)d}{r} \rfloor_< $ where $d$ is the degree of $f$.
\end{lemma}

\begin{proof}
See Remarks 4.2, 4.4 and 4.5 of \cite{Zarhin} (note that although Zarhin works over an algebraic closure of the ground field, his calculation suffices as differentials behave well with with respect to base change \cite[Lemma 5.2.26]{Liu}, alternatively one can work directly over the base field using Baker's Theorem, see for example \cite[Theorem 2.2]{Models_Curves_DVR}).
\end{proof}

From the discussion in \S \ref{subsection_endo_char}, there is a unique infinity type associated to $\Omega$, which is determined up to equivalence by Theorem \ref{thm_fite_inf_type_up_to_equiv}.
As we are working with the $\lambda$-adic avatars of $\Omega$, it is not important to determine the exact infinity type of $\Omega$.
Instead, for convenience we shall fix a representative of the equivalence class and by abuse of language refer to it as the infinity type of $\Omega$.
Thus, let us say
\[ \sum_{j=1}^{r-1} \psi_j^{-1}{\left\lfloor \frac{(r-j)d}{r} \right\rfloor_<}
\]
is the infinity type of $\Omega$, where $\psi_j^{}:\Qr \hookrightarrow \Bar{\QQ}$ is the embedding sending $\zr$ to $\zr^{j}$. We denote the reduction of $\Omega_{\lambda}$ modulo $\lambda$ by $\Bar{\Omega}_{\lambda}$.

Suppose $\Omega$ is unramified at the finite place $v $ of $\Qr$. By Proposition \ref{prop_endo_car_inertia_at_p_ne_l} this is the case if $J/\Qr$ has semistable reduction at $v$.
If $v \nmid \ell$, then the restriction of $\Bar{\Omega}_\lambda$ to the inertia group $I_v$ is trivial. In the case $v \mid \ell$ the infinity type describes the restriction of $\Bar{\Omega}_\lambda$ to $I_v$, see the discussion at the end of \S\ref{subsection_Alg_Heck_chars}.
Let us now give this description.

Let $\psi: \Qr \to \Bar{\QQ}$ be an embedding which sends $v$ to $\lambda$. This embedding $\psi$ induces a homomorphism $\theta: k_v \to k_\lambda$ which via Class Field Theory may be viewed as a fundamental character of level equal to $[k_\lambda: \Fl]$.
Any embedding $\psi_j:\Qr \to \Bar{\QQ}$ which sends $v$ to $\lambda$ may be written as $\psi_{\ell}^{m(\theta,j)} \circ \psi$ for some integer $m(\theta,j)$. The fundamental character induced by $\psi_j$ then equals $\theta^{\ell^{m(\theta,j)}}$. We obtain 
\[
\label{equation_inertia_at_l_description}
\Bar{\Omega}_\lambda|_{I_v} = \prod_{j} \theta^{\ell^{m(\theta,j)}\left\lfloor \frac{(r-j)d}{r}\right\rfloor_< }
\]
where $j$ runs over all embeddings $\psi_j: \Qr \to \Bar{\QQ}$ which send $v$ to $\lambda$.

Note that the choice of $\theta$ does not matter since we shall only use information about the image of $I_v$.
Indeed, taking a different fundamental character amounts to composing $\theta$ with an automorphism of $k_v$, that is, raising it to a $p$-th power.
As the image of $I_v$ has order prime to $p$, this is thus also an automorphism of the image of $I_v$.

In the following we shall say $J/\Qr$ has semistable reduction at the rational prime $\ell$, to signify $J$ has semistable reduction at every prime above $\ell$ in $\Qr$.

\begin{example}
\label{inertia_at_l_for_primes_of_res_deg_1_and_2}
Let $\ell \equiv 1 \mod{r}$ be a prime of semistable reduction for $J$. Then $\ell$ is totally split in $\Qr$. Let $\lambda,\lambda'$ be primes above $\ell$ in $\Qr$ and let $1\leq j \leq r-1$ be such that $\psi_j(\lambda') = \lambda$. Then, as the unique fundamental character of level one is the mod $\ell$ cyclotomic character $\chi_\ell$, we have
$$\Bar{\Omega}_\lambda|_{I_{\lambda'}} = \chi_\ell^{\left\lfloor \frac{(r-j)d}{r}\right\rfloor_<} .$$

Let $\ell \equiv -1 \mod{r}$ be a prime of good reduction. We have $[k_\lambda:\Fl]=2$ for a prime $\lambda$ above $\ell$ in $\Qr$. Let $\lambda'$ be a prime above $\ell$ in $\Qr$ and let $1 \leq j \leq r-1$ be such that $\psi_j(\lambda') = \lambda$. Then
$$\Bar{\Omega}_\lambda|_{I_{\lambda'}} = \theta^{\left\lfloor \frac{(r-j)d}{r}\right\rfloor_< + \ell\left\lfloor \frac{jd}{r}\right\rfloor_<} $$
where $\theta$ is some fundamental character of level 2.
\end{example}

\begin{example}
Let $r=7$, and $\ell \equiv 2 \mod{7}$ be a prime of semistable reduction for $J$. There are two primes $\lambda, \Bar{\lambda}$ above $\ell$ in $\QQ(\zeta_7)$. The decomposition group of each of these is given by $\{\psi_{1},\psi_{2},\psi_{4}\}$ and the map $\psi_{-1}$ sends $\lambda$ to $\Bar{\lambda}$. The above description then tells us
\begin{align*}
    \Bar{\Omega}_\lambda|_{I_\lambda} &= \theta^{\left\lfloor \frac{6d}{r}\right\rfloor_< + \ell \left\lfloor \frac{5d}{r}\right\rfloor_<  + \ell^2\left\lfloor \frac{3d}{r}\right\rfloor_<}\\
    \Bar{\Omega}_\lambda|_{I_{\Bar{\lambda}}} &= (\theta')^{\left\lfloor \frac{d}{r}\right\rfloor_< + \ell \left\lfloor \frac{2d}{r}\right\rfloor_<  + \ell^2 \left\lfloor \frac{4d}{r}\right\rfloor_<}
\end{align*}
where $\theta$ and $\theta'$ are fundamental characters of level $3$.

Similarly for $\ell \equiv 4 \mod{7}$, there are two primes $\lambda$, $\Bar{\lambda}$ above $\ell$ in $\QQ(\zeta_7)$ which are permuted by $\psi_{-1}$. Carrying out the above computation again, we find
\begin{align*}
    \Bar{\Omega}_\lambda|_{I_\lambda} &= \theta^{\left\lfloor \frac{6d}{r}\right\rfloor_< + \ell \left\lfloor \frac{3d}{r}\right\rfloor_<  + \ell^2 \left\lfloor \frac{5d}{r}\right\rfloor_< }\\
    \Bar{\Omega}_\lambda|_{I_{\Bar{\lambda}}} &= (\theta')^{\left\lfloor \frac{d}{r}\right\rfloor_< + \ell \left\lfloor \frac{4d}{r}\right\rfloor_<  + \ell^2 \left\lfloor \frac{2d}{r}\right\rfloor_< }
\end{align*}
for some fundamental characters $\theta, \theta'$ of level $3$.
\end{example}

In the following we write $m_j = \lfloor \frac{(r-j)d}{r} \rfloor_<$ and $n=\frac{2g}{r-1}$.
We recall our notation $\rho_\lambda(\GQrl)=\Glam$, which is viewed naturally as a subgroup of $\GL_n(\ell^i)$ when $i$ the inertia degree of $\ell$ in $\Qr$ is odd and as a subgroup of $\GU_n(\ell^{i/2})$ when $i$ is even. The kernel of $\det_{J[\lambda]}\colon\Glam \rightarrow \Fli^*$ is denoted by $\Slam$.
Finally, note $\Glam$ is the kernel of the cyclotomic character on $\rho_\lambda(\GQr)$. 

\begin{lemma}
\label{Lem_upper_bound_on_Glam}
Let $r =3$. Then
\[\det(\Glam) \leq \langle a^{m_1-m_2},b\rangle 
\] where $b$ is an element of order $6$, and if $i$ is odd, $a$ is a generator of $\Fl^*$, and if $i$ is even then $a$ is a generator of $(\FF_{\ell^2}^*)^{\ell-1}$ .
\end{lemma}

\begin{proof}
Let $D = \langle a^{m_1-m_2},b\rangle$. By the Chebotarev Density Theorem, it suffices to show the images of all but finitely many Frobenius elements of $\Gcyclo{3 \ell}$ belong to  $D$.
Let $\B$ be a prime in $\QQ(\zeta_{3\ell})$ such that $J/\QQ(\zeta_3)$ has semistable reduction at $\p \coloneqq \B \cap \ZZ[\zeta_3]$.
Suppose $p \neq \ell$ is the rational prime below $\B$, and let $f$ be its order modulo $\ell$. Then we have equality of Frobenius elements $\Frob_\B = \Frob_\p^f$.

Proposition \ref{prop_endo_car_inertia_at_p_ne_l} shows $\Omega_\lambda$ is unramified at $\p$, so we may use the factorisation given by the infinity type, see \S \ref{subsection_Alg_Heck_chars}. Furthermore, as $\ZZ[\zeta_3]$ has class number one,
\begin{align*}
    \Omega_\lambda(\Frob_\B) & = \Omega_\lambda(\Frob_\p)^f \\
     &= (up^{m_2}\pi^{m_1-m_2})^f \text{ where $u\in \ZZ[\zeta_3]^*$ and $\pi \in \ZZ[\zeta_3]$}\\
    & \equiv u^f\pi^{f(m_1-m_2)} \mod{\lambda}.
\end{align*}
As $\ZZ[\zeta_3]^*=\langle -\zeta_3 \rangle$ has order 6, the result now follows when $i$ is odd.
Let us suppose $i$ even.
To finish the proof, we must show $\pi^f \mod{\lambda}$ belongs to $(\FF_{\ell^2}^*)^{\ell-1}$.
Note that $(\FF_{\ell^2}^*)^{\ell-1}$ coincides with the kernel of the norm map $N_\ell^{\ell^2}:\FF_{\ell^2}^* \rightarrow \Fl^*$.
Since $N_\ell^{\ell^2}(\pi^f)=p^f \equiv 1 \mod{\lambda}$, the result follows.
\end{proof}

\begin{remark}
\label{remark_ceiling_of_g_over_3}
An easy calculation shows $m_1-m_2=\lceil \frac{g}{3} \rceil$.
\end{remark}

\begin{proposition}
\label{Prop_l_1_mod_3_full_Glam}
Set $r =3$. Let $\ell \equiv 1 \mod{3}$ be a prime of semistable reduction for $J/\QQ(\zeta_3)$. Then there is an element $\sigma \in \Glam$ such that $\det_{J[\lambda]}(\sigma) = a^{m_1-m_2}$ where $a \in \Fl^*$ is a generator.

Furthermore, if $f$ has $\p$-degree 2, where $p$ is a prime distinct from both $3$ and $\ell$ then $\det(\Glam)= \langle a^{m_1-m_2},b\rangle $
 where $b$ is an element of order $6$.

  In particular, if $\Slam \cong \SL_n(\ell)$ and $f$ has $\p$-degree 2, then 
 \[\Glam \cong  \{ \sigma \in \GL_n(\ell)| \det(\sigma) \in \langle a^{m_1-m_2},b\rangle \}.
\]
\end{proposition}

\begin{proof}
To prove the first statement, we look for an element of $\Gcyclo{3}$ whose image under $\chi_\ell$ is trivial and image by $\Bar{\Omega}_\lambda$ is a generator of $\Fl^*$ taken to the power of $m_1-m_2$.

By Example \ref{inertia_at_l_for_primes_of_res_deg_1_and_2} we have for $ \psi_j(\lambda') = \lambda$,
\[
\Bar{\Omega}_\lambda|_{I_{\lambda'}} = \chi_\ell^{m_j}.
\]
Let us denote by $I$ (resp. $I'$) an inertia group above $\psi^{-1}_j(\lambda)$ with $j= 1$ (resp.  $j= 2$). Since $\chi_\ell:I,I' \rightarrow \Fl^*$ is surjective, we may take $\tau \in I$, $\sigma \in I'$ such that $\chi_\ell(\tau)= \chi_\ell(\sigma)^{-1}$ generates $\Fl^*$. By choice of $\tau$ and $\sigma$ we have $\chi_\ell(\tau \sigma)=1$, so $\tau\sigma \in \Glam$. The above formula gives
\[
\Bar{\Omega}_\lambda(\tau \sigma) = \chi_\ell(\tau)^{m_1-m_2}.
\]
If $f$ has $\p$-degree 2, then by Example \ref{example_one_prime}, the generator of $\rho_\lambda(I_\p)$ has a unique eigenvalue of order $6$, with the others being equal to one. Taking the determinant of this element gives us $b$. 

Let us show the final statement. Let $G$ be the group we are trying to show $\Glam$ is isomorphic to. Corollary \ref{Lem_upper_bound_on_Glam} gives us $\Glam \leq G$. The above combined with the assumption $\Slam \cong \SL_n(\ell)$ shows $\Glam$ contains a group isomorphic to $G$, and thus completes the proof. 
\end{proof}

\begin{proposition}
\label{Prop_l_2_mod_3_full_Glam}
Set $r =3$. Let $\ell \equiv 2 \mod{3}$ be a prime of semistable reduction for $J/\QQ(\zeta_3)$. Then there is an element $\sigma \in \Glam$ such that $\det_{J[\lambda]}(\sigma) = a^{m_1-m_2}$ where $a \in (\Fl^*)^{\ell-1}$ is a generator.

Furthermore, if $f$ has $\p$-degree 2, where $p$ is a prime distinct from both $3$ and $\ell$, then $\det(\Glam)= \langle a^{m_1-m_2},b\rangle $
 where $b$ is an element of order $6$.
 
 In particular, if $\Slam \cong \SU_n(\ell)$ and $f$ has $\p$-degree 2, then 
 \[\Glam \cong  \{ \sigma \in \GU_n(\ell)| \det(\sigma) \in \langle a^{m_1-m_2},b\rangle \}.
\]
\end{proposition}

\begin{proof}
Recall that as the residual degree $i=2$ we have $\Glam \leq \GU_{n}(\ell)$.
Given $\det \colon \GU_{n}(\ell) \rightarrow (\FF_{\ell^2}^*)^{\ell-1}$ is surjective with kernel $\SU_n(\ell)$, we recognise $\det(\Glam)$ as being contained in $\Bar{\Omega}_\lambda(\GQr)\cap \FF_{\ell^2}^{\ell-1}$.
That is, the intersection of $\Bar{\Omega}_\lambda(\GQr)$ with the kernel of the norm map $N^{\ell^2}_\ell\colon \FF_{\ell^2}^* \rightarrow \Fl^*$.  

There is a unique prime $\lambda$ above $\ell$ in $\QQ(\zeta_3)$. Example \ref{inertia_at_l_for_primes_of_res_deg_1_and_2} shows
\[
\Bar{\Omega}_\lambda|_{I_{\lambda}} = \theta^{m_{2} + \ell m_1 }
\]
for an appropriate choice of a level 2 fundamental character $\theta$ (this choice does not concern us as they differ by an automorphism of $\FF_{\ell^2}$).

 As $\theta$ surjects onto $\FF_{\ell^2}^*$  and hence onto the kernel of the norm map, we may choose $\tau \in I_\lambda$ such that $\theta(\tau)$ generates $(\FF_{\ell^2}^*)^{\ell-1}$. Note such a $\tau$ must further be contained in $\GQrl$ since $\chi_\ell(\tau) = N^{\ell^2}_{\ell} \circ \theta(\tau) =1$.  
 The order of $\theta(\tau)$ is $\ell+1$ and thus $\theta(\tau)^{m_{2} + \ell m_1} = \theta(\tau)^{m_{2} - m_1} $ which establishes the first claim.
 
If $f$ has $\p$-degree 2, then by Example \ref{example_one_prime}, the generator of $I_\p$ has a unique eigenvalue of order $6$, the others being equal to one. Taking the determinant of this element gives us $b$. 

Let us show the final statement. Let $G$ be the group we are trying to show $\Glam$ is isomorphic to. Corollary \ref{Lem_upper_bound_on_Glam} gives us $\Glam \leq G$. The above combined with the assumption $\Slam \cong \SU_n(\ell)$ shows $\Glam$ contains a group isomorphic to $G$, and thus completes the proof.
\end{proof}

The above gives us information on $\Glam$, which is important for our study of the image of the Galois representation $\rho_\ell\colon \GQr \rightarrow \Aut(J[\ell])$, but if we are interested in the inverse Galois problem (for the classical groups) over $\Qr$ then we should study $\rho_\lambda(\GQr)$.

For $i$ even, there is not much to be said once $\Glam$ is determined, indeed $\rho_\lambda(\GQr)$ is a fixed extension of $\Glam$ not contained in $\GU_n(\ell^{i/2})$.
However, when $i$ is odd $\rho_\lambda(\GQr) $ is, like $\Glam$, contained in $\GL_n(\ell^i)$. 
This plays to our advantage and allows us to realise $\GL_n(\ell)$ when $i=1$. 

\begin{proposition}
\label{Prop_l_1_mod_r_IGP}
Suppose $2r|d$. Let  $\ell \equiv 1 \mod{r}$ be a prime of semistable reduction for $J/\Qr$.

 If $\rho_\lambda(\GQr)$ contains $\SL_n(\ell) $, then $\rho_\lambda(\GQr) = \GL_n(\ell)$.
\end{proposition}

\begin{proof}
It suffices to show $\Bar{\Omega}_\lambda\colon \GQr \rightarrow \Fl^*$ surjects. By Example \ref{inertia_at_l_for_primes_of_res_deg_1_and_2} we have for $ \psi_j(\lambda') = \lambda$
\[
\Bar{\Omega}_\lambda|_{I_{\lambda'}} = \chi_\ell^{m_j}.
\]
As the mod $\ell$ cyclotomic character $\chi_\ell\colon \GQr \rightarrow \Fl^*$ surjects, it suffices to find two values $j,j'$ such that $m_j$ and $m_{j'}$ are coprime.

Let us write $d=2ra$. We evaluate $m_j = 2a(r-j)-1$. In particular, for $j= \frac{r-1}{2}$ (resp. $j = \frac{r+1}{2}$) we have $m_j = a(r-1)-1 $ (resp. $m_j = a(r+1)-1 $).
Let us now compute the greatest common divisor of these two quantities:
\[\gcd(a(r-1)-1, a(r+1)-1) = \gcd(a(r-1)-1, 2a) = 1.\]
This proves $\det \circ \rho_\lambda = \Bar{\Omega}_\lambda: \GQr \rightarrow \Fl^*$ is surjective and thus completes the proof. 
\end{proof}

Despite the above, we can still realise $\DU_n(\ell)$ as a Galois extension of $\Qr$ for many values of $\ell \equiv -1 \mod{r}$.

\begin{proposition}
\label{prop_DU_image}
Suppose $2r|d$. Let $\ell \equiv -1 \mod{r}$, $\ell \neq 2$ be a prime of semistable reduction for $J/\Qr$.

Suppose there exists some $\delta|2r$ such that $\gcd(\frac{d}{r},\frac{\ell+1}{\delta})= \gcd(\delta, \frac{\ell+1}{\delta})=1$ and $f$ has $\p$-degree 2 for some prime of $\Qr$ with residue characteristic $p \neq r,\ell$.

If $\rho_\lambda(\GQr)$ contains $\SU_n(\ell)$, then $\rho_\lambda(\GQr) = \DU_n(\ell)$.
\end{proposition}

\begin{proof}
By Theorem \ref{thm_containment_of_the_image_of_GQr_and_GQrl} $\rho_\lambda(\GQr)$ is contained in a group isomorphic to $ \DU_n(\ell)$.
Thus it suffices to show $\det \circ \rho_\lambda = \Bar{\Omega}_\lambda : \GQr \rightarrow \FF_{\ell^2}^*$ is surjective. As $N^{\ell^2}_{\ell} \circ \Bar{\Omega}_\lambda$ coincides with the cyclotomic character, it suffices to show $\Bar{\Omega}_\lambda(\GQrl) = (\FF_{\ell^2}^*)^{\ell -1}$.

 Let $\delta|2r$ be a positive integer satisfying $\gcd(\delta, \frac{\ell+1}{\delta})=1$.
 As $f$ has $\p$-degree $2$, $\Bar{\Omega}_\lambda(I_\p) \leq (\FF^*_{\ell^2})^{\ell-1}$ has order $2r$.
It thus suffices to show $\Bar{\Omega}_\lambda(\GQrl)$ contains an element of order $\frac{\ell+1}{\delta}$.

Let $\lambda,\lambda'$ be primes above $\ell$ in $\Qr$ be such that $\psi_j(\lambda')=\lambda$. Example \ref{inertia_at_l_for_primes_of_res_deg_1_and_2} shows
\[
\Omega_\lambda|_{I_{\lambda'}} = \theta^{m_{j} + \ell m_{r-j} }
\]
for an appropriate choice of a level 2 fundamental character $\theta$ (this choice does not concern us as they differ by an automorphism of $\FF_{\ell^2}$).

As $\theta$ surjects onto $\FF_{\ell^2}^*$  and hence onto the kernel of the norm map, it suffices to compute  $\gcd(m_{j} + \ell m_{r-j} , \ell + 1)$ to determine the image of $I_{\lambda'}$ under $\Bar{\Omega}_\lambda$ which lands in $\Glam$.

Due to the congruence $m_{j} + \ell m_{r-j} \equiv m_{j} - m_{r-j}  \mod{ \ell + 1}$, we compute
\[m_{j} - m_{r-j} = \left\lfloor \frac{(r-j)d}{r}\right\rfloor_< - \left\lfloor \frac{jd}{r}\right\rfloor_< = \frac{d(r-2j)}{r}.
\]
This shows  $\frac{d}{r}$ divides $ m_{j} - m_{r-j}$ for all values of $j$. In fact, for $j = \frac{r-1}{2}$, we have $ m_{j} - m_{r-j} = \frac{d}{r}$, so this is the best contribution we will get from an inertia group at a prime above $\ell$.

The above shows $\Bar{\Omega}_\lambda(\GQrl)$ contains $(\FF^*_{\ell^2})^{\frac{d}{r}(\ell-1)}$.
Now let $\tau \in \FF_{\ell^2}^{\ell-1}$ be an element of order $\frac{\ell+1}{\delta}$.
By assumption $\gcd(\frac{d}{r},\frac{\ell+1}{\delta})=1$ and thus $\tau$ belongs to $(\FF^*_{\ell^2})^{\frac{d}{r}(\ell-1)}$, completing the proof.

\end{proof}

\end{subsection}

\end{section}

\begin{section}{Galois images}
\label{section_Galois_images}
In this section $J$ denotes the jacobian of a superelliptic curve determined by the affine model $y^r = f(x)$ where $f \in \Zr[x]$ is a squarefree monic polynomial of degree $d \geq 12$ divisible by $2r$.
Subsequently $n = \dim_{\Fli}J[\lambda] = \frac{2g}{r-1}=d-2$.

We recall our convention to say $J/\Qr$ is semistable at a rational prime $\ell$, if $J/\Qr$ is semistable at every prime above $\ell$ in $\Qr$.
Also recall that for a prime $\p_j$ of $\Qr$, we denote the rational prime below by $p_j$ (likewise for $\p$ and $p$).

The results proved thus far allow us to give a few different conditions for irreducibility and primitivity. In order to state Theorem \ref{thm_Glam_is_huge} concisely with these different conditions available, we label certain hypotheses below.

\begin{itemize}
    \item[(Irred I)]  There exist primes $q_1 < q_2 <q_3<d$ such that $q_1+q_2 = d$ and there are primes $\p_1,\p_2$ of $\Qr$, such that  $|k_{\p_1}|$ (resp. $|k_{\p_2}|$) is a primitive root modulo both $q_1,q_2$ (resp. modulo $q_3$). The set $S_{irr}=\{q_1,q_2,q_3,p_1,p_2\}$ has cardinality 5.
    \item[(Irred II)] The number $d -1 = q$ is prime and there is a prime $\p$ of $\Qr$, such that  $|k_{\p}|$ is a primitive root modulo $q$. Let $S_{irr}=\{q, p\}$.
    \item[(Prim I)] Either $3 \leq r \leq 23$ is prime or $r =31$. If $r=31$, assume GRH. If $r \in \{23,31\}$, assume there exists a prime $\frac{d}{2}<q_r<d$ congruent to $2$ modulo $3$, and let $\p_r$ be a prime of $\Qr$ such that $|k_{\p_r}|$ is a primitive root modulo $q_r$. If $r \in \{23,31\}$, let $S_{prim}=\{q_r,p_r\}$, else $S_{prim}=\emptyset$.
    \item[(Prim II)]  Assume $\Qr$ has \emph{odd} class number and there exist primes $q_1 < q_2 <d$ such that $q_1+q_2 = d$ and there is a prime $\p_1$ of $\Qr$, such that  $|k_{\p_1}|$ is a primitive root modulo both $q_1,q_2$. Let $S_{prim} = \{q_1,q_2,p_1\}$.
    \item[(A)] Hypotheses (Irred I), (Prim I) both hold and $S_{irr} \cap S_{prim} \neq \{p_r\}$. If $q_r \in S_{irr}\cap S_{prim} $, then $q_3= q_r$ and $p_2= p_r$. Moreover, $f$ has $\p_1$-degree $(q_1,q_2)$, $\p_2$-degree $q_3$ and (if $S_{irr}\cap S_{prim} = \emptyset$) $\p_r$-degree $q_r$.
    \item[(B)] Hypothesis (Irred II) holds and either (Prim I) or (Prim II) holds.
    If (Prim I) holds, then $f$ has $\p$-degree $q$ and if $r \in \{23,31\}$, then $f$ has $\p_r$-degree $q_r$ (if, and only if, $q=q_r$, one may take $\p=\p_r$). 
    If (Prim II) holds, then $f$ has both $\p_1$-degree $(q_1,q_2)$, $\p$-degree $q$ and $\p \neq \p_1$.
\end{itemize}
The technical looking conditions in (A) amount to saying our choices for the $p_i$'s should be distinct, though we may choose $p_r=p_2$ and $q_r=q_3$ if we wish to.
Note the existence of a prime $q_r$ satisfying the conditions of (Prim I) is guaranteed by Lemma \ref{lemma_prim_bertrand_postulate}.

\begin{theorem}
\label{thm_Glam_is_huge}
Suppose either (A) or (B) holds true.
Set $\pi = 1-\zr$.

Let $f(x) = x^d + a_{d-1}x^{d-1} + \ldots + a_1x+a_0 \in \Zr[x]$ satisfy  $a_0\equiv b\pi^{d-r} \mod{\pi^{r}}$ where $b \equiv 1 \mod{\pi^r}$, $a_{d-1} \equiv u\pi\mod{\pi^2}$ with $u \not \equiv 0 \mod{\pi}$, and $a_{j} \equiv 0\mod{\pi^{d-j}} $ for $1\leq j \leq d-2 $; and have
\begin{itemize}
    \item $\p_3$-degree $2$;
    \item $\p_4$-degree $r$
    with $p_3,p_4$ distinct and not in $ S_{irr} \cup S_{prim}$.
\end{itemize}
Furthermore, suppose that for any place $v$ belonging to $S_{bad}$, the set of places dividing the discriminant of $f$, one of the following holds
\begin{itemize}
    \item $v=\pi$;
    \item $f$ has prime $v$-degree of any height;
    \item $f$ has $v$-degree $(q_1',q_2')$ of any height where $q'_1+q_2' =d$ and both $q_1'$ and $q_2'$ are primes; or
    \item $J/\Qr$ has semistable reduction at $v$.
\end{itemize}

Then for any prime $\lambda$ of semistable reduction for $J/\Qr$ whose residue characteristic $\ell> \frac{n}{2}$ does not belong to $S_{bad} \cup S_{irr} \cup S_{prim}$, the group $\rho_\lambda(\GQr)$ contains $\SL_n(\ell^i)$ if $i$ is odd and $\SU_n(\ell^{i/2})$ if $i$ is even.
In particular, $\Glam$ is large. 
\end{theorem}

\begin{proof}

By Proposition \ref{prop_transvection_creation}, $\rho_\lambda(\GQr)$ contains a transvection.
As $r|d$, the dimension of the $\Fli$-vector space $J[\lambda]$ equals $d-2 \geq 10$.
Theorem \ref{thm_containment_of_the_image_of_GQr_and_GQrl} shows $\rho_\lambda(\GQr)$ is contained in $\GL(\ell^i)$ if $i$ is odd and $\DU_n(\ell^{i/2})$ if $i$ is even.
We may therefore apply Theorems \ref{thm_maximal_sbgps_GU} and \ref{thm_maximal_sbgps_GL}.
Remark \ref{remark_p_degree_and_strict_p-systems} will be used without comment in the following.

We first prove $\rho_\lambda(\GQr)$ is irreducible and not contained in a subfield subgroup.
If (A) is satisfied, then Proposition \ref{prop_irreducibility} shows $\rho_\lambda(\GQr)$ acts irreducibly on $J[\lambda]$, else this is achieved directly by Proposition \ref{irred_blocks_aux_Anni_VladDok}.
Lemma \ref{lemma_not_subfld_sbgp_not_in_GSp_n} implies $\rho_\lambda(\GQr)$ is not contained in a subfield subgroup.

Suppose $\rho_\lambda(\GQr)$ preserves a decomposition $V=\bigoplus_{j=1}^k V_j$.
Since $\rho_\lambda(\GQr)$ acts transitively on the $V_j$, it suffices to show the image of the induced homomorphism $\theta:\GQr \rightarrow S_k $ is trivial to prove primitivity.
Proposition \ref{prop_transvection_creation} and Lemmas \ref{lemma_good_reduction_at_r}, \ref{lemma_primitivity_p-systems_imply_primtive_when_no_unramified_extensions} imply $\theta$ is everywhere unramified.
If $3 \leq r \leq 19$, then $\Qr$ has no unramified extensions (see for example \cite[Appendix]{Yamamura_imaginary_abelian_NF_with_class_numbe_one}) and so we are done.
If $r=23$ or $31$, then noting $q_r -1 > \frac{n}{2}$, we apply Theorem \ref{thm_primitivity_23_31} to conclude.
Else (by assumption) both (Irred II) and (Prim II) hold. Proposition \ref{prop_primitivity_odd_class_number} then applies giving $k=1$.

We have now shown $\rho_\lambda(\GQr)$ contains $\SU_n(\ell^{i/2})$ if $i$ is even. Thus we suppose $i$ is odd in the following.
In this case we need to show $\rho_\lambda(\GQr)$ is not contained in a classical group, in particular it is sufficient to show it does not preserve a symmetric or alternating form.
The first cannot happen since $\rho_\lambda(\GQr)$ contains a transvection \cite[Prop. 5.7]{Grove} and the latter is ruled out by Lemma \ref{lemma_not_subfld_sbgp_not_in_GSp_n}. We conclude $\rho_\lambda(\GQr)$ contains $\SL_n(\ell^{i})$ when $i$ is odd.

The final statement follows from the fact that $\rho_\lambda(\GQr)/\Glam$ is cyclic and the groups $\SL_n(\ell^i)$ and $\SU_n(\ell^{i/2})$ are perfect.
\end{proof}

\begin{theorem}
\label{thm_mod_l_image_is_huge}
Suppose the conditions of Theorem \ref{thm_Glam_is_huge} are satisfied.
Then the image of  \[\rho_\ell:\GQr \rightarrow \Aut(J[\ell])\] is large provided  $J/\Qr$ is semistable at $\ell> \frac{n}{2}$ and $\ell$ is distinct from $r,q_1,q_2,q_3,p_1,p_2,p_3,p_4,p_r$.
\end{theorem}

\begin{proof}
We may view $H = \rho_\ell(\GQrl)$ as a subgroup of $C_{\Sp_{2g}(\ell)}(\zr)$.
Theorem \ref{thm_Glam_is_huge} ensures that the projection onto each factor of $C_{\Sp_{2g}(\ell)}(\zr)$ contains the commutator subgroup.
As $f$ has $\p_3$-degree 2, Example \ref{example_one_prime} allows us to apply Theorem \ref{thm_socle} to $H$, from which we deduce our representation has large image.
\end{proof}

\begin{theorem}
\label{thm_l_equiv_1_mod_r_image_of_mod_lambda}
Suppose the conditions of the Theorem \ref{thm_Glam_is_huge} hold, and in addition $\ell \equiv 1 \mod{r}$. Then \[\rho_\lambda(\GQr) = \GL_n(\ell)\]
provided $J/\Qr$ is semistable at $\ell>\frac{n}{2}$  and $\ell$ is distinct from $r,q_1,q_2,q_3,p_1,p_2,p_3,p_4,p_r$.
\end{theorem}

\begin{proof}
This is a direct consequence of Theorem \ref{thm_Glam_is_huge} and Proposition \ref{Prop_l_1_mod_r_IGP}.
\end{proof}

\begin{theorem}
\label{thm_DU_image_r=-1}
Suppose the conditions of Theorem \ref{thm_Glam_is_huge} hold and let $\ell \equiv -1 \mod{r}$ be a prime of semistable reduction for $J/\Qr$ greater than $\frac{n}{2}$ and distinct from $ r,q_1,q_2,q_3,p_1,p_2$, $p_3,p_4,p_r$.
Suppose there exists some $\delta|2r$ such that $\gcd(\frac{d}{r},\frac{\ell+1}{\delta})= \gcd(\delta, \frac{\ell+1}{\delta})=1$.
Then \[\rho_\lambda(\GQr) = \DU_n(\ell).\]
\end{theorem}

\begin{proof}
This is a direct consequence of Theorem \ref{thm_Glam_is_huge} and Proposition \ref{prop_DU_image}.
\end{proof}

In the below we let \[\GL_n(\ell)^{\left\lceil \frac{g}{3} \right\rceil, 6} = \{\sigma \in \GL_n(\ell)| \det(\sigma) \in \langle a^{\left\lceil \frac{g}{3} \right\rceil},b\rangle \}\] where $a$ generates $\Fl^*$ and $b \in \Fl^*$ has order $6$. We also write \[\GU_n(\ell)^{\left\lceil \frac{g}{3} \right\rceil, 6} =
\{ \sigma \in \GU_n(\ell)| \det(\sigma) \in \langle a^{\left\lceil \frac{g}{3} \right\rceil},b\rangle \}\] where $a$ generates $(\FF_{\ell^2}^*)^{\ell-1}$ and $b \in \FF_{\ell^2}^*$ has order $6$.

\begin{theorem}
\label{thm_r=3_exact_image_surjective}
Let $r=3$ and suppose the conditions of Theorem \ref{thm_Glam_is_huge} are satisfied.
Then for $\ell>\frac{n}{2}$ distinct from $q_1,q_2,q_3,p_1,p_2,p_3,p_4$ and semistable for $J/\QQ(\zeta_3)$, the image of  \[\rho_\ell:\Gcyclo{3} \rightarrow \Aut(J[\ell])\] is for $i$ odd:
\[\rho_\ell(\Gcyclo{3}) = \GL_n(\ell)^{\left\lceil \frac{g}{3} \right\rceil, 6} \rtimes \langle \chi_\ell\rangle\]
and for $i$ even:
\[\rho_\ell(\Gcyclo{3}) = \GU_n(\ell)^{\left\lceil \frac{g}{3} \right\rceil, 6} . \langle \chi_\ell\rangle,\]
where $\chi_\ell$ denotes the mod $\ell$ cyclotomic character.
In particular the image of $\rho_\ell$ is as large as possible.
\end{theorem}

\begin{proof}
This is a direct consequence of Theorem \ref{thm_Glam_is_huge}, Remark \ref{remark_ceiling_of_g_over_3} and Propositions \ref{Prop_l_1_mod_3_full_Glam} and \ref{Prop_l_2_mod_3_full_Glam}.
\end{proof}
\begin{subsection}{Examples}
\label{subsection_examples}
We now use the above theorems to construct superelliptic curves 
\[C\colon y^r = f(x) = x^d +a_{d-1}x^{d-1} + \ldots a_1x+a_0\] with $a_j \in \Qr$, for which the mod $\ell$ representation attached to the Jacobian of $C$ is large for all but a finite explicit set of primes $\ell$.

To do this we first find coefficients $a_j$ which ensure for each $\p_s$ as in Theorem \ref{thm_Glam_is_huge} that the Newton Polygons of $f \in \Qr_{\p_s}[x]$ will have the correct form.
We then factorise the discriminant of $f$ which allows us to verify the conditions stated in Theorem \ref{thm_Glam_is_huge} on $v \in S_{bad}$.
\begin{subsubsection}{$r=3$, $\deg(f) = 12$}
Theorem \ref{thm_Glam_is_huge} allows us to take $a_{11}= \pi$, $a_3=7\pi^9$, $a_2 = 14\pi^{10}$, $a_0 = 406\pi^9$ and every other $a_j=0$, up to a condition on the discriminant of $f$. Indeed, 7 is a primitive root modulo $11$ and $406 = 2 \times 7 \times 29 \equiv 1 \mod{3^2}$.

The discriminant of $f$ is a product of primes above $2,3,7,29$  and the primes $\zeta_3 - 5, 45\zeta_3 +
17, 8139777131\zeta_3 + 30568866704, 160690522581570205\zeta_3 -
330535909372465022$. The residue characteristics of these four last primes are $31, 1549, 751887821191463868553$, and \\ $188189419441256467739625500157072019$ respectively.

It follows that for $\ell>7$ not equal to $11, 29$ or any of the previously mentioned primes that the image of the mod $\ell$ representation attached to the jacobian of
\[y^3 = x^{12} + \pi x^{11} + 7\pi^{9} x^3 + 14\pi ^{10} x^2 + 406\pi^9
\]
is equal to 
\[\rho_\ell(\Gcyclo{3}) = \GL_{10}(\ell)^{2,6} \rtimes \langle \chi_\ell \rangle \text{ for $\ell \equiv 1 \mod{3}$, and}
\]
\[\rho_\ell(\Gcyclo{3}) = \GU_{10}(\ell)^{2,6} . \langle \chi_\ell \rangle \text{ for $\ell \equiv 2 \mod{3}$.}
\]
Let $\lambda|\ell$ with $\ell$ as above. If $\ell \equiv 1 \mod{3}$ we have \[\rho_\lambda(\Gcyclo{3})=\GL_{10}(\ell)\]
and if $\ell \equiv 5 \mod{12}$, then
\[\rho_\lambda(\Gcyclo{3})=\DU_{10}(\ell).\]
\end{subsubsection}
\begin{subsubsection}{$r=3$, $\deg(f) = 18$}
The prime $7$ is also a primitive root modulo $17$, and thus Theorem \ref{thm_Glam_is_huge} allows us to take $a_{17}= \pi$, $a_3=7\pi^{15}$, $a_2 = 14\pi^{16}$, $a_0 = 406\pi^{15}$ and every other $a_j=0$, up to a condition on the discriminant of $f$.

The discriminant of $f$ is a product of primes above $2,3,7,29$  and the primes $\zeta_3-3, 13\zeta_3 +
15, 79\zeta_3 + 72, 335\zeta_3 + 444, 1888757\zeta_3 + 3108290, 6313875129\zeta_3 - 22078748747, 22017526552863\zeta_3 - 3454026061453$, and $193191848791723\zeta_3 + 116736896365287$.
The residue characteristics of the last eight primes are $13, 199, 5737, 160621$, $ 7358065233619, 666738627970882050013$, \\ $572750882061546018557057917$ and $28397976581546156385381781597$ respectively.

It follows that for $\ell>7$ not equal to $17,29$ or any of the other previously mentioned primes that the image of the mod $\ell$ representation attached to the jacobian of
\[y^3 = x^{18} + \pi x^{17} + 7\pi^{15} x^3 + 14\pi ^{16} x^2 + 406\pi^{15}
\]
is equal to 
\[\rho_\ell(\Gcyclo{3}) = \GL_{16}(\ell)^{6,6} \rtimes \langle \chi_\ell \rangle \text{ for $\ell \equiv 1 \mod{3}$, and}
\]
\[\rho_\ell(\Gcyclo{3}) = \GU_{16}(\ell)^{6,6} . \langle \chi_\ell \rangle \text{ for $\ell \equiv 2 \mod{3}$.}
\]
Let $\lambda|\ell$ with $\ell$ as above. If $\ell \equiv 1 \mod{3}$ we have \[\rho_\lambda(\Gcyclo{3})=\GL_{16}(\ell)\]
and if $\ell \equiv 5,29 \mod{36}$, then
\[\rho_\lambda(\Gcyclo{3})=\DU_{16}(\ell).\]
\end{subsubsection}
\begin{subsubsection}{$r=3$, $\deg(f) = 24$}
The primes dividing the discriminant of $f(x) = x^{24} + \pi x^{23} + 7\pi^{21} x^3 + 14\pi ^{22} x^2 + 406\pi^{21}$ are those above $2,3,7,29$ and $7\zeta_3 + 3, 7\zeta_3 -
3, 11\zeta_3 + 3, 233\zeta_3 + 291, 3454821\zeta_3 + 5114984$, and \\ $990700272353375069264170600482740996166905135076413552894471\zeta_3 + \\
676177052735320299803168203356524886996207359914421817393363$.
Moreover the last six primes divide the discriminant exactly once and have residue characteristics $37, 79, 97, 71167,$ $20427495324433$ and a 120 digit prime number respectively.

Thus for $\ell>11$ not equal to $23, 29$ or any of the other previously mentioned primes, the image of the mod $\ell$ representation attached to the jacobian of
\[y^3 = x^{24} + \pi x^{23} + 7\pi^{21} x^3 + 14\pi ^{22} x^2 + 406\pi^{21}
\]
is equal to 
\[\rho_\ell(\Gcyclo{3}) = \GL_{22}(\ell)^{8,6} \rtimes \langle \chi_\ell \rangle \text{ for $\ell \equiv 1 \mod{3}$, and}
\]
\[\rho_\ell(\Gcyclo{3}) = \GU_{22}(\ell)^{8,6} . \langle \chi_\ell \rangle \text{ for $\ell \equiv 2 \mod{3}$.}
\]
Let $\lambda|\ell$ with $\ell$ as above. If $\ell \equiv 1 \mod{3}$ we have \[\rho_\lambda(\Gcyclo{3})=\GL_{22}(\ell)\]
and if $\ell \equiv 5 \mod{12}$, then
\[\rho_\lambda(\Gcyclo{3})=\DU_{22}(\ell).\]
\end{subsubsection}
\begin{subsubsection}{$r=3$, $\deg(f) = 30$}
The primes dividing the discriminant of $f(x) = x^{30} + \pi x^{29} + 19\pi^{27} x^3 + 38\pi ^{28} x^2 + 190\pi^{27}$ are those above $2,3,5,19$ and $7\zeta_3 + 3,
21\zeta_3 + 8, 23\zeta_3 + 57, 91969915\zeta_3 + 21624639, 1200258023\zeta_3 +
748167174, 3886224025627\zeta_3 - 3573324917796$, and $70201873885476577416120986535507248428567\zeta_3\\  -
98523085051934612037267607266794508313388$.
Moreover the last seven primes divide the discriminant exactly once and have residue characteristics $37, 337, 2467, 6937274066251861,$ $1102379788888277803,$ $41758129292412755682598837$ and an 83 digit prime number respectively.

Thus for $\ell>13$ not equal to $19,29$ or any of the other previously mentioned primes, the image of the mod $\ell$ representation attached to the jacobian of
\[y^3 = x^{30} + \pi x^{29} + 19\pi^{27} x^3 + 38\pi ^{28} x^2 + 190\pi^{27}
\]
is equal to 
\[\rho_\ell(\Gcyclo{3}) = \GL_{28}(\ell)^{10,6} \rtimes \langle \chi_\ell \rangle \text{ for $\ell \equiv 1 \mod{3}$, and}
\]
\[\rho_\ell(\Gcyclo{3}) = \GU_{28}(\ell)^{10,6} . \langle \chi_\ell \rangle \text{ for $\ell \equiv 2 \mod{3}$.}
\]
Let $\lambda|\ell$ with $\ell$ as above. If $\ell \equiv 1 \mod{3}$ we have \[\rho_\lambda(\Gcyclo{3})=\GL_{28}(\ell)\]
and if $\ell$ satisfies both $\ell \equiv 5 \mod{12}$ and $\ell \not \equiv -1 \mod{5} $, then
\[\rho_\lambda(\Gcyclo{3})=\DU_{28}(\ell).\]
\end{subsubsection}

\begin{subsubsection}{$r=5$, $\deg(f) = 20$}
Let $f(x) = x^{20} + 3\pi x^{19} + 41\pi^{15}x^5 + 82\pi^{18}x^2 + 3526\pi^{15}$. The norm of the discriminant of $f$ is the product of powers of $2,5,41,43$ and a 265 digit prime. It follows from the theorems in the previous section that the jacobian attached to the superelliptic curve
\[ y^5 = x^{20} + 3\pi x^{19} + 41\pi^{15}x^5 + 82\pi^{18}x^2 + 3526\pi^{15}
\]
has large mod $\ell$ image for $\ell > 7$ and $\ell \neq 19,41,43$ or the 265 digit prime mentionned above.

In particular, if $\lambda|\ell$ with $\ell \equiv 1 \mod{5}$, we have \[\rho_\lambda(\Gcyclo{5})=\GL_{18}(\ell)\]
and for $\ell \equiv 9 \mod{20}$
\[\rho_\lambda(\Gcyclo{5})=\DU_{18}(\ell).\]
\end{subsubsection}
\begin{subsubsection}{$r=7$, $\deg(f) = 14$}
Let $f(x) = x^{14}+ \pi x^{13} + 2\pi^7 x^7 + 6\pi^{12}x^2 + 246\pi^7$. The norm of the discriminant of $f$ is the product of powers of $2,3,7,41$ and the primes $701, 11039501386253916593179$ along with a 211 digit prime. In particular, the discriminant of $f$ is squarefree away from $2,3,7,41$.

It follows from the theorems in the previous section that the jacobian attached to the superelliptic curve
\[ y^7 = x^{14} + \pi x^{13} + 2\pi^7 x^7 + 6\pi^{12} x^2 + 246\pi^7
\]
has large mod $\ell$ image for $\ell > 7$ and $\ell \neq 13, 41, 701, 11039501386253916593179$ or the 211 digit prime dividing the norm of the discriminant of $f$.

In particular, if $\lambda|\ell$ with $\ell \equiv 1 \mod{7}$, we have \[\rho_\lambda(\Gcyclo{7})=\GL_{12}(\ell)\]
and for $\ell \equiv 13 \mod{28}$
\[\rho_\lambda(\Gcyclo{7})=\DU_{12}(\ell).\]
\end{subsubsection}
\end{subsection}

\end{section}

\bibliographystyle{alpha}
\bibliography{bib}

\end{document}